\documentclass[10pt]{amsart}
\usepackage{amsmath}
\usepackage{amssymb}
\usepackage{amsfonts}
\usepackage{amsaddr}

\usepackage{mathrsfs}
\usepackage{graphicx}
\usepackage{epsfig}
\usepackage[T1]{fontenc}
\numberwithin{equation}{section}       

\theoremstyle{plain}
\newtheorem{theo}{Theorem}
\newtheorem{prop}{Proposition}[section]

\newtheorem{lemm}[prop]{Lemma}

\newtheorem*{mainthm}{Main Theorem}
\newtheorem*{thm2}{Theorem 2'}

\theoremstyle{definition}
\newtheorem{defi}[prop]{Definition}

\theoremstyle{remark}

\newtheoremstyle{citing}
  {3pt}
  {3pt}
  {\itshape}
  {}
  {\bfseries}
  {.}
  {.5em}
  {\thmnote{#3}}

\theoremstyle{citing}

\DeclareMathAlphabet{\mathpzc}{OT1}{pzc}{m}{it} 

\newcommand{\N}{\mathbb{N}}

\newcommand{\Q}{\mathbb{Q}}
\newcommand{\R}{\mathbb{R}}

\newcommand{\Z}{\mathbb{Z}}
\newcommand{\bfF}{\textbf{F}}

\newcommand{\cC}{\mathcal{C}}

\newcommand{\cH}{\mathcal{H}}
\newcommand{\cI}{\mathcal{I}}

\newcommand{\cK}{\mathcal{K}}
\newcommand{\cL}{\mathcal{L}}

\newcommand{\cN}{\mathcal{N}}
\newcommand{\cO}{\mathcal{O}}
\newcommand{\cP}{\mathcal{P}}

\newcommand{\cR}{\mathcal{R}}

\newcommand{\cT}{\mathcal{T}}
\newcommand{\cU}{\mathcal{U}}

\newcommand{\cX}{\mathcal{X}}

\newcommand{\sA}{\mathscr{A}}

\newcommand{\sD}{\mathscr{D}}
\newcommand{\sE}{\mathscr{E}}
\newcommand{\sF}{\mathscr{F}}
\newcommand{\sG}{\mathscr{G}}

\newcommand{\sL}{\mathscr{L}}
\newcommand{\sM}{\mathscr{M}}

\newcommand{\sP}{\mathscr{P}}

\newcommand{\hE}{\widehat{E}}

\newcommand{\hJ}{\widehat{J}}

\newcommand{\hS}{\widehat{S}}

\newcommand{\hh}{\widehat{h}}

\newcommand{\hdelta}{\widehat{\delta}}

\newcommand{\hLambda}{\widehat{\Lambda}}

\newcommand{\tA}{\widetilde{A}}
\newcommand{\tB}{\widetilde{B}}

\newcommand{\tJ}{\widetilde{J}}
\newcommand{\tK}{\widetilde{K}}

\newcommand{\teta}{\widetilde{\teta}}

\newcommand{\eps}{\varepsilon}
\DeclareMathOperator{\dom}{dom}

\newcommand{\Bad}{\textrm{Bad}}
\newcommand{\Leb}{\textrm{Leb}}
\newcommand{\dist}{d}

\DeclareMathOperator{\Dist}{Dist}

\DeclareMathOperator{\CV}{\operatorname{CV}}

\newcommand{\bff}{\textbf{f}}
\newcommand{\bfg}{\textbf{g}}
\newcommand{\bfh}{\textbf{h}}
\newcommand{\Crit}{\mathcal{C}}
\newcommand{\SC}{\mathcal{S}}
\newcommand{\LD}{\mathcal{LD}}

\begin{document}

\title[Stochastic stability]{On stochastic stability of non-uniformly expanding interval maps}
\author[W. Shen]{Weixiao Shen}

\address{Department of Mathematics, National University of Singapore, Block S17,
10, Lower Kent Ridge Road,
Singapore 119076 (Email: matsw@nus.edu.sg)
}
\thanks{The work was partially supported by AcRF Tier 1 Grant No. R-146-000-128-133.}
\date{\today}



%
%
\begin{abstract}
We study the expanding properties of random perturbations of regular interval maps satisfying the summability condition of exponent one. Under very general conditions on the interval maps and perturbation types, we prove strong stochastic stability.
\end{abstract}

\maketitle
\section{Introduction}
Non-uniformly expanding interval maps play an important role in the theory of dynamical systems.
The statistical properties of these systems as well as the persistency of these properties have attracted much interest.
In particular, the celebrated work of Jakobson~\cite{J} showed that non-uniformly expanding maps are abundant among interval maps.  Extensions and generalizations of this work have produced many of the main examples of non-uniformly hyperbolic dynamical systems in dimension one or higher, see~\cite{BC, BC2, V, WY, R} among others.

In this paper, we study random perturbations of non-uniformly expanding interval maps $f$, modeled by iterates of random maps. So we shall study composition of maps of the form
$g_{n-1}\circ \cdots \circ g_1 \circ g_0$, where $g_0, g_1, \ldots$ are independently chosen random maps from a small neighborhood of $f$ in a suitable space of interval maps. Under very general conditions on the dynamics $f$ and on perturbation types, we shall prove stochastic stability: a typical random orbit $g_{n-1}\circ \cdots \circ g_1 \circ g_0(x)$ has roughly the same asymptotic distribution in the phase space as a typical orbit of $f$, in a strong sense.

To be more precise, let $f: [0,1]\to [0,1]$ be a multimodal interval map of class $C^3$ with non-flat critical points. We shall assume that $f$ has no attracting or neutral cycles. In the deterministic case, existence of an ergodic invariant probability measure which is absolutely continuous with respect to the Lebesgue measure (acip) has been extensively studied. Such a measure is a {\em physical measure} in the sense that there exists a subset $E$ of $[0,1]$ with positive Lebesgue measure such that for a.e. $x\in E$, $$\frac{1}{n}\sum_{i=0}^{n-1} \delta_{f^i(x)}\to \mu\mbox{ as }n\to\infty$$ in the weak star topology.
In the recent work~\cite{BSS, BRSS}, existence of acip was proved under the following {\em large derivatives condition}: for each critical value $v$, we have
\begin{equation}\label{eqn:LD}
|Df^n(v)|\to\infty \mbox{ as } n\to\infty,
\end{equation}
which generalizes earlier results in~\cite{CE,NS, BLS} among others, where stronger growth conditions on $|Df^n(v)|$ were assumed.
In general, an interval map $f$ may have more than one acip's. However, by~\cite{BL, SV}, if $\omega(c)\cap \omega(c')\not=\emptyset$ for any critical points $c, c'$, then $f$ is ergodic with respect to the Lebesgue measure, so $f$ has a unique acip (under the condition (\ref{eqn:LD})). In particular, it is the case if $f$ is unimodal.

To consider random perturbations, we shall need a stronger assumption on $f$: {\em the summability condition of exponent $1$}, which means that 
for each critical value $v$, we have
\begin{equation}\label{eqn:scbeta}
\sum_{n=0}^\infty \frac{1}{|Df^n(v)|}<\infty.
\end{equation}
We shall define a space $\Omega\ni f$ of interval maps which we call {\em admissible} and from which all random maps will be taken. The precise definition can be found in \S\ref{subsec:space}. At this moment, let us mention that when all critical points of $f$ are of quadratic type and are contained in the open interval $(0,1)$, we may take $\Omega$ to be a small neighborhood of $f$ in the $C^2$ topology.
For $\eps>0$, let $\Omega_\eps$ denote the $\eps$-neighborhood of $f$ in $\Omega$ with respect to {\em the $C^1$ metric}.
A sequence $\{x_n\}_{n=0}^\infty$ is called an {\em $\eps$-random orbit} if for each $n\ge 0$, $x_{n+1}=g_n(x_n)$ for some $g_n\in\Omega_\eps$.
Given a Borel probability measure $\nu_\eps$ supported in $\Omega_\eps$, the measure $\Leb\times \nu_\eps^\N$ on $[0,1]\times \Omega_\eps^\N$
naturally induces a probability measure on the space of $\eps$-random orbits which is our reference measure on the space of $\eps$-random orbits.
A Borel probability measure $\mu_\eps$ is called {\em physical for $\eps$-perturbations} if the set of $\eps$-random orbits $\{x_n\}_{n=0}^\infty$ with the following property has positive measure:
\begin{equation}\label{eqn:physicalrandom}
\frac{1}{n}\sum_{i=0}^n \delta_{x_i}\to \mu_\eps \mbox{ as } n\to\infty \mbox{ in the weak star topology}.
\end{equation}

To obtain meaningful results, we shall need to assume certain regularity of $\nu_\eps$. Given $\nu_\eps$, we define a family $\sP_\eps=\{p_\eps(\cdot|x)\}_{x\in [0,1]}$ of probability measures in $[0,1]$ as follows:
\begin{equation}\label{eqn:tranprob}
p_{\eps}(E| x)=\nu_\eps \left(\left\{g\in\Omega: g(x)\in E\right\}\right).
\end{equation}
Then each $p_\eps(\cdot|x)$ is supported in the $\eps$-neighborhood of $f(x)$. We write $\nu_\eps\in \sM_\eps(L)$ if
for each $x\in [0,1]$, and each Borel set $E\subset [0,1]$, we have
\begin{equation}\label{eqn:regulartranpro}
p_\eps(E|x)\le L \left(\frac{|E|}{2\eps}\right)^{1/L},
\end{equation}
where $|E|$ denote the Lebesgue measure of $E$.

\begin{mainthm}
Let $f:[0,1]\to [0,1]$ be a map of class $C^3$ with non-flat critical points and without attracting or neutral cycle and let $\Omega\ni f$ be an admissible space of interval maps.
For each $\eps>0$ small, let $\nu_\eps$ be a Borel probability measure on $\Omega_\eps$. Assume that
\begin{itemize}
\item $f$ satisfies the summability condition of exponent $1$.
\item $f$ is ergodic with respect to the Lebesgue measure.
\item there exists $L>1$ such that
$\nu_\eps\in \sM_\eps(L)$ for all $\eps>0$ small.
\end{itemize}
Then there exists $\eps_0>0$ such that the following hold:
\begin{enumerate}
\item [(i)] For each $\eps\in (0,\eps_0]$, there exists a unique physical measure $\mu_\eps$ for $\eps$-perturbations. Moreover, $\mu_\eps$ is absolutely continuous with respect to the Lebesgue measure and (\ref{eqn:physicalrandom}) holds for a.e. random orbits. 
\item [(ii)] $f$ is strongly stochastically stable in the following sense: as $\eps\to 0$, the density functions $d\mu_\eps/d\Leb$ converge in the $L^1$ topology to the density function of the unique acip $\mu$ of $f$.
\end{enumerate}
\end{mainthm}

Indeed, for each $\eps>0$ small, $\mu_\eps$ is the unique stationary measure for homogenous Markov chains $\chi^\eps$ with $p_\eps(\cdot|x)$ as transition probabilities.
Recall that a probability measure $\mu_\eps$ on $[0,1]$ is called a {\em stationary measure} for $\chi^\eps$ (or for $\sP_\eps$, or for $\nu_\eps$) if for
each Borel set $A\subset [0,1]$, we have
\begin{equation}\label{eqn:stationarym}
\mu_\eps(A)= \int_0^1 p_\eps(A|x) d\mu_\eps(x)=\int_{\Omega} \mu_\eps (g^{-1}(A)) d\nu_\eps(g).
\end{equation}

Stochastic stability of dynamical systems was introduced by Kolmogrov and Sinai. It is natural in consideration that any system arising from real world is unavoidably affected by external noises. In the long term project proposed by Palis~\cite{Palis}, this notion is used to replace structural stability introduced by Andronov and Pontrjagin, as a property which may be satisfied by ``most'' dynamical systems. It is now well-known that structural stability is tightly linked to uniform hyperbolicity, hence too restrictive. For instance, by~\cite{KSS}, the map $f$ in the Main Theorem is not structurally stable in any $C^r$ topology. On the other hand, it had been shown in~\cite{Ly, AM} that almost every quadratic polynomial $x\mapsto ax(1-x)$ satisfies either Axiom A, or the Collet-Eckmann condition and hence is strongly stochastically stable with respect to the random perturbations considered here. 

An extensive historical account on stochastic stability of dynamical systems can be found in~\cite{LK} or~\cite{BDV}.
For non-uniformly expanding interval maps, stochastic stability was previously studied in \cite{KK,T1,BeY,BaV}.
All these works considered interval maps which satisfy a condition of the Benedicks-Carleson type (or an even stronger condition). So our assumption on $f$ is significantly weaker. The perturbation types allowed here are also more general than those in~\cite{T1, BeY, BaV}. In particular, we allow the density functions of the transition probabilities to have singularity. However, in \cite{KK}, random perturbations modeled by general Markov chains were considered (for logistic maps of the  Misiurewicz type). In general it is unclear what Markov chain perturbations can be realized by iterates of smooth random maps, although some results in this direction were obtained in ~\cite[Section 1.1]{Kifer}, \cite{Q}, \cite{BeV} and \cite{JKR}. It would be interesting to find a formulation of perturbation types which covers both of the (independent) settings in~\cite{KK} and here, even in the logistic Misiurewicz case.

Strong stochastic stability for multimodal Collet-Eckmann maps was stated as Conjecture 1 in~\cite[Section 1]{BLS}, where the authors also suggested that ``possibly the Collet-Eckmann condition itself can be replaced by a much weaker growth condition''. Stochastic stability for interval maps satisfying a  summability condition stronger than (\ref{eqn:scbeta}) was posed as a problem in~\cite[Problem E.5]{BDV}.

The Main Theorem will be proved using an inducing scheme. Our Theorem~\ref{theo:reduced} asserts that for each $\eps>0$ small, under the regularity assumption on the measure $\nu_\eps$, almost every random orbit has a large scale time which is integrable in $\Leb\times \nu_\eps^\N$ and uniformly in $\eps$. The idea of inducing is well-known and powerful in the research on deterministic non-uniformly hyperbolic dynamics. Implementation of the idea in the random setting appeared in~\cite{AA,AV,BBM}. However, it seems more difficult to verify the assumptions adopted in ~\cite{AA,AV} than construct an inducing scheme directly, at least in the set-up of this paper. (In \cite{BBM}, stochastic stability was not discussed.)

To obtain Theorem~\ref{theo:reduced}, we shall first obtain lower bounds on growth of derivatives along random orbits, see Theorem~\ref{theo:dergwrothrandom}. This is based on a combination of analysis on expansion of the deterministic dynamics and the binding argument initiated by Jakobson~\cite{J} and Benedicks-Carleson~\cite{BC}. While the former argument requires only the large derivatives condition, we need the stronger condition (\ref{eqn:scbeta}) for the latter. Based on a result of \cite{BRSS}, the backward contraction property for an interval map $f$ with the large derivatives condition, we shall show that the total distortion of the first landing maps to suitably chosen critical neighborhoods of $f$ is small, see Lemma~\ref{lem:derVSdiscrt}. This result plays a crucial role in making a delicate choice of the {\em preferred binding period}, see Propositions~\ref{prop:shadowablegrowth} and ~\ref{prop:psorbretpre}. As a consequence of Theorem~\ref{theo:dergwrothrandom}, we shall prove the first landing map of $\eps$-random orbits into a suitably chosen critical neighborhood $\tB(\eps)$ of $f$ has small total distortion, see Proposition~\ref{prop:psorblocdis}. Theorem~\ref{theo:dergwrothrandom} also provides control of the recurrence of most $\eps$-random orbits into $\tB(\eps)$ under the regularity assumption on $\nu_\eps$, see Proposition~\ref{prop:rettbeps}.

As an application of the method and results presented here, a generalization of the Jakobson's theorem is given in~\cite{GS}.

The paper is organized as follows. In \S\ref{sec:prel}, we give precise definitions of the space of interval maps considered here and state Theorems~\ref{theo:dergwrothrandom} and~\ref{theo:reduced}. This section also contains lemmas about distortion and shadowing. The proof of the Main Theorem is given in \S\ref{sec:pfmain} assuming the two theorems just mentioned. Theorem~\ref{theo:dergwrothrandom} is proved in \S\ref{sec:pseudoorb}, after some preparatory study on the deterministic dynamics done in \S\ref{sec:derterministic}. The last three sections, \S\S\ref{sec:thm2}-\ref{sec:inducing}, are devoted to the proof of Theorem~\ref{theo:reduced}. In particular, in \S\ref{sec:thm2}, an outline of the proof of this theorem is given. In \S\ref{sec:slowrec}, we study the recurrence of $\eps$-random orbits to a suitably chosen small neighborhood of critical points, and estimate diffeomorphic return times to various critical regions. The final inducing step is carried out in \S\ref{sec:inducing}.

\section{Preliminaries and Statement of results}\label{sec:prel}
\subsection{Space of interval maps}\label{subsec:space}
For each $k=0,1,\ldots$, we use $\sF_k$ to denote the space of all $C^k$ maps from $[0,1]$ into itself, endowed with the $C^k$ metric. For $g\in\sF_1$, let $\Crit(g)$ denote the set of critical points of $g$ and let $\CV(g)=g(\Crit(g))$.

For $k=1,2,\ldots,$ let $\sA_k$ be the collection of maps $g\in\sF_1$ which have only hyperbolic repelling periodic points and which are of \emph{of class~$C^k$ with non-flat critical points}. The latter means that the following properties hold for $g$:
$g$ is of class~$C^k$ outside $\Crit(g)$;
and for each $c \in \Crit(g)$, there exists a number $\ell_c>1$ (called \emph{the order of~$g$ at~$c$})
and diffeomorphisms $\phi, \psi$ of~$\R$ of class~$C^k$ with $\phi(c)=\psi(g(c))=0$ such that,
$|\psi\circ g(x)|=|\phi(x)|^{\ell_c}$
holds on a neighborhood of~$c$ in~ $[0,1]$.
Note that each $g\in \sA_1$ has at least one critical point and let $\ell_{\max}(g)$ denote the maximum order of the critical points of $g$.

A subspace $\Omega$ of $\sF_1$ is called {\em admissible} if the following two conditions are satisfied:
\begin{enumerate}
\item[(i)] there exists a constant $C>0$ such that for any $x, y\in [0,1]$ and $g\in \Omega$, we have
\begin{equation}\label{eqn:constantK0}
2\dist(x, y) <\dist (x,\Crit(g)) \Rightarrow \left|\log \frac{|Dg(x)|}{|Dg(y)|}\right|\le C\frac{\dist(x,y)}{\dist (x,\Crit(g))}.
\end{equation}
\item[(ii)] there exist an integer $n\ge 0$, real numbers $\ell_1, \ell_2, \ldots, \ell_n>1$ and $\delta>0$, $O_1>0, O_2>0$ such that the following hold:
\begin{itemize}
\item A map $g\in \Omega$ has exactly $n$ critical points, denoted by $c_1(g)<c_2(g)<\cdots <c_n(g)$;
\item When $\dist(x,c_i(g))<\delta$, we have
$$O_1\dist(x,c_i(g))^{\ell_i-1}\le |Dg(x)|\le O_2 \dist(x,c_i(g))^{\ell_i-1}.$$
\end{itemize}
\end{enumerate}
Clearly, for $f\in\sF_2$ with non-degenerate critical points all of which lie in $(0,1)$, a small neighborhood of $f$ in $\sF_2$ is admissible.

Let $\LD$ denote the collection of all maps $f\in\sA_3$ which satisfies the large derivatives condition (\ref{eqn:LD}), and let $\SC_1$ denote the collection of all maps $f\in\sA_3$ which satisfies the summability condition of exponent $1$, i.e., (\ref{eqn:scbeta}). Clearly, $\SC_1\subset \LD$ and a map $f\in \LD$ has no critical relation: for any $c\in\Crit(f)$ and any integer $n\ge 1$, $f^n(c)\not\in\Crit (f)$. So if $f$ has a critical point $c\in\{0,1\}$ then it is not contained in the image of $f$ and hence dynamically irrelevant.
\subsection{Notations}

Unless otherwise stated, in the following, $f\in\sA_3$ and $\Omega\ni f$ is a subspace of $\sF_1$.
We shall always endow $\Omega$ with the $C^1$ metric and denote by $\Omega_\eps$ the $\eps$-neighborhood of $f$ in $\Omega$.

For $\bfg=(g_0,g_1,\ldots)\in\Omega^\N$ and $n\ge 1$, let
$$\bfg^n=g_{n-1}\circ g_{n-2}\circ \cdots \circ g_1\circ g_0.$$
Let $\bfg^0(x)=x$. We shall often consider the skew product $F: [0,1]\times \Omega^\N \to [0,1]\times \Omega^\N$, defined as
\begin{equation}\label{eqn:skewF}
F(x,\bfg)=(g_0(x), \sigma(\bfg)),
\end{equation}
where $\sigma:\Omega^\N\to\Omega^\N$ is the shift map.

For a subset $X$ of $[0,1]\times \Omega^\N$, let $X^\bfg$ denote the fiber of $X$ over $\bfg$, i.e.
$X^\bfg=\{x\in [0,1]| (x,\bfg)\in X\}.$

Given a $C^1$ diffeomorphism $\varphi: J\to I$ between bounded intervals, define
$$\Dist(\varphi|J)=\sup_{x, y\in J} \log \frac{|\varphi'(x)|}{|\varphi'(y)|},$$
and define $$\cN(\varphi|J)=\sup_{J'} \Dist (\varphi|J')\frac{|J|}{|J'|},$$
where the supremum is taken over all subintervals $J'$ of $J$. Note that when $\varphi$ is $C^2$, we have
$$\cN(\varphi|J)=\sup_{x\in J} \frac{|\varphi''(x)|}{|\varphi'(x)|} |J|.$$

Let $\Crit=\Crit(f)$, $\CV=f(\Crit)$ and $\ell_{\max}=\ell_{\max}(f)$. For each $c\in\Crit$ and $\delta>0$, let $\tB(c;\delta)$ denote the component of $f^{-1}(B_\delta(f(c)))$ that contains $c$, let $\ell_c$ denote the order of $c$, and let
\begin{equation}\label{eqn:Dcdelta}
D_c(\delta)=\frac{\delta}{|\tB(c;\delta)|}\asymp  \delta^{1-1/\ell_c}.
\end{equation}
Moreover, let $\tB(\delta)=\bigcup_{c\in\Crit}\tB(c;\delta)$.

Throughout we fix a small constant $\delta_*=\delta_*(f)>0$ such that the intervals $\tB(c;2\delta_*)$, $c\in \Crit$ are pairwise disjoint, and let
\begin{equation}\label{eqn:diststar}
\dist_* (x, \Crit)=\left\{\begin{array}{ll}
\dist (f(x), \CV) & \mbox{ if } x\in \tB(\delta_*)\\
\delta_* &\mbox{otherwise.}
\end{array}
\right.
\end{equation}
Replacing $\delta_*$ by a smaller constant, we may assume the following: for any $c\in\Crit$,
\begin{equation}\label{eqn:dist*der}
x\in \tB(c;\delta_*), \delta=\dist_*(x,\Crit)\mbox{ and  }g\in\Omega_\delta \Rightarrow |Dg(x)|\ge D_c(\delta).
\end{equation}

\subsection{Two intermediate theorems}

\begin{theo}\label{theo:dergwrothrandom}
Consider $f\in\SC_1$ and $\Omega=\sF_1$. For each $\eps>0$ small enough, there exist $\Lambda(\eps)>1$ and $\alpha(\eps)>0$ such that
$$\lim_{\eps\to 0} \Lambda(\eps)=\infty,\mbox{ and }\lim_{\eps\to 0}\alpha(\eps)=0,$$ and such that the following hold:
\begin{enumerate}
\item[(i)]
Let $x\in [0,1]$ and $\bfg\in \Omega_\eps^\N$ be such that $\dist(x, \CV)\le 4\eps$,  $\bfg^j(x)\not\in\tB(\eps)$ for $j=1,2,\ldots, s-1$ and $\bfg^s(x)\in \tB(c;2\eps)$ for some $c\in \Crit$. Then
    \begin{equation}\label{eqn:theo21}
    |D\bfg^s(x)|\ge \frac{\Lambda(\eps)}{D_{c}(\eps)}\exp \left(\eps^{\alpha(\eps)} s\right).
    \end{equation}
\item[(ii)] Let $x\in [0,1]$ and  $\bfg\in\Omega_\eps^\N$ be such that $\bfg^j(x)\not\in \tB(\eps)$ for $0\le j<s$. Then
    \begin{equation}\label{eqn:theo22}
    |D\bfg^s(x)|\ge A\eps^{1-\ell_{\max}^{-1}} \exp (\eps^{\alpha(\eps)} s),
    \end{equation}
    where $A>0$ is a constant independent of $\eps$.
\end{enumerate}
\end{theo}

In the case that $f$ satisfies the Collet-Eckmann condition, i.e. for each $v\in\CV$, $\liminf_{n\to\infty} \log |Df^n(v)|/n>0$, our proof shows that a stronger version of the theorem holds: we can replace $\eps^{\alpha(\eps)}$ by a positive constant independent of $\eps$. For an $S$-unimodal map satisfying a stronger condition of the Benedicks-Carleson type,
this was proved in~\cite{BaV}.

Before we state our Theorem~\ref{theo:reduced}, let us introduce nice sets for random perturbations, as an analogue to the deterministic case.
A {\em nice set for $\eps$-random perturbations}  is a measurable subset $V$ of $[0,1]\times \Omega_\eps^\N$ with the following properties:
\begin{itemize}
\item for each $\bfg\in \Omega_\eps^\N$, $V^\bfg$ is an open neighborhood of $\Crit$, and each component of $V^\bfg$ contains exactly one point of $\Crit$;
\item for each $\bfg\in \Omega_\eps^\N$, $x\in\partial V^\bfg$ (as a subset of $[0,1]$) and $n\ge 1$, we have
$$\bfg^n(x)\not\in V^{\sigma^n \bfg}.$$
\end{itemize}
For each $\bfg\in\Omega_\eps^\N$ and $c\in\Crit$, we use $V^\bfg_c$ to denote the component of $V^\bfg$ which contains $c$.  A positive integer $m$ is called a {\em Markov inducing time} of $(x,\bfg)\in V$,  if there exists an interval $J\ni x$ and a critical point $c\in\Crit$ such that
\begin{itemize}
\item $\bfg^m$ maps $J$ diffeomorphically onto $V^{\sigma^m \bfg}_c$ with
$\cN(\bfg^m|J)\le 1$;
\item if $x\in V_{c_0}^\bfg$ for $c_0\in\Crit$, then 
$$\inf_{y\in J} |D \bfg^m(y)|\ge 2 \frac{|V^{\sigma^m \bfg}_c|}{|V^\bfg_{c_0}|}.$$
\end{itemize}
For $(x,\bfg)\in V$, let $m_V(x,\bfg)$ denote the minimal Markov inducing time of $(x,\bfg)$. (If such a time does not exist, set $m_V(x,\bfg)=\infty$.)

\begin{theo}\label{theo:reduced}
Consider $f\in\SC_1$ with $\Crit\subset (0,1)$, an admissible space $\Omega\ni f$ and a family of probability measures $\nu_\eps\in \sM_\eps(L)$ on $\Omega_\eps$. Fix $p\ge 1$. Then for each $\delta_0>0$ small,  there exist constants $C_0>0$ and $\eps_0>0$ with the following property: For each $\eps\in (0,\eps_0]$,
there exists a nice set $V$ for $\eps$-random perturbations such that
$$\tB(\delta_0)\times \Omega_\eps\subset V\subset \tB(2\delta_0)\times \Omega_\eps;$$
and such that
$$\Leb\times\nu_\eps^\N\left(\{(x,\bfg)\in V: m_V(x,\bfg)>m\}\right)\le C_0 m^{-p}.$$
\end{theo}


\subsection{Estimate of distortion}
In order to control the distortion of iterates of random perturbations of $f$, we shall use the well-known ``telescope'' technique. For interval maps with non-flat critical points, the following notation, appearing in \cite{BC, T1}, is natural to consider.
For $x\in [0,1]$, $\bfg=(g_0,g_1,\ldots)\in \sF_1^\N$ and an integer $n\ge 1$, let
\begin{equation}\label{eqn:dfnA}
A(x, \bfg, n)=\sum_{i=0}^{n-1}\frac{|D\bfg^i(x)|}{\dist(\bfg^i(x), \Crit(g_i))}.
\end{equation}
So if $\bfg^j(x)\in\Crit(g_j)$ for some $j\in\{0,1,\ldots, n-1\}$, then $A(x,\bfg; n)=\infty$.

{\em Note.} In the above formula, we set $\dist(\bfg^i(x),\Crit(g_i))=1$ if $\Crit(g_i)=\emptyset$.

The following lemma is essentially proved in \cite{T1}.
\begin{lemm} \label{lem:theta0}
Let $\Omega\subset \sF_1$ be a space of interval maps which satisfies the admissibility condition (\ref{eqn:constantK0}). Then there exists a constant $\theta_0>0$ such that
for any $(x,\bfg)\in [0,1]\times \Omega^\N$ and any integer $n\ge 1$ with $A(x,\bfg, n)<\infty$, putting
$$J=\left[x-\frac{\theta_0} {A(x,\bfg, n)}, x+\frac{\theta_0}{ A(x,\bfg, n)}\right]\cap [0,1],$$ we have that $\bfg^n|J$ is a diffeomorphism and
$\cN(\bfg^n|J)\le 1.$ Moreover, for each $y\in J$, we have
\begin{equation}\label{eqn:DfAcompare}
e^{-1}\frac{|D\bfg^n(x)|}{A(x,\bfg, n)}\le \frac{|D\bfg^n(y)|}{A(y,\bfg, n)}\le e \frac{|D\bfg^n(x)|}{A(x,\bfg, n)}.
\end{equation}
\end{lemm}
\begin{proof} By assumption, there exists $\theta_0\in (0, (6e)^{-1})$ such that for each interval $I\subset [0,1]$, $y\in I$ and $g\in\Omega$,
$$2|I|\le \dist(y,\Crit(g))\Rightarrow \cN(g|I)\le \frac{1}{4e\theta_0}\frac{|I|}{\dist (y,\Crit)}.$$
Let $n_0$ be the maximal integer in $\{1,2,\ldots, n\}$ such that
\begin{equation}\label{eqn:sumagn}
\sum_{i=0}^{n_0-1}\frac{|\bfg^i(J)|}{\dist(\bfg^i(x), \Crit(g_i))} \le 2e \theta_0<\frac{1}{3}.
\end{equation}
Note that the inequality holds for $n_0=1$. We shall prove that $n_0=n$.

Indeed, (\ref{eqn:sumagn}) implies that for each $0\le i<n_0$, $2|\bfg^i(J)|\le \dist (\bfg^i(x),\Crit(g_i)$. Thus
for each $1\le m\le n_0$, we have
$$\Dist (\bfg^m |J)\le \sum_{i=0}^{m-1} \Dist (g_i|\bfg^i(J))\le \frac{1}{4e\theta_0}\sum_{i=0}^{m-1} \frac{|\bfg^i(J)|}{\dist(\bfg^i(x), \Crit)}\le \frac{1}{2},$$
hence $|\bfg^m(J)|\le e |D\bfg^m(x)||J|$.
If $n_0<n$, then we would have
$$\sum_{i=0}^{n_0} \frac{|\bfg^i(J)|}{\dist(\bfg^i(x), \Crit(g_i))}\le e A(x,\bfg, n_0+1) |J|\le 2 e\theta_0,$$
contradicting the maximality of $n_0$. Thus $n_0=n$.

For each interval $J'\subset J$, we have $|\bfg^i(J')|\le e |J'||\bfg^i(J)|/|J|$, hence
$$\Dist (\bfg^n |J') \le \frac{1}{4e\theta_0} \sum_{i=0}^{n-1}\frac{|\bfg^i(J')|}{\dist(\bfg^i(x),\Crit)}\le \frac{|J'|}{2|J|}.$$
This proves that $\cN(\bfg^n|J)\le 1/2<1.$

For each $y\in J$, and $0\le m< n$, we have
$$\left|\frac{\dist(\bfg^m(y),\Crit(g_m))}{\dist(\bfg^m(x),\Crit(g_m))}-1\right|\le
\frac{|\bfg^m(J)|}{\dist(\bfg^m(x),\Crit(g_m))}\le \frac{1}{3}.$$
Since $e^{-1/2}|D\bfg^m(x)|\le |D\bfg^m(y)|\le e^{1/2}|D\bfg^m(x)|$, this implies
$$e^{-1}\frac{|D\bfg^m(x)|}{\dist(\bfg^m(x),\Crit(g_m))}\le \frac{|D\bfg^m(y)|}{\dist(\bfg^m(y),\Crit(g_m))}\le e \frac{|D\bfg^n(x)|}{\dist(\bfg^m(x),\Crit(g_m))}.$$
The inequality (\ref{eqn:DfAcompare}) follows.
\end{proof}

\subsection{A binding lemma}

We shall use the following lemma for the binding argument.

\begin{defi} Consider $f\in\sA^2$ and $\Omega=\sF_1$. Given $v\in [0,1]$, $\eps>0$ and $C>0$, a positive integer $N$ is called {\em a $C$-binding period} for $(v,\eps)$ if
for each $y\in [0,1]$ with $\dist(y,v)\le  \eps$, each $\bfg\in\Omega_\eps^\N$ and each $0\le j<N$, the following hold:
\begin{align}
\label{eqn:bindcrit}
2|\bfg^j(y)-f^j(v)|& \le \dist(f^j(v),\Crit);\\
\label{eqn:bindder}
e^{-1}|Df^{j+1}(v)|& \le |D\bfg^{j+1}(y)|\le e |Df^{j+1}(v)|;\\
\label{eqn:bindderdis}
C \eps |Df^{j+1}(v)|& \ge |\bfg^{j+1}(y)-f^{j+1}(v)|.
\end{align}
\end{defi}
For any $f\in\sF_1$, $x\in [0,1]$ and $n\ge 1$, let $A(x,f,n)=A(x, \underline{f}, n),$ where $$\underline{f}=(f,f,\ldots).$$

\begin{lemm} \label{lem:shadow}
Let $f\in\sA^2$ and $\Omega=\sF_1$.
Then there exists $\theta_1>0$ such that the following holds
provided that $\eps>0$ is small enough. Let $v\in [0,1]$ and let $N$ be a positive integer such that
$$W:=\sum_{j=0}^{N}|Df^j(v)|^{-1}<\infty\mbox{ and }A(v,f, N)W \le \theta_1/\eps. $$
Then $N$ is an $eW$-binding period for $(v,\eps)$.
\end{lemm}

\begin{proof}
Let $\theta_1\in (0,1/(2e))$ be a small constant such that for each interval $J\subset [0,1]$ and each $x\in J$,
\begin{equation}\label{eqn:theta1}
2|J|\le \dist(x,\Crit)\Rightarrow \cN(f|J)\le \frac{1}{2e\theta_1}\frac{|J|}{\dist(x,\Crit)}.
\end{equation}

Given $y\in [0,1]$ with $\dist(y,v)\le \eps$ and $\bfg\in\Omega_\eps^\N$, let $I_i$ be the closed interval bounded by $y_i:=\bfg^i(y)$ and $v_i:=f^i(v)$, for each $i=0,1,\ldots$. Then we have
\begin{equation}\label{eqn:lengthIi}
|I_{i+1}|\le |f(I_i)|+\eps.
\end{equation}

{\bf Claim.} For each $n=0,1,\ldots, N$, we have the following three inequalities:
\begin{equation}\label{eqn:iea}
|I_n|\le eW \eps |Df^{n}(v)|;
\end{equation}
\begin{equation}\label{eqn:ieb}
\sum_{0\le i<n} \frac{|I_i|}{\dist(f^i(v),\Crit)}\le e\theta_1,
\end{equation}
\begin{equation}\label{eqn:iec}
\sum_{0\le i<n} \Dist (f|I_i)\le \frac{1}{2}.
\end{equation}

We shall prove the claim by induction on $n$.  The case $n=0$ is clear: (\ref{eqn:iea}) follows from
the construction of $I_0$ and (\ref{eqn:ieb}) and (\ref{eqn:iec}) are trivial.
Assume that these statements hold for all $n$ not greater than some $n_0\in \{0,1,\ldots, N-1\}$. Let us consider the case $n=n_0+1$.
We first prove that (\ref{eqn:ieb}) holds for $n=n_0+1$. Indeed, since (\ref{eqn:iea}) holds for $n=0,1,\ldots, n_0$, we have
\begin{equation*}
\sum_{i=0}^{n_0}\frac{|I_i|}{\dist(v_i, \Crit)}\le
\sum_{i=0}^{n_0}\frac{e W \eps|Df^{i}(v)|}{\dist (v_i,\Crit)}\le  e \eps W A(v,f, N)
\le e \theta_1.
\end{equation*}
This proves (\ref{eqn:ieb}). In particular, for each $0\le i\le n_0$,
$2|I_i|\le \dist(v_i,\Crit)$.
By (\ref{eqn:theta1}), it follows that
$$\sum_{0\le i<{n_0+1}} \Dist (f|I_i)\le \frac{1}{2e\theta_1}\sum_{0\le i<n_0+1} \frac{|I_i|}{\dist(v_i,\Crit)}\le \frac{1}{2}<1.$$
So (\ref{eqn:iec}) holds for $n=n_0+1$.
Now let us prove that (\ref{eqn:iea}) holds for $n=n_0+1$.
By the mean value theorem and (\ref{eqn:lengthIi}), for each $i=0,1,\ldots, n_0$, there exists
$x_i\in I_i$ such that $|I_{i+1}|\le D_i |I_i|+\eps,$
where $D_i=|Df(x_i)|$. Thus
$$|I_{n_0+1}| \le \prod_{i=0}^{n_0} D_i |I_0| + \eps \left(1+ \sum_{k=1}^{n_0} \prod_{i=k}^{n_0} D_i \right). $$
Since (\ref{eqn:iec}) holds for $n=n_0+1$,  for each $0\le k\le n_0$ we have
$$\prod_{i=k}^{n_0} D_i\le e \prod_{i=k}^{n_0} |Df(v_i)|=e\frac{|Df^{n_0+1}(v)}{|Df^k(v)|}.$$
Therefore,
\begin{align*}
|I_{n_0+1}|
& \le e |Df^{n_0+1}(v)| \left(|I_0|+ \frac{\eps}{e |Df^{n_0+1}(v)|}+ \sum_{k=1}^{n_0}\frac{\eps}{|Df^k(v)|}\right)\\
& \le e \eps |Df^{n_0+1}(v)|\left( 1 + \sum_{k=1}^{n_0+1} |Df^{k}(v)|^{-1}\right) \le eW |Df^{n_0+1}(v)| \eps.
\end{align*}
This proves that (\ref{eqn:iea}) holds for $n=n_0+1$. We have completed the induction step, and thus the proof of the claim.

Now let us verify that the three inequalities in the definition of binding period holds with $C=eW$, in the case that $\eps>0$ is small enough.
Clearly, (\ref{eqn:bindcrit}) follows from (\ref{eqn:ieb}), and (\ref{eqn:bindderdis}) follows from (\ref{eqn:iea}).
To prove (\ref{eqn:bindder}), note
\begin{equation}\label{eqn:vi}
\frac{\theta_1}{\eps}\ge A(v,f, N)W\ge\sum_{i=0}^{N-1} \frac{|Df^i(v)|}{\dist (v_i,\Crit)} \frac{1}{|Df^{i+1}(v)|}= \sum_{i=0}^{N-1} \frac{1}{|Df(v_i)|\dist(v_i,\Crit)}.
\end{equation}
This implies that for each $0\le i<N$,
$|Df(v_i)|\dist(v_i,\Crit)\ge \eps$. By the non-flatness of critical points of $f$, it follows that $|Df(v_i)|\succ \eps^{1-\ell_{\max}^{-1}}$ for all $0\le i<N$. Provided that $\eps$ is small enough, we have
$$|Df(y_i)|\ge |Df(v_i)|/e\ge 10 \eps \ge 10|Dg_i(y_i)-Df(y_i)|.$$ In particular, $Dg_i(y_i)$ and $Df(y_i)$ have the same sign for each $0\le i<N$. Moreover, for each $0\le j<N$ we have
\begin{equation*}
\left|\sum_{i=0}^{j}\log\left( \frac{D g_i(y_i)}{D f(y_i)}\right) \right|\le \sum_{i=0}^{j} \frac{|D g_i(y_i)-D f(y_i)|}{|D f(y_i)|} \le  \sum_{i=0}^{N-1} \frac{e\eps} {|D f(v_i)|}\le e\theta_1,
\end{equation*}
where the last inequality follows from (\ref{eqn:vi}) since $\dist(v_i,\cC)\le 1$.
By (\ref{eqn:iec}), we have
$$\left|\sum_{i=0}^{j}\log\left( \frac{Df(y_i)}{Df(v_i)}\right) \right|\le \sum_{i=0}^{N-1} \Dist (f|I_i)\le \frac{1}{2}.$$
Since
$$\log \left(\frac{|D\bfg^{j+1}(y)|}{|Df^{j+1}(v)|}\right)=\sum_{i=0}^{j}\log\left( \frac{Df(y_i)}{Df(v_i)}\right) +\sum_{i=0}^{j}\log \left(\frac{Dg_i(y_i)}{Df(y_i)}\right),$$
(\ref{eqn:bindder}) follows.
\end{proof}

\subsection{Expanding away from critical points}
The following is a well-known result due to M\'a\~ne, see~\cite{Mane}.
\begin{prop}\label{prop:mane}
Let $f\in\sA_2$. For each neighborhood $U$ of $\Crit$, there exists $C>0$ and $\lambda>1$ such that for each $x\in [0,1]$
and $n\ge 1$, if $x, f(x), \ldots, f^{n-1}(x)\not\in U$, then  $|Df^n(x)|\ge C\lambda^n$.
Moreover, for a.e. $x\in [0,1]$ there exists an integer $n\ge 1$ such that $f^n(x)\in U$.
\end{prop}

The following consequence is also known. We provide a proof for the reader's convenience. 

\begin{prop} \label{prop:mane1}
Consider $f\in\sA_2$ and an admissible space $\Omega\ni f$.  For any   neighborhood $U$  of $\Crit$, there exist $K>1$ and $\eta>0$ such that the following hold provided that $\eps>0$ is small enough:
\begin{enumerate}
\item[(i)] For $x\in [0,1]$, $\bfg\in\Omega_\eps^\N$ and $n\ge 1$, if $\bfg^j(x)\not\in U$ for all $0\le j<n$, then $|D\bfg^n(x)|\ge K^{-1} e^{\eta n}$.
\item[(ii)] For any $\bfg\in\Omega_\eps^\N$ and $n\ge 1$,
$$\left|\{x\in [0,1]: \bfg^j(x)\not\in U\mbox{ for } 0\le j<n\}\right|\le Ke^{-\eta n}.$$
\end{enumerate}
\end{prop}
\begin{proof}
(i). Let $U_0$ be a neighborhood of $\Crit$ such that $U_0\Subset U$. Let $C>0$ and $\lambda>1$ be given by Proposition~\ref{prop:mane} for $U_0$ and let $N$ be an integer such that $C\lambda^{N}>4$. By continuity, provided that $\eps>0$ is small enough,  for any $x\in [0,1]$, any $\bfg\in \Omega_{\eps}^\N$ and any $0\le n \le N$, we have
$$|f^n(x)-\bfg^n(x)|<\dist(\partial U, \partial U_0), \,\, |Df^n(x)-D\bfg^n(x)|< C/2.$$

Now consider $x\in [0,1]$, $\bfg\in\Omega_{\eps}^\N$ and $n$ such that $\bfg^j(x)\not\in U$ for all $0\le j<n$.  Assume $\eps$ is small, and write $n=kN+r$ with $k\in\N$ and $0\le r<N$. Then we have
$$|D\bfg^{r}(x)|\ge |Df^r(x)|-C/2\ge C\lambda^r/2,$$
and for each $0\le i<k$,
$$|D\bfg^{N} (\bfg^{iN+r}(x))|\ge |Df^{N} (\bfg^{iN+r}(x))|-C/2\ge C\lambda^{N}/2 \ge 2. $$
Thus $$|D\bfg^{n}(x)|\ge 2^k \frac{C\lambda^r}{2}\ge K^{-1} e^{\eta n},$$
where $K=4/C$ and $\eta=\log 2/N$. 

(ii) For each $\bfg\in\Omega_\eps^\N$ and each $n\ge 1$, let
$$\Lambda_n^\bfg(U)=\{x\in [0,1]:\bfg^j(x)\not\in U \mbox{ for } 0\le j< n\},$$
and let $\Lambda_\infty^\bfg(U)=\bigcap_{n=1}^\infty \Lambda_n^\bfg(U)$. Let $U_0\Subset U$ be an open neighborhood of $\Crit$ and define $\Lambda_n^\bfg (U_0)$, $\Lambda_\infty^\bfg(U_0)$ similarly.

Assume $\eps>0$ small.
By statement (i) of this proposition, for each $x\in \Lambda_n^\bfg(U)$, we have $A(x,\bfg,n)\asymp |D\bfg^n(x)|$. So by Lemma~\ref{lem:theta0}, there exists $\tau>0$ independent of $n, x$ such that $\bfg^n$ maps a neighborhood $J_n(x)$ of $x$ diffeomorphically onto an interval of length $\tau$ with $\cN(\bfg^n|J_n(x))\le 1.$ By shrinking $\tau$ if necessary, we may assume that $J_n(x)\subset \Lambda_n^\bfg(U_0)$ for each $x\in\Lambda_n^\bfg(U)$.

Let $\rho$ be a small constant to be determined. Let us prove that there exists $N=N(\rho)$ such that $|\Lambda^\bfg_N(U_0)|<\rho$ holds whenever $\eps>0$ is small enough. Indeed, for $\bff=(f,f,\cdots),$ since $\Lambda_\infty^\bff(U_0)$ is a compact set with Lebesgue measure zero, there exists $N$ such that $|\Lambda_N^\bff(U_0)|<\rho$. Assuming $\eps>0$ small, then for each $\bfg\in\Omega_\eps^\N$, $\Lambda_N^\bfg(U_0)$ is contained in a small neighborhood of $\Lambda_N^\bff(U_0)$.  The statement follows.

Now let $\eta_k=\sup_{\bfg\in\Omega_\eps} |\Lambda_{kN}^\bfg(U)|$. It suffices to prove that $\eta_{k+1}<\eta_k/2$ holds for all $k=1,2,\ldots$ provided that $\eps>0$ is small. Let $\hLambda^\bfg$ be the union of the intervals $J_N(x)$, $x\in \Lambda_{N}^\bfg(U)$. Then $|\hLambda^\bfg|\le |\Lambda_{N}^\bfg(U_0)|<\rho$. Since $\bfg^N(\Lambda_{(k+1)N}^\bfg(U))\subset \Lambda_{kN}^{\sigma^N \bfg}(U)$, it follows that
$$|J_N(x)\cap \Lambda_{(k+1)N}^\bfg(U)|< e \eta_k\tau^{-1} |J_N(x)|.$$  By Besicovic's covering lemma, there exists a subfamily of $\{J_N(x): x\in \Lambda_N^\bfg(U)\}$ with uniformly bounded intersection multiplicity which forms a covering of $\Lambda_N^\bfg(U)$. Thus
$$|\Lambda_{(k+1)N}^\bfg(U)|\le C \eta_k |\hLambda^\bfg|\tau^{-1}< C\eta_k\rho/\tau,$$
where $C$ is a universal constant.
We can choose $\rho$ small so that the right hand side is less than $1/2$.  This proves $\eta_{k+1}<\eta_k/2$. 
\end{proof}
\section{Proof of the Main Theorem}\label{sec:pfmain}
In this section, we shall deduce the Main Theorem from Theorem~\ref{theo:reduced}. So let $f, \Omega$ and $\nu_\eps$ be as in the Main Theorem.  As we mentioned in \S\ref{subsec:space}, we may assume
$\Crit\subset (0,1)$.

\subsection{Physical measures for random perturbations}
In this section, we shall prove statement (i) of the Main Theorem.

Let $\cP$ denote the set of Borel probability measures on $[0,1]$ and let $\cT_\eps: \cP\to \cP$ be defined as
$$\cT_\eps m(A)=\int_\Omega m(g^{-1}(A)) d\nu_\eps(g)\mbox{ for each Borel set } A\subset [0,1].$$
Note that a stationary measure $\mu_\eps$ for $\nu_\eps$ is just a fixed point of $\cT_\eps$.
The following is standard.

\begin{lemm}\label{lem:exist}
For each $m\in \cP$, any accumulation point of $\frac{1}{n}\sum_{j=0}^{n-1} \cT_\eps^j m$ in the weak star topology is a stationary measure.
\end{lemm}
Assume $\nu_\eps\in \sM_\eps(L)$. Then for each $m\in\cP$ and each Borel set $A\subset [0,1]$, we have
$$\cT_\eps m(A)=\int_0^1 p_\eps(A| x) dm(x)\le L \left(\frac{|A|}{2\eps}\right)^{1/L}.$$
In particular, we have
\begin{lemm}\label{lem:abc}
A stationary measure $\mu_\eps$ for $\nu_\eps$ is absolutely continuous.
\end{lemm}

A subset $E$ of $[0,1]$ is called {\em almost forward invariant for $\eps$-perturbations} if $|g(E)\setminus E|=0$ holds for $\nu_\eps$-a.e. $g\in \Omega_\eps$.
If both $E$ and $E^c=[0,1]\setminus E$ are almost forward invariant for $\eps$-perturbations, then we say that $E$ is {\em almost completely invariant for $\eps$-perturbations.}

\begin{lemm} \label{lem:ergodic0}
There exists $\eta>0$ such that for each $\eps>0$ small, if $E\subset [0,1]$ is a Borel set which is almost forward invariant for $\eps$-perturbations and with $\Leb(E)>0$, then there exists an interval $J\subset [0,1]$ such that $|J|\ge \eta$ and $|J\setminus E|=0$.
\end{lemm}

The lemma will be proved in the next subsection after we recall Theorem~\ref{theo:reduced}.

We say that a stationary measure $\mu_\eps$ is {\em ergodic} if for each Borel set $E\subset [0,1]$ which is almost completely invariant for $\eps$-perturbations, we have either $\mu_\eps(E)=0$ or $\mu_\eps(E)=1$.

\begin{lemm} \label{lem:ergodic}
For each $\eps>0$ small enough, there exists a unique stationary measure $\mu_\eps$. This stationary measure $\mu_\eps$ is ergodic.
\end{lemm}
\begin{proof}
Take an integer $N> 3/\eta$ and let $I_i=[i/N, (i+1)/N]$, $0 \le i <N$, where $\eta$ is as in Lemma~\ref{lem:ergodic0}.
Since $f$ is ergodic with respect to the Lebesgue measure, for each $0\le i, j\le n$ there exist positive integers $k_{ij}$ and $m_{ij}$ such that
$f^{k_{ij}}(I_i)\cap f^{m_{ij}}(I_j)$ is a non-degenerate interval. By continuity, there exists $\eps_0>0$ such that for each $\bfg\in\Omega_{\eps_0}$,
$\bfg^{k_{ij}}(I_i)\cap \bfg^{m_{ij}}(I_j)$ is a non-degenerate interval.

Now assume $\eps\in (0,\eps_0]$ is small so that the conclusion of Lemma~\ref{lem:ergodic0} holds. Then for each $E\subset [0,1]$ which is almost completely invariant set for $\eps$-perturbations, we have either $|E|=0$ or $|E^c|=0$. Indeed, otherwise, there exists $i,j\in \{0,1,\ldots, N-1\}$ such that $|I_i\setminus E|=|I_j\setminus E^c|=0$. By invariance, for $\nu_\eps^\N$-a.e. $\bfg\in\Omega_\eps$ we have $\bfg^{k_{ij}}(I_i)\cap \bfg^{m_{ij}}(I_j)$ is contained in $E\cap E^c$ up to a set with Lebesue measure zero. That is, $\bfg^{k_{ij}}(I_i)\cap \bfg^{m_{ij}}(I_j)$ has Lebesgue measure zero, which contradicts what we proved above.

Since a stationary measure is absolutely continuous, it follows that a stationary measure is ergodic. The uniqueness follows.
\end{proof}

\begin{proof}[Proof of the Main Theorem (i)]
By Lemma~\ref{lem:ergodic}, for each $\eps>0$ small enough,  there exists a unique stationary measure $\mu_\eps$ and it is ergodic.
For $x\in [0,1]$, let
$$\Omega_\eps^\N(x)=\left\{\bfg\in\Omega_\eps^\N: \frac{1}{n}\sum_{j=0}^{n-1} \delta_{\bfg^j(x)}\to \mu_\eps
\mbox{ in the weak star topology}\right\},$$ and let $u_\eps(x)=\nu_\eps^\N(\Omega_\eps^\N(x))$.
To complete the proof, we shall prove that for each $x\in [0,1]$, $u_\eps(x)=1$.

To this end, we first note that $\mu_\eps\times \nu_\eps^\N$ is an ergodic probability invariant measure for $F$. See for example ~\cite[Section 7.2]{Ar}.
By Birkhorff's Ergodic Theorem, $u_\eps(x)=1$ holds for $\mu_\eps$-a.e. $x\in [0,1]$, so
\begin{equation}\label{eqn:ae1}
\int_0^1 u_\eps(y) d \mu_\eps(y)=1.
\end{equation}
For each $n=1,2,\ldots$, let $p_n(\cdot|x):=\cT_\eps^n(\delta_x)$ and $\mu_{n,x}=\frac{1}{n}\sum_{j=1}^n p_n(\cdot|x)$. Then $\mu_{n,x}(A)\le L(|A|/2\eps)^{1/L}$ holds for each $n=1,2,\ldots$ and each Borel set $A\subset [0,1]$.
Since $\mu_{n,x}$ converges to the unique stationary measure $\mu_\eps$ and $0\le u_\eps\le 1$, we have
$$\int_0^1 u_\eps(y) d\mu_{n,x}(y) \to \int_0^1 u_\eps(y) d\mu_\eps(y).$$
Note that for each $x$, $\bfg$, if $\sigma\bfg\in \Omega_\eps^\N (g_0(x))$ then $\bfg\in \Omega_\eps^\N (x)$. Thus
$$u_\eps(x)\ge \int_{\Omega_\eps} u_\eps(g(x)) d\nu_\eps(g)= \int_0^1 u_\eps(y) p_1(dy|x).$$
This implies that for each  $n$,
$$u_\eps(x)\ge \int_0^1 u_\eps(y) p_n(dy|x),$$
hence $$u_\eps (x)\ge \int_0^1 u_\eps (y) d\mu_{n,x}(y) \to\int_0^1 u_\eps (y) d\mu_\eps(y)=1,$$
where the last equality follows from (\ref{eqn:ae1}).
\end{proof}

\subsection{Strong stochastic stability}
In the rest of this section, we shall prove $\mu_\eps\to \mu$ in the strong topology.

By Theorem~\ref{theo:reduced}, there exist $\delta_0>0$ and $\eps_0>0$ such that for $\eps\in (0,\eps_0]$ there exists a nice set $V=V_\eps$ for $\eps$-random perturbations with $\tB(\delta_0)\subset V^\bfg\subset \tB(2\delta_0)$ for all $\bfg\in\Omega_\eps$ and with
\begin{equation}\label{eqn:tail}
\Leb\times \nu_\eps^N \left(\{(x,\bfg)\in V: m_V(x,\bfg) >m \}\right)\le C_0 m^{-2},
\end{equation}
where $C_0>0$ is a constant. In the following we fix such a choice of $V_\eps$ for each $0<\eps\le \eps_0$. Moreover, let $U_\eps=\{(x,\bfg)\in V_\eps: m_{V_\eps}(x,\bfg)<\infty\}$ and let $G_\eps: U_\eps\to V_\eps$ denote the map $(x,\bfg)\mapsto F^{m_{V_{\eps}}(x,\bfg)}(x,\bfg)$, where $F$ is the skew product defined in (\ref{eqn:skewF}). Since $\Leb\times \nu_\eps^\N(V_\eps\setminus U_\eps)=0$, $G_\eps^n(x,\bfg)$ is well-defined for each $n\ge 0$ and almost every $(x,\bfg)\in V_\eps$.
Note that if $(x,\bfg)\in \dom (G_\eps^n)$ and $k=\sum_{i=0}^{n-1} m_{V_\eps}(G_\eps^i(x,\bfg))$, then $\bfg^k$ maps an interval $J_k^\bfg(x)$ diffeomorphically onto a component of $V_\eps^{\sigma^k\bfg}$ and $\cN(\bfg^k|J_k^\bfg(x))\le 2$,
$|J_k(x)|\le 2^{-n}$.

\begin{proof}[Proof of Lemma~\ref{lem:ergodic0}]
Assume $\eps$ small. 
Then by Proposition~\ref{prop:mane1}, 
we have
\begin{equation}\label{eqn:enter}
\Leb\left(\bigcup_{n=0}^\infty \bfg^{-n}(\tB(\delta_0))\right)=1.
\end{equation}
for each $\bfg\in\Omega_\eps^\N$.
Thus $|E\cap \tB(\delta_0)|>0$. So there exist $x\in \tB(\delta_0)\cap E$ and $\bfg\in \Omega_\eps^\N$ with the following properties:
\begin{itemize}
\item $x$ is a Lebesgue density point of $E$;
\item $|\bfg^n(E)\setminus E|=0$ for each $n=0,1,\ldots$;
\item $(x,\bfg)\in \dom(G_\eps^n)$ for each $n=0,1,\ldots.$
\end{itemize}
The last property implies that there exists a sequence of integers $s_1<s_2<\cdots$ and $c\in\Crit$ such that $\bfg^{s_n}$ maps a neighborhood $I_n$ of $x$ diffeomorphically onto $\tB(c;\delta_0)$ with $\cN(\bfg^{s_n}|I_n)\le 2$ and $|I_n|\to 0$ as $n\to\infty$. Thus $$\frac{|\tB(c;\delta_0)\setminus E|}{|\tB(c;\delta_0)|}\le e^2 \frac{|I_n\setminus E|}{|I_n|}\to 0.$$ It follows that $|\tB(c;\delta_0)\setminus E|=0$.
\end{proof}

Let $L^1=L^1([0,1])$ denote the Banach space of all $L^1$ functions $\varphi: [0,1]\to \R$ with respect to the Lebesgue measure and let $\|\varphi\|_1$ denote the $L^1$ norm of $\varphi$.
Given $J\subset [0,1]$, $\bfg\in \Omega^\N$ and an integer $n\ge 0$, define
$$\sL_{J,n}^\bfg(x)=
\sum_{y\in \bfg^{-n}(x)\cap J} \frac{1}{|D\bfg^n(y)|}, \mbox{ and } \widehat{\sL}_{J,n}^\bfg=\frac{1}{|J|}\sL_{J,n}^ \bfg.
$$
These are functions in $L^1$, supported on $\bfg^n(J)$ and $\sL_{J,n}^\bfg$ is the density function of the absolutely continuous measure $(\bfg^n)_*(\Leb|J).$

We shall need the following lemma. 
\begin{lemm}\label{lem:L1compact}
For each $\rho>0$ there exists a compact subset $\cK(\rho)$ of $L^1$ such that for any interval $J\subset [0,1]$, any $\bfg\in\Omega$ and any integer $n\ge 0$, if $|\bfg^n(J)|>\rho$, $\bfg^n$ maps $J$ diffeomorphically onto its image and $\cN(\bfg^n|J)\le 2$, then
$\widehat{\sL}_{J,n}^\bfg\in \cK(\rho)$.
\end{lemm}
\begin{proof}
For $C>1$, let $\sD_{C}$ denote the subset of $L^1$ consisting of maps $\psi: [0,1]\to \R$  for which there exists an interval $I=I_\psi \subset [0,1]$ such that
\begin{itemize}
\item $|I|\ge C^{-1}$;
\item $\psi(x)=0$ for all $x\in [0,1]\setminus I$;
\item $\psi(x)>0$ for $x\in I$;
\item  $|\psi(x)-\psi(y)|\le C \psi(x) |x-y|$ for all $x, y\in I$;
\item $\int_0^1 \psi(x) dx =1$.
\end{itemize}
Clearly, $\sD_C$ is a compact subset of $L^1$. Moreover,
for each $\rho>0$ there exists $C>1$ such that for any $\bfg, J, n$ as in the lemma, we have $\widehat{\sL}_{J,n}^\bfg\in \sD_C$. So taking $\cK(\rho)=\sD_C$ completes the proof.
\end{proof}
\begin{proof}[Proof of the Main Theorem (ii)]
Take an arbitrary $c\in\Crit$ and let $Z=\tB(c;\delta_0)$. Let $$\varphi_n(x)=\int_{\bfg\in \Omega_\eps^\N}\sL_{Z,n}^\bfg(x) d\nu_\eps^\N.$$
Since $\varphi_i(x) dx=\cT_\eps^i(\Leb|Z)$, we have that $\frac{1}{n}\sum_{i=0}^{n-1}\varphi_i(x) dx$ converges to the unique stationary measure $\mu_\eps$. Thus it suffices to prove that there exists a compact subset $\cK$ of $L^1$ independent of $\eps$ and $n$ such that $\varphi_n\in \cK$.
To this end, we shall prove that for each $\eta>0$, there exists a compact subset $\cK_\eta$ of $L^1$ such that for each $n$,
$\varphi_n$ can be written in the following form:
\begin{equation}\label{eqn:varphidec}
\varphi_n=\varphi_n^0+\varphi_n^1+\varphi_n^2,
\end{equation}
where $\|\varphi_n^i\|_1\le 2\eta$, $i=0,1$ and $\varphi_n^2\in\cK_\eta$.

Let $V=V_\eps$ and $G=G_\eps$. For $(x,\bfg)\in V$, let $\sM(x, \bfg)$ denote the collection of positive integers of the form $\sum_{j=0}^{n-1} m_V(G^j(x,\bfg))$, where $n$ runs over all positive integers for which $(x,\bfg)\in \dom (G^n)$.
For $m\ge 1$ and $k\ge 1$, let
$$U_{0,m}=\{(x,\bfg)\in V: m_V(x,\bfg)=m\},$$
and
$$U_{k,m}=\{(x,\bfg)\in V: k\in \sM(x,\bfg) \mbox{ and } F^k(x,\bfg)\in U_{0,m}\}.$$
Moreover, let $H_{k,m}=U_{k,m}\cap (Z\times \Omega_\eps^\N)$ for each $k\ge 0$ and $m\ge 1$.
Let $\cH_{k,m}^\bfg$ (resp. $\cU_{k,m}^\bfg$) denote the collection of the components of $H_{k,m}^\bfg$ (resp. $U_{k,m}^\bfg$). Note that if $x\in J\in \cH_{k,m}^\bfg$, then $k+m\in\sM(x,\bfg)$ and $J\subset J_{k+m}^\bfg(x)$, hence
\begin{equation}\label{eqn:distorcH}
\cN(\bfg^n|J)\le 2.
\end{equation}

Fix $n\ge 0$ and let $\Sigma_n=\{(k,m)\in \N^2: 0\le k\le n, m+k>n\}$.
Then the sets $H_{k,m}$, $(k,m)\in \Sigma_n$ are pairwise disjoint, and for a.e. $\bfg\in \Omega_\eps^\N$, $\cH^\bfg:=\bigcup_{(k,m)\in\Sigma_n}\cH_{k,m}^\bfg$
forms a measurable partition of $Z$ up to a set of Lebesgue measure zero. Then
\begin{equation}\label{eqn:decomp}
\varphi_n=\int_{\Omega_\eps^\N} \sum_{J\in\cH^\bfg} \sL_{J,n}^\bfg d\nu_\eps(\bfg).
\end{equation}

Now fix $\eta>0$. For each $\bfg\in\Omega_\eps^\N$, we shall introduce a decomposition
\begin{equation}
\cH^\bfg=\cH^{\bfg,0}\cup \cH^{\bfg,1} \cup\cH^{\bfg,2},
\end{equation}
and write
$$\varphi_n^i=\int_{\Omega_\eps^\N} \sum_{J\in\cH^{\bfg, i}} \sL_{J,n}^\bfg d\nu_\eps^\N.$$

The set $\cH^{\bfg,0}$ is the collection of elements $J$ of $\cH^\bfg$ for which $\partial J\cap\partial Z\not=\emptyset$ and $|J|<\eta$. For each $\bfg\in\Omega_\eps^\N$, $\cH^{\bfg,0}$ has at most two elements. Thus
$\|\varphi_n^0\|_1\le 2\eta$.

To define $\cH^{\bfg, 1}$, we first observe that for each $(k,m)\in\Sigma_n$ and $\bfg\in\Omega_\eps^\N$, we have
$|H_{k,m}^\bfg|\le |U_{k,m}^\bfg|\le C_1 |U_{0,m}^{\sigma^k\bfg}|,$
where $C_1>0$ is a constant. Indeed, for any $(x,\bfg) \in U_{k,m}$,
we have $\bfg^k(U_{k,m}^\bfg\cap J_k^\bfg(x))\subset U_{0,m}^{\sigma^k\bfg}$. Since $\cN(\bfg^k|J_{k}^\bfg(x))\le 2$, the statement follows.
Let $M$ be a positive integer such that $C_0C_1M^{-1}<\eta$ and let $\cH^{\bfg, 1}$ be the collection of all components of
$\bigcup_{(k,m)\in \Sigma_n, m>M}\cH^\bfg_{k,m}$ which are not contained in $\cH^{\bfg,0}$.

Let us prove $\|\varphi_n^1\|_1<2\eta$.
Let $G_{k,M}=\bigcup_{(k,m)\in\Sigma_n, m>M}H_{k,m}$ and let $G_M=\bigcup_{k=0}^n G_{k,M}$.
Then, for each $0\le k\le n$,
\begin{multline*}
\Leb\times \nu_\eps^\N (G_{k,M})
=\int_{\Omega_\eps^\N} |G_{k,M}^\bfg| d\nu_\eps^\N(\bfg)
 \le C_1\int_{\Omega_\eps^\N} \sum_{m>\max(M, n-k)} |U_{0,m}^{\sigma^k \bfg}| d\nu_\eps^\N(\bfg)\\
 =C_1\int_{\Omega_\eps^\N} \sum_{m>\max(M, n-k)} |U_{0,m}^{\bfg}| d\nu_\eps^\N(\bfg)\\
 =C_1 \Leb\times\nu_\eps^\N \left(\left\{(x,\bfg)\in V: m_V(x,\bfg)> \max(M, n-k)\right\}\right),
\end{multline*}
which implies by (\ref{eqn:tail})  that
$$\|\varphi_n^1\|_1\le \Leb\times \nu_\eps^\N(G_M) \le C_0C_1 \sum_{k=0}^n \max (M, n-k)^{-2}< 2C_0C_1 M^{-1}<2\eta.$$

Finally, define $\cH^{\bfg,2}=\cH^\bfg\setminus (\cH^{\bfg, 0}\cup \cH^{\bfg, 1})$.
Let us show that for each $J\in\cH^{\bfg,2}$, $|\bfg^n(J)|$ is bounded from below by a constant $\rho=\rho(\eta)>0$.
Indeed, letting
$(k,m)\in\Sigma_n$ be such that $J\in\cH_{k,m}^\bfg$, it suffices to show that $|\bfg^{k+m}(J)|$ is bounded away from zero, since $$|\bfg^{k+m}(J)|\le (\sup|f'|+\eps)^{k+m-n} |\bfg^n(J)|\le (\sup|f'|+\eps)^M |\bfg^n(J)|.$$
If $\partial J\cap\partial Z=\emptyset$ then $\bfg^{k+m}(J)$ is a component of $V^{\bfg'}$, $\bfg'=\sigma^{k+m}\bfg$, hence its length is bounded away from zero. If $\partial J\cap \partial Z=\emptyset$, then $|J|\ge \eta$, so by definition of $m_V$, $|\bfg^{k+m}(J)|$ is bounded away from zero as well.

Together with (\ref{eqn:distorcH}), by Lemma~\ref{lem:L1compact}, this implies that there is a compact subset $\cK(\rho)$ of $L^1$ such that  for each $J\in\cH^{\bfg,2}$, $\widehat{\sL}_{J,n}^\bfg\in \cK(\rho)$. Therefore $\varphi_n^2$ is contained in some compact subset $\cK_\eta$ of $L^1$.
\end{proof}

\section{Some properties of the deterministic dynamics}\label{sec:derterministic}
In this section, we study the dynamics of an interval map $f\in \LD$, see (\ref{eqn:LD}).
The main results are the following Propositions~\ref{prop:shadowablegrowth} and~\ref{prop:quantifiedmane}, which will be used to study derivative growth along random orbits in the next section.
Recall that $d_*$ was defined in (\ref{eqn:diststar}).

\begin{prop}  \label{prop:shadowablegrowth}
Given $f\in\LD$, $L>1$, $\theta\in (0,1)$ and $\zeta>0$, for any critical value $v$ and any $\delta>0$ small enough
there exists a positive integer $M_v(\delta)$ such that
the following hold:
\begin{equation}\label{eqn:smalla}
A(v, f, M_v(\delta))\le \theta /\delta,
\end{equation}
\begin{equation}\label{eqn:noretmv}
f^j(v)\not\in \tB(L\delta)\mbox{ for each } j=0,1,\ldots, M_v(\delta)-1,
\end{equation}
and
\begin{equation}\label{eqn:dermv}
|Df^{M_v(\delta)+1}(v)|\ge  \left(\frac{\delta'}{\delta}\right)^{1-\zeta},
\end{equation}
where $\delta'=\max(\dist_* (f^{M_v(\delta)}(v),\Crit),\delta)$.
Moreover, we have
\begin{equation}\label{eqn:mvdelta}
M_v(\delta)\to\infty \mbox{ as } \delta\to 0.
\end{equation}
\end{prop}

Let $\cL_c(\delta)$ denote the collection of all orbits $\{f^j(x)\}_{j=0}^n$ with
$f^j(x)\not\in \tB(\delta)$ for each $j=0,1,\ldots, n-1$ and $f^n(x)\in \tB(c;2\delta)$.
The following proposition is a variation of a result in \cite{BS} using a different argument.
\begin{prop} \label{prop:quantifiedmane}
Given $f\in \LD$, there exists a constant $\kappa_0>0$ such that for each $\delta>0$,
the following holds.   For $\{f^j(x)\}_{j=0}^n\in \cL_c(\delta) $,
putting $\delta''=\max (\dist(x, \CV), \delta),$
we have $$|Df^n(x)|\ge \frac{\kappa_0}{D_c(\delta)} \left(\frac{\delta}{\delta''}\right)^{1-\ell_{\max}^{-1}}.
$$
\end{prop}

We start by stating some known facts in \S~\ref{subsec:deterpre}. Propositions~\ref{prop:shadowablegrowth} and~\ref{prop:quantifiedmane} will be proved in \S~\ref{subsec:shadowable} and~\S~\ref{subsec:1landing} respectively.
\subsection{Some facts}\label{subsec:deterpre}
The following is ~\cite[Theorem 1]{BRSS}.
\begin{prop}\label{prop:bc}
If $f\in\LD$, then $f$ is backward contracting in the following sense: For each $\delta>0$ small, there exists $\widehat{r}(\delta)>1$ such that $\lim_{\delta\to 0} \widehat{r}(\delta)=\infty$ and such that for each $\delta\in (0,\delta_0)$ and each integer $s\ge 0$, if $W$ is a component of $f^{-s}(\tB(\widehat{r}(\delta)\delta))$ and
$\dist(W, \CV)\le \delta$, then $|W|<\delta$.
\end{prop}

The notion of ``backward contraction'' was introduced in~\cite{RL}. Actually, the converse of the proposition is also true, see \cite{LiS}.
We shall use the following consequence of the proposition.

\begin{lemm} \label{lem:derivativeBC}
Assume $f\in\LD$. Then for each $\delta>0$ small, there exists $r(\delta)>1$ with $r(\delta)\to\infty$ as $\delta\to 0$ such that
if $\{f^j(y)\}_{j=0}^n\in\cL_c(\delta)$ and $\dist(y,\CV)\le \delta$,
then
\begin{equation}\label{eqn:derivativebc}
|Df^{n}(y)|\ge \frac{r(\delta)}{D_c(\delta)}.
\end{equation}
\end{lemm}
\begin{proof}
Assume that $\delta>0$ is small so that $\hat{r}(\delta)>4$. By backward contraction, there exists an interval $J\ni y$ such that $f^n$ maps $J$ diffeomorphically onto $\tB(c;\hat{r}(\delta)\delta)$ and such that $|J|<\delta$. Since $f^n(y)$ lies roughly in the middle of $\tB(c;\hat{r}(\delta)\delta)$, by the Koebe principle (see~\cite[Theorem C]{SV}), there exists a universal constant $C>0$ such that
$$|Df^{n}(y)|\ge 2C\frac{|\tB(c;\hat{r}(\delta)\delta)|}{\delta}\ge \frac{C {\hat{r}(\delta)^{1/\ell_c}}}{D_c(\delta)}.$$
The lemma follows.
\end{proof}

\subsection{Proof of Proposition~\ref{prop:shadowablegrowth}}\label{subsec:shadowable}
The following lemma is the key to the proof of Proposition~\ref{prop:shadowablegrowth}.
\begin{lemm}[Small total distortion] \label{lem:derVSdiscrt}
For each $f\in\LD$ and $\delta>0$, there exists a constant $\theta(\delta)\ge 0$ such that for
any orbit $\{f^j(x)\}_{j=0}^s\in\cL_c(\delta)$  for some $c\in\Crit$  we have
\begin{equation}\label{eqn:derVSdis}
A(x, f, s)\le \theta(\delta) \frac{|Df^s(x)|}{|\tB(c; \delta)|},
\end{equation}
and such that
\begin{equation*}\label{eqn:kdelsmall}
\theta(\delta)\to 0\mbox{ as } \delta\to 0.
\end{equation*}
\end{lemm}
\begin{proof}
Given $\delta>0$, let $\theta(\delta)$ be the minimal non-negative number such that (\ref{eqn:derVSdis}) holds for each orbit in $\cL_c(\delta)$, $c\in\Crit$. Such a number exists because $f$ is uniformly expanding outside $\tB(\delta)$ (Proposition~\ref{prop:mane}).
Again by this proposition, for each $\delta_0>0$, $\theta(\delta)$ is bounded from above for $\delta\ge \delta_0$.
To complete the proof,
it suffices to prove that for $\delta>0$ small enough, we have
\begin{equation}\label{eqn:thetaind}
\theta(\delta/2)\le  \kappa (\theta(\delta) +\rho(\delta)),
\end{equation}
where $$\kappa^2=\sup_{c\in\Crit}\sup_{\delta\in (0,\delta_*]} \frac{|\tB(c;\delta/2)|}{|\tB(c;\delta)|}<1,$$
and $\rho(\delta)\to 0$ as $\delta\to 0$.

To this end, consider an orbit $\{f^j(x)\}_{j=0}^s$ in $\cL_{c}(\delta/2)$. Let $0\le s_1<s_2<\cdots< s_m=s$ be all the integers such that $f^{s_i}(x)\in \tB(\delta)$, let $c_i$ be such that $f^{s_i}(x)\in \tB(c_i,\delta)$ and let $\rho_i=\dist_*(f^{s_i}(x),\Crit)$.
For each $i=1,2,\ldots, m-1$, $\rho_i\ge \delta/2$, so by (\ref{eqn:dist*der}), $|Df(f^{s_i}(x))|\ge D_{c_i}(\rho_i)$.  By Lemma~\ref{lem:derivativeBC}, it follows that
$$|Df^{s_{i+1}-s_i}(f^{s_i}(x))|\ge \frac{r(\delta)D_{c_i}(\rho_i)}{D_{c_{i+1}}(\delta)}
=\frac{r(\delta) \rho_i}{\delta} \frac{|\tB(c_{i+1}; \delta)|}{|\tB(c_i;\rho_i)|}
\ge \frac{r(\delta)}{2}\frac{|\tB(c_{i+1}; \delta)|}{|\tB(c_i;\rho_i)|},$$
which implies
\begin{equation}\label{eqn:momentsi}
\frac{|Df^{s_i}(x)|}{|\tB(c_i;\delta)|}\le \frac{|Df^{s_i}(x)|}{|\tB(c_i;\rho_i)|}\le \frac{2}{r(\delta)} \frac{|Df^{s_{i+1}}(x)|}{|\tB(c_{i+1};\delta)|}.
\end{equation}
Thus
\begin{equation}\label{eqn:momentsi2}
\frac{|Df^{s_i}(x)|}{\dist (f^{s_i}(x), \Crit)}\le C\frac{|Df^{s_i}(x)|}{|\tB(c_i;\delta)|}\le C\rho_1(\delta)^{m-i} \frac{|Df^s(x)|}{|\tB(c; \delta)|},
\end{equation}
where $C$ is a universal constant and $\rho_1(\delta)=2/r(\delta)$.
Since $\{f^j(x)\}_{j=s_i+1}^{s_{i+1}}$ is in $\cL_{c_{i+1}}(\delta)$, we have
$$\sum_{j=s_i+1}^{s_{i+1}-1} \frac{|Df^j(x)|}{\dist (f^j(x), \Crit(f))}\le \theta(\delta)\frac{|Df^{s_{i+1}}(x)|}{\tB(c_{i+1},\delta)}.
$$
Similarly, if $s_1\not=0$, then we have
$$\sum_{j=0}^{s_1-1} \frac{|Df^j(x)|}{\dist (f^j(x), \Crit(f))}\le \theta(\delta)\frac{|Df^{s_1}(x)|}{\tB(c_1;\delta)}.$$
It follows that
$$
A(x, f, s)
\le (1+\theta(\delta)) \sum_{i=1}^{m-1} \frac{|Df^{s_i}(x)|}{\dist(f^{s_i}(x),\Crit)}+\theta(\delta) \frac{|Df^s(x)|}{|\tB(c;\delta)|},
$$
which implies by (\ref{eqn:momentsi2}) that
$$A(x,f,s)\le \left(\theta(\delta)+ \rho(\delta)(1+\theta(\delta)) \right)\frac{|Df^s(x)|}{|\tB(c;\delta)|},$$
where $\rho(\delta)=C\rho_1(\delta)/(1-\rho_1(\delta)$. When $\delta>0$ is small enough, we have $1+\rho(\delta)<\kappa^{-1}$.
Since $|\tB(c;\delta/2)|\le \kappa^2 |\tB(c;\delta)|$, it follows that
$$ A(x,f,s)\le \kappa(\theta(\delta)+\rho(\delta))\frac{|Df^s(x)|}{|\tB(c;\delta/2)|}.$$
This proves (\ref{eqn:thetaind}), completing the proof of the lemma.
\end{proof}

\begin{proof}[Proof of Proposition~\ref{prop:shadowablegrowth}]
Note that there exists $L(\delta)>L$ for each $\delta>0$ small enough such that $L(\delta)\to\infty$, $L(\delta)\delta\to 0$ and $ r(L(\delta)\delta)/ L(\delta)\to\infty$ as $\delta\to 0$, where $r(\cdot)$ is as in Lemma~\ref{lem:derivativeBC}.
Replacing $L(\delta)$ by a smaller function if necessary, we may assume
\begin{equation}\label{eqn:deep2large}
\inf\{|Df^{n+1}(v)|: v\in\CV, \dist_*(f^s(v),\Crit)\le L(\delta)\delta\}\ge L(\delta)^{1-\zeta}.
\end{equation}
Consider $v\in\CV$ and $\delta>0$ small. Let $N=N_v(\delta)$ be the maximal positive integer such that $A(v,f, N)\le {\theta}/{\delta}$. Let us first prove that
\begin{equation}\label{eqn:noretnv}
f^j(v)\not\in \tB(L(\delta)\delta)\mbox{ for each }j=0,1,\ldots, N-1,
\end{equation}
 provided that $\delta$ is small enough.
Indeed, otherwise, there exists a minimal $t\in \{1,\ldots, N-1\}$ such that
$f^t(v)\in \tB(c; L(\delta)\delta)$ for some $c\in\Crit$. By Lemma~\ref{lem:derivativeBC},
\begin{equation}\label{eqn:derat1st}
|Df^t(v)|\ge \frac{r(L(\delta)\delta)}{D_c(L(\delta)\delta)},
\end{equation}
hence
$$A(v, f, N)\ge \frac{|Df^t(v)|}{\dist (f^t(v),\Crit)}\ge \frac{r(L(\delta)\delta)}{|\tB(c;L(\delta)\delta)|D_c(L(\delta)\delta)}
=\frac{r(L(\delta)\delta)}{L(\delta)\delta} > \frac{\theta}{\delta},$$
provided that $\delta>0$ is small enough. This is a contradiction.

In the following we shall define an integer $M=M_v(\delta)\in \{1,2,\ldots, N\}$ for $\delta>0$ small enough, such that (\ref{eqn:dermv}) holds. This will complete the proof of the proposition.
Indeed, (\ref{eqn:dermv}) implies (\ref{eqn:mvdelta}): when $\delta>0$ is small, if $\delta'\ge \sqrt{\delta}$ then $|Df^{M_v(\delta)+1}(v)|$ is large, so $M_v(\delta)$ is large, and if $\delta'\le \sqrt{\delta}$ then $M_v(\delta)$ is also large since the forward orbit of $v$ is disjoint from $\Crit$.

If $f^N(v)\in \tB(L(\delta)\delta)$, then we take $M(v;\delta)=N$.  Since $\delta'\le L\delta$, the inequality (\ref{eqn:dermv}) follows from (\ref{eqn:deep2large}).
In the following we assume that $f^{N}(v)\not\in \tB(L(\delta)\delta)$.
For each $k=0,1,\ldots$, let $\delta_k=2^{-k}\delta_*$ and let $V_k=\tB(\delta_k)$.
Let $k_1$ be the minimal non-negative integer such that
$f^j(v)\not\in V_{k_1}$ for each $0\le j\le N$.  Then $2\delta_{k_1}> L(\delta)\delta$. It follows that
$$\sum_{0\le k< k_1} \left(\frac{\delta}{\delta_k}\right)^{\zeta/2}
\le \left(\frac{\delta}{2\delta_{k_1}}\right)^{\zeta/2}\frac{1}{1-2^{-\zeta/2}}\le \frac{L(\delta)^{-\zeta/2}}{1-2^{-\zeta/2}}\le \frac{\theta}{2},
$$
provided that $\delta>0$ is small enough.
Let $s_{k_1}=-1$ and for each $0\le k< k_1$, define
$$s_k=\sup\{0\le j\le N: f^j(v)\in V_k\}.$$
Then, either
\begin{equation}\label{eqn:out}
\sum_{n=s_0+1}^{N}\frac{|Df^n(v)|}{\dist (f^{n}(v), \Crit)}\ge \frac{\theta}{ 2\delta};
\end{equation}
or
\begin{equation}\label{eqn:in}
\sum_{n=s_{k+1}+1}^{s_{k}} \frac{Df^n(v)}{\dist (f^{n}(v),\Crit)} \ge  \frac{1}{\delta}\left(\frac{\delta}{\delta_k}\right)^{\zeta/2}
\mbox{ for some }0\le k <k_1.
\end{equation}

If (\ref{eqn:out}) holds, then we define $M=N$. Since $f^n(v)\not\in\tB(\delta_*)$ for each $s_0<n\le N$, by Proposition~\ref{prop:mane},
there exists a constant $C_1>0$ such that
$$|Df^{N+1}(v)|\ge C_1 \sum_{n=s_0+1}^N \frac{|Df^n(v)|}{\dist(f^n(v),\Crit)}\ge
\frac{C_1\theta}{2\delta}\ge \left(\frac{\delta_*}{\delta}\right)^{1-\zeta},$$ provided that $\delta>0$ is small enough.
Thus (\ref{eqn:dermv}) holds  in this case.

If (\ref{eqn:in}) holds then we  define $M=s_k$.
By definition of $s_k$ there exists $c_k\in\Crit$ such that $f^{s_k}(v)\in\tB(c_k; \delta_k)$ and $\delta'\le \delta_k$.
By Lemma~\ref{lem:derVSdiscrt}, there exists a constant $C_2>0$ such that
$$\frac{|Df^{s_k}(v)|}{|\tB(c_k;\delta_k)|}\ge C_2
\sum_{n=s_{k+1}+1}^{s_{k}} \frac{|Df^n(v)|}{\dist (f^{n}(v),\Crit)}\ge \frac{C_2}{\delta}\left(\frac{\delta}{\delta_k}\right)^{\zeta/2},
$$
where the second inequality follows from (\ref{eqn:in}).
Since $s_k>s_{k+1}$, $f^{s_k}(v)\not\in \tB(c_k;\delta_k/2)$, hence
$|Df(f^{s_k}(v))||\tB(c_k;\delta_k)|\ge C_3 \delta_k,$
where $C_3>0$ is constant.
Therefore
$$|Df^{s_k+1}(v)|=|Df^{s_k}(v)||Df(f^{s_k}(v))|\ge C_2C_3\frac{\delta_k}{\delta}\left(\frac{\delta}{\delta_k}\right)^{\zeta/2}. $$
Since $\delta_k/\delta\ge 2\delta_{k_1}/\delta\ge L(\delta)$, this implies (\ref{eqn:dermv}) provided that $\delta$ is small enough.
\end{proof}

\subsection{Proof of Proposition~\ref{prop:quantifiedmane}}\label{subsec:1landing}
In this section, we study the derivative of first landing map to critical neighborhoods for the map $f$ and prove Proposition~\ref{prop:quantifiedmane}.
\begin{proof}[Proof of Proposition~\ref{prop:quantifiedmane}]
Let $\delta_0>0$ be a small constant such that Lemma~\ref{lem:derivativeBC} applies for all $\delta\in (0,\delta_0)$ with $r(\delta)>2$.  By Proposition~\ref{prop:mane} we only need to prove the proposition in the case that $\delta\in (0,\delta_0/2)$.

Consider an orbit $\{f^j(x)\}_{j=0}^n$ in $\cL_c(\delta)$. Since $f^n(x)\in \tB(2\delta)$, there exists
a sequence of non-negative integers $n_1<n_2<\cdots<n_m=n$ with the following properties:
\begin{itemize}
\item $n_1$ is the minimal non-negative integer such that $f^{n_1}(x)\in \tB(\delta_0)$;
\item for each $i=1,2,\ldots, m-1$, $n_{i+1}$ is the minimal integer with $n_{i+1}>n_i$ and
$$\dist_*(f^{n_{i+1}}(x),\Crit)\le \min\left(\delta_0, 2\dist_*(f^{n_i}(x),\Crit)\right).$$
\end{itemize}
For each $i=1,2,\ldots, m$, let $c_i\in\Crit$ be such that $f^{n_i}(x)\in \tB(c_i; \delta_0)$ and let
$\rho_i=\dist_* (f^{n_i}(x), \Crit)$. So $c_m=c$ and $\rho_1, \rho_2, \ldots, \rho_{m-1}\in [\delta, \delta_0)$.

For each $i=1,2,\ldots, m-1$, applying Lemma~\ref{lem:derivativeBC} to the orbit $\{f^j(x)\}_{j=n_i+1}^{n_{i+1}}$ with $\delta=\rho_i$, we obtain
\begin{align*}
|Df^{n_{i+1}-n_i}(f^{n_i}(x))|& =|Df(f^{n_i}(x))||Df^{n_{i+1}-n_i-1} (f^{n_i+1}(x))|\\
& \ge  D_{c_i}(\rho_i) \frac{r(\rho_i)}{D_{c_{i+1}}(\rho_i)}
\ge 2 \frac{D_{c_i}(\rho_i)}{D_{c_{i+1}}(\rho_i)}.
\end{align*}
Setting $\rho_0=\delta''$, we also have
$|Df^{n_1}(x)|\ge \frac{2\kappa_0} {D_{c_1}(\rho_0)},$
where $\kappa_0>0$ is a constant.
Indeed, if $\delta''<\delta_0$ then the inequality follows from Lemma~\ref{lem:derivativeBC}; otherwise, it follows from M\~an\'e's theorem (Proposition~\ref{prop:mane}).
Thus
\begin{multline*}
|Df^n(x)| \ge \frac{2\kappa_0}{D_{c_1}(\rho_0)}\prod_{i=1}^{m-1}\frac{2D_{c_i}(\rho_i)}{D_{c_{i+1}}(\rho_i)}
 =\frac{2\kappa_0}{D_{c_m}(\rho_{m-1})}
\prod_{i=1}^{m-1}\left( 2 \frac{D_{c_i}(\rho_i)}{D_{c_i}(\rho_{i-1})}\right).
\end{multline*}
Reducing $\delta_0>0$ if necessary,
this implies that
\begin{multline*}
|Df^n(x)|\ge \frac{2\kappa_0}{D_{c_m}(\rho_{m-1})}
\prod_{i=1}^{m-1}\left(\frac{\rho_i}{\rho_{i-1}}\right)^{1-\ell_{c_i}^{-1}}
 \ge \frac{2\kappa_0}{D_{c_m}(\rho_{m-1})}
\left(\frac{\rho_{m-1}}{\delta''}\right)^{1-\ell_{\max}^{-1}},
\end{multline*}
hence
$$|Df^n(x)| \ge \frac{ \kappa_0 }{D_{c_m}(\delta)}\left(\frac{\delta}{\delta''}\right)^{1-\ell_{\max}^{-1}},$$
since $\rho_{m-1}\ge \delta$.
\end{proof}

\section{Growth of derivatives along pseudo-orbits}\label{sec:pseudoorb}
The main goal of this section is to prove Theorem~\ref{theo:dergwrothrandom}. So we will be working on $f\in\SC_1$ and $\Omega=\sF_1$.
In~\S\ref{subsec:returntocn}, we apply the binding argument to deduce a part of the first statement of the theorem from the results obtained in \S~\ref{sec:derterministic}. We shall decompose a random orbit into pieces, each of which is shadowed by either the true orbit of a critical value or a true orbit corresponding to a first landing into a critical neighborhood. In~\S\ref{subsec:exprate}, we complete the proof of Theorem~\ref{theo:dergwrothrandom} by combing this result in~\S\ref{subsec:returntocn} with M\~an\'e's theorem.
In \S\ref{subsec:moreprop}, we collect a few properties for return maps to the critical neighborhood $\tB(\eps)$ under $\eps$-random maps. These are deduced from Theorem~\ref{theo:dergwrothrandom}, and reveal that many of the properties of the deterministic dynamics remain under random perturbations.

\subsection{Return to critical neighborhoods}\label{subsec:returntocn}
In this section, we shall prove the following proposition which asserts that part (i) of Theorem~\ref{theo:dergwrothrandom} holds with $\eps^{\alpha(\eps)}$ replaced by $0$.

\begin{prop}\label{prop:dergrowthps}
Consider $f\in\SC_1$ and $\Omega=\sF_1$. For each $\eps>0$ small, there exists a constant $\hLambda(\eps)>0$ such that $\lim_{\eps\to 0} \hLambda(\eps)=\infty$ and such that for each
$\bfg\in\Omega_\eps$, $x\in [0,1]$ with $\dist(x,\CV)\le 4\eps$ and an integer $s\ge 1$, if $\bfg^j(x)\not\in \tB(\eps)$ for $1\le j<s$ and $\bfg^s(x)\in \tB(c;2\eps)$ for some $c\in\Crit$, then
$$|D\bfg^s(x)|\ge \hLambda(\eps)/D_c(\eps).$$
\end{prop}

To prove this proposition, we shall first define a binding period for each $v\in\CV$ and each $\delta>0$ small as follows.

Let $C_0=\max_{[0,1]}|Df|\ge 1$, let $\eta_*\in (0,1)$ be a constant which is smaller than the distance between any two distinct critical points and let
$$W_0=\max_{v\in\CV} \sum_{n=0}^\infty |Df^n(v)|^{-1}.$$
Let $\theta>0$ be a small constant such that
\begin{equation}\label{eqn:theta}
4\theta W_0 \le \theta_1 \,\, \mbox{ and }\,\, 16 e \theta W_0 C_0 \le \eta_*,
\end{equation}
where $\theta_1>0$ is as in Lemma~\ref{lem:shadow}. Moreover, fix constants $L>2^{\ell_{\max}}$ and $\zeta\in (0,\ell_{\max}^{-1})$. For $v\in\CV$ and $\delta>0$ small, we fix a positive integer $M_v(\delta)$, called {\em the preferred binding period} for $(v,\delta)$, such that the conclusion of Proposition~\ref{prop:shadowablegrowth} holds for these constants $\theta$, $L$ and $\zeta$.
Since $M_v(\delta)\to \infty$ as $\delta\to 0$ for each $v\in\CV$, we have
\begin{equation}\label{eqn:dfnLambda}
\Lambda_0(\delta):=\inf_{v\in\CV} |Df^{M_v(\delta)+1}(v)|\to \infty \mbox{ as }\delta\to 0.
\end{equation}

\begin{prop}\label{prop:psorbretpre}
Given $f\in\SC_1$ and $\Omega=\sF_1$, there exist positive constants $\zeta_1, \zeta_2$ with the following property.
For $\delta>0$ sufficiently small, and $v\in\CV$, let
$M=M_v(\delta)\ge 1$ be the preferred binding period defined as above. Then
for any $\bfg\in\Omega_\delta^\N$ and $y\in [0,1]$ with $\dist(y, v)\le 4 \delta$, we have
\begin{align}
\label{eqn:yjposition}
y_j :=\bfg^j(y_0)& \not\in \tB(2\delta)\mbox{ for all }0\le j<M;\\
\label{eqn:derlastMp}
|D\bfg^{M}(y_0)|& \ge \frac{\Lambda_0(\delta)^{\zeta_1}}{D_c(\delta)},
\end{align}
where  $c$ is the critical point of $f$ closest to $y_M:=\bfg^M(y)$.
Moreover, if $y_M\not\in \tB(\delta)$ then
\begin{equation}
\label{eqn:derlastM}
|D\bfg^{M+1}(y_0)| \ge \Lambda_0(\delta)^{\zeta_2}\left(\frac{\dist_*(y_{M}, \Crit)}{\delta}\right)^{1-\ell_{\max}^{-1}},
\end{equation}
\end{prop}

\begin{proof}
Fix $v\in\Crit$ and $\delta>0$ small and write $M=M_v(\delta)$. Let $y$ and $\bfg$ be as in the proposition and let $c$ be the critical point of $f$ closest to $y_M$. By Lemma~\ref{lem:shadow}, (\ref{eqn:smalla}) implies that $M$ is an $eW_0$-binding period for $(v,4\delta)$.
By (\ref{eqn:bindcrit}) and (\ref{eqn:noretmv}), the statement (\ref{eqn:yjposition}) holds provided that $\delta>0$ is small enough. By (\ref{eqn:bindderdis}), we have
\begin{multline*}
\eta_M: =|y_M-f^M(v)|\le 4e\delta W_0 |Df^{M}(v)|\le 4e\delta W_0 C_0 |Df^{M-1}(v)|\\
\le 4e\delta W_0 C_0 A(x,f, M)\le 4e W_0 C_0 \theta \le  \eta_*/4,
\end{multline*}
so $\dist(f^M(v), c)\le 3 \dist (f^M(v),\Crit)$. Thus there exists  $C_1>1$ such that
\begin{equation}\label{eqn:dffMv}
|Df(f^M(v))| \le C_1 D_c(\delta'),
\end{equation}
where $\delta'=\max(\dist_*(f^{M}(v),\Crit),\delta)$.
Let $\zeta_1= (\ell_{\max}^{-1}-\zeta)/(2-2\zeta)$. By (\ref{eqn:dermv}) and the definition of $\Lambda_0(\delta)$ we obtain
\begin{equation}\label{eqn:dfM+1}
|Df^{M+1}(v)|\ge \Lambda_0(\delta)^{2\zeta_1} \left(\frac{\delta'}{\delta}\right)^{1-\ell_{\max}^{-1}}.
\end{equation}

Let us prove the inequality (\ref{eqn:derlastMp}).  By (\ref{eqn:bindder}), (\ref{eqn:dfM+1})  and (\ref{eqn:dffMv}), we have
$$|D\bfg^M(y_0)|\ge \frac{|Df^{M+1}(v)|}{e|Df(f^M(v))|}\ge  \frac{\Lambda_0(\delta)^{2\zeta_1}}{C_1e D_c(\delta')} \left(\frac{\delta'}{\delta}\right)^{1-\ell_{\max}^{-1}}\ge  \frac{C_2\Lambda_0(\delta)^{2\zeta_1}}{D_c(\delta)},$$
where $C_2>0$ is a constant and we used $\delta'\ge \delta$ for the last inequality. Provided that $\delta$ is small enough, $C_2\Lambda_0(\delta)^{\zeta_1}>1$, so (\ref{eqn:derlastMp}) holds.

Finally, let us assume $\delta'':=\dist_*(y_M, \Crit)\ge \delta$ and prove that (\ref{eqn:derlastM}) holds with $\zeta_2=\zeta_1/\ell_{\max}$ provided that $\delta$ is small enough. By (\ref{eqn:dist*der}), we have
\begin{equation}\label{eqn:dfyM}
 |Dg_M(y_M)| \ge C_3 D_c(\delta''),
\end{equation}
where $C_3>0$ is a constant.
We now distinguish two cases.

{\em Case 1. }
$\delta''> \Lambda_0(\delta)^{2\zeta_1} \delta'\ge \Lambda_0(\delta)^{2\zeta_1} \delta$. Provided that $\delta>0$ is small enough, $\Lambda_0(\delta)$ is large, so $f^M(v)$ is much closer to $c$ than $y_M$. Thus there exists $C_4>0$ such that
$\eta_M\ge C_4 |\tB(c;\delta'')|,$
which implies by (\ref{eqn:dfyM})
$$\eta_M |Dg_M(y_M)|\ge C_3C_4 |\tB(c; \delta'')|D_{c}(\delta'')=C_3C_4 \delta''.$$
By (\ref{eqn:bindder}) and (\ref{eqn:bindderdis}),
$$|D\bfg^{M+1}(y_0)|\ge \frac{|Df^M(v)|}{e}|Dg_M(y_M)|\ge \frac{\eta_M|Dg_M(y_M)|}{4e^2 W\delta } \ge C_5 \frac{\delta''}{\delta}\ge C_5 \Lambda_0(\delta)^{2\zeta_2}\left(\frac{\delta''}{\delta}\right)^{1-\ell_{\max}^{-1}},
$$
where $C_5>0$ is a constant. The inequality (\ref{eqn:derlastM}) follows provided that $\delta$ is small enough.

{\em Case 2. } $\delta''\le \Lambda_0(\delta)^{2\zeta_1}\delta'$.
In this case, combining (\ref{eqn:bindder}), (\ref{eqn:dfM+1}), (\ref{eqn:dffMv}) and (\ref{eqn:dfyM}), we obtain
\begin{multline*}
|D \bfg^{M+1} (y_0)|\ge \frac{|Df^{M+1}(v)|}{e}\frac{|D g_M(y_M)|}{|D f(f^M(v))|}\ge C_6\Lambda_0(\delta)^{2\zeta_1} \left(\frac{\delta'}{\delta}\right)^{1-\ell_{\max}^{-1}}\left(\frac{\delta''}{\delta'}\right)^{1-\ell_c^{-1}}\\
= C_6\Lambda_0(\delta) ^{2\zeta_1} \left(\frac{\delta''}{\delta}\right)^{1-\ell_{\max}^{-1}}\left(\frac{\delta'}{\delta''}\right)^{\ell_c^{-1}-\ell_{\max}^{-1}}\ge C_6 \Lambda_0(\delta)^{2\zeta_1(1-\ell_c^{-1}+\ell_{\max}^{-1})} \left(\frac{\delta''}{\delta}\right)^{1-\ell_{\max}^{-1}},
\end{multline*}
where $C_6>0$ is a constant. The inequality (\ref{eqn:derlastM}) follows provided that $\delta$ is small enough.
\end{proof}

Let $\cO^\eps(\delta)$ denote the collection of  $\eps$-random orbits $\{x_j\}_{j=0}^n$ for which $x_j\not\in \tB(\delta)$ for each $0\le j<n$, and for each $c\in\Crit$, let $\cL_c^\eps(\delta)$ denote the collection of $\eps$-random-orbits $\{x_j\}_{j=0}^n\in\cO^\eps(\delta)$ for which $x_n\in \tB(c;\delta)$.

\begin{lemm}\label{lem:pseudoorbitland}
Consider $f\in\LD$ and $\Omega=\sF_1$. For each $\delta>0$, there exist $\eps=\eps(\delta)>0$, $C(\delta)>0$ and $\widehat{\eta}=\widehat{\eta}(\delta)>0$ such that for any $\eps$-random orbit $\{\bfg^j(x)\}_{j=0}^n\in \cL_c^\eps(\delta)$, $c\in\Crit$, we have
\begin{equation}\label{eqn:psland}
|D\bfg^n(x)|\ge \frac{\kappa}{D_c(\delta)} \left(\frac{\delta}{\delta''}\right)^{1-\ell^{-1}_{\max}}\exp (\widehat{\eta}n),
\end{equation}
where $\delta''=\max(\dist(x, \CV),\delta)$ and $\kappa>0$ is a constant independent of $\delta$.
\end{lemm}
\begin{proof} Fix $\delta>0$.
By Proposition~\ref{prop:mane1} (i), there exists $C>0$ and $\eta>0$ such that if $\eps>0$ is small enough, then for any $\eps$-random orbit $\{\bfg^j(x)\}\in \cO^\eps(\delta)$ we have $|D\bfg^n(x)|\ge C e^{\eta n}$. If $C e^{\eta n}\ge 1/D_c(\delta)$, then the desired estimate holds with $\kappa=1$ and $\widehat{\eta}=\eta$. So assume the contrary. Then $n$ is bounded from above by a constant $N(\delta)$. When $\eps$ is small enough, we have
$\{f^j(x)\}_{j=0}^n\in \cL_c(0.9\delta)$ and $|D\bfg^n(x)|\ge |Df^n(x)|/2$.
By Proposition~\ref{prop:quantifiedmane}, there is a constant $\kappa_0>0$ such that
$$|D\bfg^n(x)|\ge \frac{\kappa_0}{2 D_c(\delta)}\left(\frac{\delta}{\delta''}\right)^{1-\ell^{-1}_{\max}}.$$
Taking $\eta'>0$ such that $\exp \left(\eta' N(\delta)\right)<2$, we obtain the inequality (\ref{eqn:psland}) with $\widehat{\eta}=\eta'$ and $\kappa=\kappa_0/4$.
\end{proof}

Let $\cI^\eps_c(\delta, \hdelta)$ denote the collection of $\eps$-random-orbits $\{x_j\}_{j=0}^n$ for which there exists $v\in \CV$ such that $\dist(x_0, v)\le 4\delta$ and such that one of the following holds:
\begin{itemize}
\item {\em either} $x_{M_v(\delta)}\in \tB(c,\hdelta)$ and $n= M_v(\delta)$;
\item {\em or }
 $x_{M_v(\delta)}\not\in \tB(\hdelta)$, $n>M_v(\delta)$ and $\{x_j\}_{j=M_v(\delta)+1}^n\in \cL^\eps_c(\hdelta)$.
\end{itemize}
In the language of \cite{BC}, $n$ is the first {\em free return } of the random orbit $\{x_j\}_{j=0}^n$ into $\tB(\hdelta)$.

\begin{lemm}\label{lem:iorbit}
There exists a constant $\zeta_3>0$ such that the following holds. For each $\delta_0>0$ small enough, there exists $\eps_0>0$ such that for  each $\{\bfg^j(x)\}_{j=0}^n\in \cI^\eps_c(\delta,\delta_0)$ with $\delta\in (0,\delta_0]$ and $0\le \eps\le \min (\eps_0,\delta)$,
we have
\begin{equation}\label{eqn:Dbfgn}
|D\bfg^n(x)|\ge \frac{\Lambda_0(\delta)^{\zeta_3}}{D_c(\delta)}.
\end{equation}
Moreover, if $x_n\not\in \tB(\delta)$ then
\begin{equation}\label{eqn:Dbfgn+1}
|D\bfg^{n+1}(x)|\ge \Lambda_0(\delta)^{\zeta_3}\left(\frac{\dist_*(x_n, \Crit)}{\delta}\right)^{1-\ell_{\max}^{-1}}.
\end{equation}
\end{lemm}
\begin{proof}
In the following, we assume that $\delta_0>0$ is a small constant such that for all $\delta\in (0,\delta_0]$ the conclusion of Proposition~\ref{prop:psorbretpre} holds, and let $\eps_0=\eps(\delta_0)$ be the constant determined by Lemma~\ref{lem:pseudoorbitland}. Let $\zeta_3=\min(\zeta_1,\zeta_2)/2$.

Assume $\delta\in (0,\delta_0]$ and $0\le \eps\le \min (\eps_0,\delta)$, and consider $\{x_j\}_{j=0}^n=\{\bfg^j(x)\}_{j=0}^n\in \cI^\eps_c(\delta,\delta_0)$. Let $v\in\CV$ be such that $\dist(x_0,v)\le 4\delta$ and let $M=M_v(\delta)$.
If $x_M\in\tB(\delta_0)$ then $n=M$ and the desired estimates hold, by Proposition~\ref{prop:psorbretpre}.
Assume $x_M\not\in \tB(\delta_0)$.
Let $\delta':=\dist_*(x_M, \Crit)\ge \delta_0$ and let $\delta''=\dist(x_{M+1}, \CV)$. Then $\delta''\le \delta'+\eps\le 2\delta'$.
By Lemma~\ref{lem:pseudoorbitland}, we obtain
\begin{equation*}
\prod_{j=M+1}^{n-1} |Dg_j(x_j)|\ge \frac{\kappa} {D_c(\delta_0)} \left(\frac{\delta_0}{\delta''}\right)^{1-\ell_{\max}^{-1}}\ge \frac{\kappa} {2 D_c(\delta_0)} \left(\frac{\delta_0}{\delta'}\right)^{1-\ell_{\max}^{-1}}.
\end{equation*}
Combining with the estimates given by (\ref{eqn:derlastM}) in Proposition~\ref{prop:psorbretpre}, this implies
\begin{equation}\label{eqn:Dgn}
|D\bfg^n(x)|\ge  \frac{\kappa \Lambda_0(\delta)^{\zeta_2}}{2 D_c(\delta_0)} \left(\frac{\delta_0}{\delta}\right)^{1-\ell_{\max}^{-1}}.
\end{equation}
Since $\delta_0\ge \delta$, it follows that there exists a constant $C_1>0$ such that
$$|D\bfg^n(x)|\ge C_1 \frac{\Lambda_0(\delta)^{\zeta_2}}{D_c(\delta)}.$$
Provided that $\delta_0$ is small enough, $\Lambda_0(\delta)$ is large, so (\ref{eqn:Dbfgn})
follows.
To prove the second inequality, assume $\rho:=\dist_*(x_n, \Crit)\ge \delta$. Then  $|Dg_n(x_n)|\ge D_c(\rho)$. Since $\rho<\delta_0$, by (\ref{eqn:Dgn}),   there exists a constant $C_2>0$ such that
$$|D\bfg^{n+1}(x)|\ge C_2 \Lambda_0(\delta)^{\zeta_2} \left(\frac{\rho}{\delta}\right)^{1-\ell_{\max}^{-1}}.$$
The inequality (\ref{eqn:Dbfgn+1}) follows provided that $\delta_0$ is small enough.
\end{proof}

\begin{proof}[Proof of Proposition~\ref{prop:dergrowthps}]
Let $\delta_0>0$ be a small constant such that $\Lambda_0(\delta)>1$ for all $\delta\in (0,\delta_0]$. Reducing $\delta_0>0$ if necessary, we may assume that there exists $\eps_0>0$ such that the conclusion of Lemma~\ref{lem:iorbit} holds.
Consider $0\le \eps\le  \min(\eps_0,\delta_0/2)$. Let $\{x_j\}_{j=0}^s=\{\bfg^j(x)\}_{j=0}^s$ be an $\eps$-random-orbit with $\dist(x_0, v)\le 4\eps$ for some $v\in \CV$, $x_j
\not\in\tB(\eps)$ for each $1\le j<s$ and $x_s\in \tB(c;2\eps)$ for some $c\in \Crit$.
We shall prove that $|D\bfg^s(x)|\ge \Lambda_0(\eps)^{\zeta_3} /D_c(\eps)$.

Let $s_1$ be the minimal integer such that $s_1\ge M_v(\eps)$ and $x_{s_1}\in \tB(\delta_0)$, and let $c_1\in \Crit$ be such that $x_{s_1}\in \tB(c_1;\delta_0)$.
Then $\{x_j\}_{j=0}^{s_1}\in\cI_{c_1}^\eps (\eps, \delta_0)$.
If $s_1=s$ then
the desired estimate follows from (\ref{eqn:Dbfgn}).

Assume $s_1<s$. Then $\delta_1=\dist_*(x_{s_1},\Crit)\ge 2\eps$. By (\ref{eqn:Dbfgn+1}), we have
\begin{equation}\label{eqn:derats1}
|D\bfg^{s_1+1}(x)|\ge \Lambda_0(\eps)^{\zeta_3} \left(\frac{\delta_1}{\eps}\right)^{1-\ell_{\max}^{-1}}.
\end{equation}
Let $v_1=f(c_1)$ and define $s_2$ to be the minimal integer such that $s_2\ge M_{v_1}(\delta_1)$ and  $x_{s_2}\in \tB(\delta_0)$. If $s_2=s$ then we stop. Otherwise, we define $c_2, v_2, \delta_2$ and $s_3$ similarly. The procedure continues until we get $s_k=s$.
Then for each $i=1,2,\ldots, k-1$, $\{x_j\}_{j=s_i+1}^{s_{i+1}}\in \cI_{c_{i+1}}^\eps (\delta_i, \delta_0)$.
By (\ref{eqn:Dbfgn+1}) again, we obtain
$$\prod_{j=s_i+1}^{s_{i+1}}|Dg_j(x_j)|\ge \Lambda_0(\delta_i)^{\zeta_3} \left(\frac{\delta_{i+1}}{\delta_i}\right)^{1-\ell_{\max}^{-1}}$$
for all $i=1,2,\ldots, k-2$ and by (\ref{eqn:Dbfgn}),
$$\prod_{j=s_{k-1}+1}^{s-1} |Dg_j(x_j)|\ge \frac{\Lambda_0(\delta_{k-1})^{\zeta_3}}{D_c(\eps)}.$$
Combining these inequalities, we obtain
\begin{multline*}
|D\bfg^s(x)|=|D\bfg^{s_1+1}(x)|\left(\prod_{i=1}^{k-2}\prod_{j=s_i+1}^{s_{i+1}}|Dg_j(x_j)|\right)
\prod_{j=s_{k-1}+1}^{s-1}|Dg_j(x_j)|\\
\ge \frac{\Lambda_0(\eps)^{\zeta}}{D_c(\eps)}\prod_{j=1}^{k-1}\Lambda_0(\delta_j)^{\zeta_3} \left(\frac{\delta_{k-1}}{\eps}\right)^{1-\ell_{\max}^{-1}}\ge \frac{\Lambda_0(\eps)^{\zeta_3}}{D_c(\eps)}.
\end{multline*}
Thus the desired estimate holds.
\end{proof}

\subsection{Exponential rate of expansion}\label{subsec:exprate}
We shall complete the proof of Theorem~\ref{theo:dergwrothrandom}.
Let $\eps_0$ be a small constant such that Proposition~\ref{prop:dergrowthps} holds for all $\eps\in (0,\eps_0]$ with $\hLambda(\eps)>2e$, and let $\eps_1=\eps(\eps_0)$ and $\widehat{\eta}=\widehat{\eta}(\eps_0)$ be constants determined by Lemma~\ref{lem:pseudoorbitland} for $\delta=\eps_0$. Replacing $\eps_1$ by a smaller constant if necessary, we assume $\eps_1<\eps_0$.

Let $\cR^\eps_c(\delta)$ denote the collection of $\eps$-random-orbits $\{x_j\}_{j=0}^s$ for which $\dist (x_0, \CV)\le 4\delta$, $x_j\not\in \tB(\delta)$ for $1\le j<s$ and $x_s\in \tB(c;2\delta)$.
Let $\eta_0(\delta)$ be the maximal number in $[0,1]$ such that for any $\{x_j\}_{j=0}^s=\{\bfg^j(x)\}_{j=0}^s\in \cR^\eps_c(\delta)$ with $0\le \eps\le \min (\eps_1, \delta)$, we have
\begin{equation}\label{eqn:dfneta}
|D\bfg^s(x)|\ge \frac{e}{D_c(\delta)} \exp (\eta_0(\delta) s).
\end{equation}
For an orbit $\{\bfg^j(x)\}_{j=0}^s \in \cR^\eps_c(\eps_0)$ with $0\le \eps\le \eps_1$, $|D\bfg^s(x)|$ is exponentially large in $s$. Combining with Proposition~\ref{prop:dergrowthps},  this gives us
\begin{equation}\label{eqn:eta0eps0}
\eta_0(\eps_0)>0.
\end{equation}
Let $\kappa: (0,\eps_0]\to (0,1)$ be a continuous function such that
\begin{itemize}
\item $\kappa(\eps)\to 0$ as $\eps\to 0$,
\item $\Lambda(\eps):=\hLambda(\eps)^{\kappa(\eps)}e^{1-\kappa(\eps)}\ge 2e$ for all $\eps\in (0,\eps_0]$ and
\item $\hLambda(\eps)^{\kappa(\eps)}\to\infty$ as $\eps\to 0$,
\end{itemize}
and let
\begin{equation}\label{eqn:teta}
\widetilde{\eta}_0(\delta):=(1-\kappa(\delta))\eta_0(\delta).
\end{equation}
Combining the estimate given by Proposition~\ref{prop:dergrowthps} with (\ref{eqn:dfneta}), we obtain that
\begin{equation}\label{eqn:etadelta1}
|D\bfg^s(x)|\ge \frac{\Lambda(\delta)}{D_c(\delta)} \exp \left(\widetilde{\eta}_0(\delta) s\right),
\end{equation}
holds for each $\{\bfg^j(x)\}_{j=0}^s\in \cR^\eps_c(\delta)$.

\begin{lemm} \label{lem:eta}
For each $\delta\in (0,\eps_0]$ and $\delta'\in [\delta/2, \delta)$, we have
$\eta_0(\delta')\ge \widetilde{\eta}_0(\delta)$.
\end{lemm}

\begin{proof}
Given a random orbit $\{\bfg^j(x)\}_{j=0}^s\in \cR^\eps_c(\delta')$ with $0\le \eps\le \min (\eps_1, \delta')$, let us prove
\begin{equation}\label{eqn:etavary}
|D\bfg^s(x)|\ge \frac{e}{D_c(\delta')}\exp (\widetilde{\eta}_0(\delta) s).
\end{equation}
Let $s_1<s_2< \cdots< s_k=s$ be all the positive integers such that $x_{s_i}\in \tB(2\delta)$. Then for each $i=0,1,\ldots, k-1$, $\{x_j\}_{j=s_i+1}^{s_{i+1}}\in \cR^\eps_{c_{i+1}}(\delta)$, where $s_0=-1$ and $c_i$ is the critical point of $f$ which is closest to $x_{s_{i}}$. Therefore, by (\ref{eqn:etadelta1}), for each $0\le i<k$, we have
\begin{equation}\label{eqn:ai}
D_i:= \prod_{j=s_i+1}^{s_{i+1}-1}|Dg_j(x_j)|\ge \frac{2e}{D_{c_{i+1}}(\delta)}\exp (\widetilde{\eta}_0(\delta) (s_{i+1}-s_i-1)).
\end{equation}
Thus
\begin{multline*}
|D \bfg^s(x)|=\prod_{i=0}^{k-1}D_{i} \prod_{i=1}^{k-1} |D g_{s_i}(x_{s_i})|  \ge \frac{(2e)^k}
{D_{c_k}(\delta)} \exp (\widetilde{\eta}_0(\delta) (s-k))\prod_{i=1}^{k-1} \frac{|D g_{s_i}(x_{s_i})|}{D_{c_i}(\delta)}.
\end{multline*}
Since $D_{c_k}(\delta)\le  2 D_{c_k}(\delta')$, $|Dg_{s_i}(x_{s_i})|\ge D_{c_i}(\delta')\ge D_{c_i}(\delta)/2$ and $\widetilde{\eta}_0(\delta)\le 1,$ this implies
(\ref{eqn:etavary}).
\end{proof}
\begin{proof} [Proof of Theorem~\ref{theo:dergwrothrandom}, part (i)]
The lemma above implies that $\log_\eps \eta_0(\eps)\to 0$ as $\eps\to 0$. Hence $\alpha(\eps):=\log_\eps \widetilde{\eta}_0(\eps)\to 0$ as $\eps\to 0$. By (\ref{eqn:etadelta1}), the statement holds.
\end{proof}
\begin{proof}[Proof of Theorem~\ref{theo:dergwrothrandom}, Part (ii).]
Take $\eps_0$, $\eps_1$ and $\widehat{\eta}$ as above. Let $\eta$ be the constant given by Proposition~\ref{prop:mane1} for $U=\tB(\eps_0)$. For each $\delta\in (0,\eps_0]$, let $\eta_0(\delta)$ be as above and let
$$\eta(\delta)=\min \left(\hat{\eta}, \eta, \inf\{\widetilde{\eta}_0(\delta'): \delta'\in [\delta, \eps_0]\}\right).$$
Then $\alpha(\delta):= \log\eta(\delta)/\log\delta\to 0$ as $\delta\to 0$.

Now let $\eps\in (0,\eps_1]$ and consider an $\eps$-random-orbit $\{x_j\}_{j=0}^s=\{\bfg^j(x)\}_{j=0}^s$ with $x_j\not\in \tB(\eps)$ for all $0\le j<s$.
For each $j=0,1,\ldots, s$, let $c_j$ be a critical point of $f$ closest to $x_j$ and let $\rho_j=\dist_*(x_j,\Crit)$.  By Proposition~\ref{prop:mane1} (i), the desired estimate holds if
$\rho_j\ge \eps_0$ for all $j=0,1,\ldots, s-1$. Without loss of generality, we may assume that $\rho_0<\eps_0$ and $\rho_{s-1}<\eps_0$.

If there exists $s'<s-1$ such that $\rho_{s'}<\rho_{s-1}$, then letting $s'$ be the maximal integer with this property, we have
$\{x_j\}_{j=s'+1}^{s-1}\in \cR^\eps_{c_{s-1}} (\rho_{s-1})$, hence by (\ref{eqn:etadelta1}),
$$\prod_{j=s'+1}^{s-1} |Dg_j(x_j)|\ge \frac{|Dg_{s-1}(x_{s-1})|}{D_{c_{s-1}}(\rho_{s-1})}\Lambda(\rho_{s-1}) e^{\eta(\eps) (s-s'-1)}> 2 e^{\eta(\eps) (s-s')}.$$
It follows that we only need to prove the desired estimate under the further assumption that $\rho_{s-1}\le \rho_j$ for each $0\le j<s$.
In this case, let $s_0<s_1<\cdots<s_k=s-1$ be a sequence of integers such that  $s_0=0$ and such that for each $0\le i<k$,  $s_{i+1}$ is the minimal integer such that $\rho_{s_{i+1}}\le \rho_{s_i}$.
Then
for $0\le i<k$, $\{x_j\}_{j=s_{i}+1}^{s_{i+1}}\in\cR_{c_{s_{i+1}}}^\eps (\rho_{s_{i}})$, so
by (\ref{eqn:dfneta}), we have
\begin{equation*}
\prod_{j=s_i+1}^{s_{i+1}-1}|D g_j(x_j)|\ge \frac{2 e^{\eta(\eps)(s_{i+1}-s_i)}}{D_{c_{s_{i+1}}}(\rho_{s_i})},
\end{equation*}
which implies that
\begin{align*}
|D\bfg^{s}(x)|& \ge e^{\eta(\eps) (s-1)} |Dg_0(x_0)|\prod_{i=1}^{k}\frac{2|Dg_{s_i}(x_{s_i})|}{D_{c_{s_i}}(\rho_{s_{i-1}})}\\
& \ge e^{\eta(\eps) (s-1)} D_{c_0}(\rho_0)\prod_{i=1}^k \frac{D_{c_i}(\rho_{s_i})}{D_{c_i}(\rho_{s_{i-1}})}
\ge A e^{\eta(\eps) s}\rho_{s-1}^{1-\ell_{\max}^{-1}},
\end{align*}
where $A>0$ is a constant. Since $\rho_{s-1}\ge \eps$, the inequality (\ref{eqn:theo22}) follows.
\end{proof}

\subsection{More properties of return maps to $\tB(\eps)$}\label{subsec:moreprop}
The following proposition is an analogue of Lemma~\ref{lem:derVSdiscrt} for iterates of random maps. It provides distortion control of first landing maps of $\eps$-random maps taken from an admissible space into $\tB(\eps)$.
\begin{prop}\label{prop:psorblocdis}
Consider $f\in\SC_1$ and $\Omega=\sF_1$. For each $\eps>0$ small there exists $\theta(\eps)\ge 0$ such that $\lim_{\eps\to 0}\theta(\eps)=0$ and such that the following holds: For $x\in [0,1]$ and $\bfg\in \Omega_\eps^\N$, if $n\ge 1$ is an integer such that $\bfg^j(x)\not\in \tB(\eps)$, $j=0,1,\ldots, n-1$, and $\bfg^n(x)\in \tB(c; \eps)$ for some $c\in\Crit$, then
$$A(x, \bfg, n) |\tB(c;\eps)|\le \theta(\eps) |D\bfg^n(x)|.$$
\end{prop}
\begin{proof}
This can be proved in the same way as Lemma~\ref{lem:derVSdiscrt}. Indeed, by part (ii) of Theorem~\ref{theo:dergwrothrandom}, for each $\eps>0$ small, there exists a minimal number $\theta(\eps)\ge 0$ such that for any $x\in [0,1]$, $\bfg\in \Omega_\eps^\N$, if $\bfg^j(x)\not\in\tB(\eps)$ for $0\le j<s$ and $\bfg^s(x)\in\tB(c;\eps)$, then $A(x,\bfg, s) |\tB(c;\eps)|\le \theta(\eps) |D\bfg^n(x)|$. Replacing Lemma~\ref{lem:derivativeBC} by part (i) of Theorem~\ref{theo:dergwrothrandom}
and arguing as in the proof of Lemma~\ref{lem:derVSdiscrt}, we show that $\theta(\eps/2)\le \kappa (\theta(\eps)+\rho(\eps))$ for some constant $\kappa\in (0,1)$ and $\rho(\eps)\to 0$ as $\eps\to 0$. Thus $\theta(\eps)\to 0$.
\end{proof}

\begin{prop}\label{prop:psorbbc}
Consider $f\in\SC_1$ and $\Omega=\sF_1$. Given any $0<\xi<\xi'\le 2$, the following holds for each $\eps>0$ small:
for any $\bfg\in\Omega_\eps^\N$ and any integer $s\ge 1$, if
$W$ is an interval intersecting $\tB(\xi\eps)$ and $\bfg^s(W)\subset \tB(2\eps)$, then
$W\subset \tB(\xi'\eps)$.
\end{prop}
\begin{proof}
We first prove the proposition assuming that
\begin{equation}\label{eqn:psorbbcsp}
\bfg^j(W)\cap\tB(\eps)=\emptyset\mbox{ for all }1\le j<s.
\end{equation}
Our strategy is to show that $W':=W\cap (\tB(\xi'\eps)\setminus \tB(\xi\eps))$ is {\em compactly} contained in $\tB(\xi'\eps)$, which clearly implies that $W\subset \tB(\xi'\eps)$.
Let $c, c_0\in\Crit$ be such that $W'\subset \tB(c_0, 2\eps)$ and $\bfg^s(W)\subset \tB(c;2\eps)$.
Provided that $\eps>0$ is small enough, $|Dg_0(x)|\ge C D_{c_0}(\eps)$ for each $x\in \tB(c_0; 2\eps)\setminus \tB(c_0;\xi\eps)$, where $C>0$ is a constant depending only on $\xi$.
By part (i) of Theorem~\ref{theo:dergwrothrandom}, it follows that
$$|D\bfg^s(x)|=|D(\sigma\bfg)^{s-1}(g_0(x))||Dg_0(x)|\ge \frac{\Lambda(\eps)}{D_c(\eps)} CD_{c_0}(\eps)$$
for each $x\in W'$. Thus
$$|W'|\le \frac{|\tB(c;2\eps)|}{\inf_{x\in W'} |D\bfg^s(x)|}\le \frac{2|\tB(c_0;\eps)|}{C\Lambda(\eps)}.$$
Provided that $\eps>0$ is small enough, this implies that $|W'|$ is much smaller than $|\tB(c_0;\eps)|$. Since $W'$ intersects the boundary of $\tB(c_0,\xi\eps)$, it follows that $W'$ is compactly contained in $\tB(\xi'\eps)$. This proves the proposition under the assumption (\ref{eqn:psorbbcsp}).

Now assume that (\ref{eqn:psorbbcsp}) does not hold. Let $0=s_0<s_1<\cdots<s_k<s$ be all the integers such that $g^{s_i}(W)\cap \tB(\eps)\not=\emptyset$. Then replacing $\xi'$ and $\xi$ by $2$ and $1$ respectively in the argument above, we obtain that $\bfg^{s_k}(W)\subset \tB(2\eps)$. Repeating the procedure we obtain $\bfg^{s_{k-1}}(W)\subset \tB(2\eps)$, $\bfg^{s_{k-2}}(W)\subset \tB(2\eps)$, $\ldots$, $\bfg^{s_1}(W)\subset \tB(2\eps)$, and finally $W\subset \tB(\xi'\eps).$
\end{proof}

The following proposition provides us nice sets. The proof is very similar to the deterministic case provided in~\cite{BRSS} which followed the original argument of Rivera-Letelier~\cite{RL} for complex rational maps.
\begin{prop} \label{prop:nice}
Consider $f\in\SC_1$ and $\Omega=\sF_1$. If $0<\eps\le \delta$ are small enough, then there exists a nice set $V$ for $\eps$-random perturbations such that
for $\bfg\in \Omega_\eps^\N$, we have
$\tB(\delta)\subset V^\bfg\subset \tB(2\delta).$
\end{prop}

\begin{proof} Assume that $0<\eps\le \delta$ are small.
By Proposition~\ref{prop:psorbbc}, for any $\bfg\in\Omega_\eps\subset \Omega_\delta$, if $J$ is an interval intersecting $\tB(\delta)$ and $\bfg^n(J)\subset \tB(2\delta)$ for some integer $n\ge 1$ then $J\subset \tB(2\delta)$.
For $c\in\Crit$, $\bfg\in \Omega_\eps^\N$ and $n\ge 0$, let
$V_c^\bfg(n)$ be the component of $\bigcup_{i=0}^n \bfg^{-i}(\tB(\delta))$ that contains $c$. Let $V_c^\bfg:=\bigcup_{n=0}^\infty V_c^\bfg(n)$.
It is easy to check that $V=\bigcup_{c\in\Crit} \bigcup_{\bfg\in\Omega_\eps^N} V_c^\bfg\times \{\bfg\}$ is a nice set for $\eps$-random perturbations. It remains to show that for each $n\ge 0$, we have
$$\tB(c;\delta)\subset V_c^\bfg (n) \subset \tB(c;2\delta), \mbox{ for each } c\in\Crit\mbox{ and }\bfg\in \Omega_\eps^\N.$$

To this end, we proceed with induction on $n$. The case $n=0$ is trivial. Assume that the statement holds for some integer $n\ge 0$. Fix $c, \bfg$. To show that $V_c^\bfg(n+1)\subset \tB(c;2\delta)$, it suffices to show that each component $J$ of $V_c^\bfg(n+1)\setminus \tB(c;\delta)$ is contained in $\tB(2\delta)$.  To this end, let $m\in \{0,1,\ldots, n\}$ be minimal such that $\bfg^{m+1}(J)\cap \tB(\delta)\not=\emptyset$. Then we have
$\bfg^{m+1}(J)\subset \bigcup_{c'\in \Crit} V_{c'}^{\sigma^{m+1}\bfg} (n-m)$. By induction hypothesis, this implies that $\bfg^{m+1}(J)\subset \tB(2\delta)$, hence $J\subset \tB(2\delta)$.  This completes the induction step and hence the proof of the proposition.
\end{proof}

\section{Structure of proof of Theorem~\ref{theo:reduced}}\label{sec:thm2}
This section and the rest of this paper are devoted to the proof of Theorem~\ref{theo:reduced}. Unless otherwise stated, $f\in\SC_1$, $\Crit=\Crit(f)\subset (0,1)$, $\Omega\ni f$ is an admissible space, and for each $\eps>0$ small, $\nu_\eps$ is a probability measure on $\Omega_\eps$ which belongs to the class $\sM_\eps(L)$, where $L>1$ is a fixed constant. Moreover, write $P_\eps=\Leb|_{[0,1]}\times \nu_\eps^\N$.

Let $\theta_0>0$ be a small constant determined by Lemma~\ref{lem:theta0}.
For each $x\in [0,1]$, $\bfg\in \Omega^\N$ and $n\ge 1$, let
\begin{equation}\label{eqn:defnJ}
\hJ_{x, n}^\bfg =\left[x-\frac{\theta_0}{A(x,\bfg, n)}, x+ \frac{\theta_0}{A(x,\bfg, n)}\right],\mbox{ and } J_{x,n}^\bfg= \hJ_{x,n}^\bfg \cap [0,1].
\end{equation}
Then $\bfg^n$ maps $J_{x,n}^\bfg$ diffeomorphically onto its image and $\cN(\bfg^s|J_{x,n}^\bfg)\le 1.$
Note that for $0<\eps\le \delta$ small enough, if $x\in\tB(\delta)$ and $\bfg\in\Omega_\eps^\N$, then $\hJ_{x,n}^\bfg=J_{x,n}^\bfg \subset (0,1)$.

We say that an integer $s\ge 1$ is a {\em $\theta$-good return time of  $(x,\bfg)$ into $\tB(\delta)\times\Omega^\N$} if there exists $c\in\Crit$ such that
$\bfg^s(x)\in \tB(c;\delta)$ and such that
\begin{equation}\label{eqn:excellentrt}
\theta |D\bfg^s(x)|\ge A(x,\bfg;s)|\tB(c;\delta)|.
\end{equation}
So if $\theta\le \theta_0/e$, $x\in \tB(\delta)$ and $\bfg\in\Omega_\eps^\N$ for $0<\eps\le \delta$ small enough, then $\bfg^s(J_{x,s}^\bfg)$ contains $\tB(c;\delta)$.
We say that a positive integer $s$ is a {\em $\tau$-scale expansion time of $(x,\bfg)$} if $$\theta_0 |D\bfg^n(x)|\ge e \tau A(x,\bfg, s).$$

We shall use the following notations:
\begin{equation}
h_\delta^\theta(x,\bfg)  =\inf\left\{s\ge 1:  s \mbox{ is a $\theta$-good return time of } (x,\bfg)\mbox{ into } \tB(\delta)\times\Omega^\N\right\},
\end{equation}
\begin{equation}
T_\tau(x,\bfg) =\inf\left\{s\ge 1:  s \mbox{ is a $\tau$-scale time of }(x,\bfg)\right\},
\end{equation}
\begin{equation}
\hh_{\delta,\tau}^\theta (x,\bfg) =\min \left(\inf_{\delta'\ge \delta} h_{\delta'}^\theta(x,\bfg), T_\tau(x,\bfg)\right),
\end{equation}
and
\begin{equation}\label{eqn:dfnldelta}
l_{\delta}(x,\bfg)=\inf\left\{s\ge 0: \bfg^s(x)\in \tB(\delta)\right\}.
\end{equation}
The following is an easy consequence of Proposition~\ref{prop:psorblocdis}:
\begin{lemm}\label{lem:thetadelta}
Given $\theta>0$ there exists $\delta_0>0$ such that for $x\in [0,1]\setminus \tB(\delta_0)$ and $\bfg\in \Omega_{\eps}$ with $\eps\in (0,\delta_0]$, we have $h_{\delta_0}^\theta (x,\bfg)=l_{\delta_0}(x,\bfg).$
\end{lemm}
\begin{proof} By definition, $h_{\delta_0}^\theta (x,\bfg)\ge l_{\delta_0}(x,\bfg)$.
By Proposition~\ref{prop:psorblocdis}, $l_{\delta_0}(x,\bfg)$, if finite, is a $\theta$-good return time of $(x,\bfg)$ into $\tB(\delta_0)\times \Omega^\N$ provided that $\delta_0$ is small enough. Thus $h_{\delta_0}^\theta (x,\bfg)\le l_{\delta_0}(x,\bfg)$.  The lemma follows.
\end{proof}

\begin{thm2} Let $f, \Omega, \nu_\eps$ be as in Theorem~\ref{theo:reduced} and let $\theta>0$ and $p\ge 1$ be constants. Then for each $\delta_0>0$ small there exist $\eps_0>0$ and $C>0$ such that for each $\eps\in (0,\eps_0]$ the following holds:
\begin{equation}
\int\int_{\tB(\delta_0)\times \Omega_\eps^\N} (h_{\delta_0}^\theta (x,\bfg))^p dP_\eps\le C.
\end{equation}
\end{thm2}

Let us deduce Theorem~\ref{theo:reduced} from Theorem 2'.

\begin{proof}[Proof of Theorem~\ref{theo:reduced}]
Fix $p>1$ and $\theta=\theta_0/4$. Let $\delta_0>0$ be small such that the conclusion of Theorem 2' holds.
Reducing $\delta_0>0$ if necessary, by Lemma~\ref{lem:thetadelta}, $l_{\delta_0}(x,\bfg)= h_{\delta_0}^\theta (x,\bfg)$ holds for all $x\in [0,1]\setminus \tB(\delta_0)$ and $\bfg\in\Omega_{\delta_0}^\N$.
Thus by Proposition~\ref{prop:mane1},
$$\int\int_{([0,1]\setminus \tB(\delta_0))\times \Omega_\eps^\N} (h_{\delta_0}^\theta(x,\bfg))^p dP_\eps= \int\int_{([0,1]\setminus \tB(\delta_0))\times \Omega_\eps^\N} (l_{\delta_0}(x,\bfg))^p dP_\eps$$
is bounded from above by a constant, provided that $\eps>0$ is small enough.
Together with Theorem 2', it follows that there exists a constant $C>0$ such that
$$\int_{[0,1]\times \Omega_\eps^\N} (h_{\delta_0}^\theta (x,\bfg))^p dP_\eps\le C$$
holds when $\eps>0$ is small enough.

Reducing $\delta_0>0$ if necessary, by Proposition~\ref{prop:nice}, when $0<\eps\le \delta_0$, there exists a nice set $V$ for $\eps$-perturbations  such that $\tB(\delta_0)\subset V^\bfg\subset \tB(2\delta_0)$ for each $\bfg\in\Omega_\eps^\N$. By Lemma~\ref{lem:theta0}, for each $(x,\bfg)\in V$, $h_{\delta_0}^\theta(x,\bfg)$ is a Markov inducing time, so $m_V(x,\bfg)\le h_{\delta_0}^\theta(x,\bfg)$. Therefore $$\int_V (m_V(x,\bfg))^p dP_\eps\le
\int_{[0,1]\times \Omega_\eps^\N} (h_{\delta_0}^\theta (x,\bfg))^p dP_\eps\le C.$$
Theorem~\ref{theo:reduced} follows.
\end{proof}

Let us outline the proof of Theorem 2'. By analyzing recurrence of $\eps$-random orbits into the critical region $\tB(\eps)$, we shall first prove the following propositions in \S\ref{sec:slowrec}.
\begin{prop} \label{prop:indsmall}
Given $\theta>0$, $p>1$ and $\gamma>0$, there exists $\tau>0$ such that for any $c\in \Crit$, we have
\begin{equation}\label{eqn:speps}
\frac{1}{|\tB(c;\eps)|}\int\int_{\tB(c;\eps)\times \Omega_\eps^\N} \left(\hh_{\eps,\tau}^\theta(x,\bfg)\right)^p dP_\eps  < \eps^{-\gamma},
\end{equation}
provided that $\eps>0$ is small enough.
\end{prop}
\begin{prop}\label{prop:indsmall(1)}
Given $\theta>0$, $\alpha>0$ and $b>0$, there exists $\tau>0$ such that the following holds provided that $\eps>0$ is small enough:
For each $x\in \tB(\eps)\setminus \tB(b\eps)$ and $\bfg\in\Omega_\eps^\N$, we have $\hh_{\eps,\tau}^\theta(x,\bfg)\le \eps^{-\alpha}$.
\end{prop}

Next,
for $p\ge 1$, $c\in\Crit$ and $0<\eps\le \delta\le\delta_0/e$, write
\begin{align}
\label{eqn:intnot1}
S^\theta_p(\delta,\eps; c,\delta_0)& =\frac{1}{|\tB(c;\delta)|}\int\int_{\tB(c;\delta)\times \Omega_\eps^\N} \left(\inf_{\delta'\in [\delta,\delta_0]} h_{\delta'}^\theta(x,\bfg)\right)^p dP_\eps,\\
\label{eqn:intnot1h}
\hS^\theta_{p}(\delta,\eps;\delta_0)& =\int\int_{(\tB(\delta_0)\setminus \tB(\delta))\times \Omega_\eps^\N}
\frac{1}{\dist(x,\Crit)}\left(\inf_{\delta'\in [e\delta, \delta_0]}h_{\delta'}^\theta(x,\bfg)\right)^p dP_\eps.
\end{align}

We shall prove the following two propositions in \S~\ref{sec:inducing}.
\begin{prop}\label{prop:indsma}
Fix $\theta>0$, $\gamma>0$ and $p\ge 1$. For each $\delta_0>0$ small enough, there exist $\eps_0>0$ and $C>0$
such that the following hold provided that $0<\delta\le \delta_0/e$ and $0<\eps\le \min(\eps_0, \delta)$:
\begin{enumerate}
\item [(i)] $S^\theta_p(\eps,\eps;c,\delta_0)\le C\eps^{-\gamma}$ for each $c\in\Crit$ and
\item [(ii)] $\hS^\theta_{p}(\delta,\eps;\delta_0)\le C\delta^{-\gamma}$.
\end{enumerate}
\end{prop}

\begin{prop} \label{prop:indind}
Fix $p\ge 1$, $\gamma>0$ and $\lambda\in (e^{-\ell_{\max}^{-1}},1)$. There exists $\theta_*>0$ such that for each $\theta\in (0,\theta_*)$ the following holds:  For each $c\in\Crit$,
\begin{equation}\label{eqn:tailind}
S_p^\theta (e\delta,\eps; c,\delta_0)<\lambda \left(
\max_{c'\in\Crit} S_{p}^\theta (\delta,\eps; c',\delta_0) +2\hS^{\theta/e}_p(\delta,\eps;\delta_0)\right),
\end{equation}
provided that $0<\eps\le\delta\le \delta_0/e$ are small enough.
\end{prop}

Now let us assume these propositions and prove Theorem 2'.
\begin{proof}[Proof of Theorem 2'] Take $\lambda\in (e^{-\ell_{\max}^{-1}},1)$ and $\gamma>0$ such that $\lambda_0:=\lambda e^\gamma<1$. Let $p\ge 1$ and $\theta>0$ be given.
We may certainly assume that $\theta\in (0,\theta_*)$. So by Propositions~\ref{prop:indsma} and~\ref{prop:indind}, for each $\delta_0>0$
small, there exist $\eps_0>0$ and $C_1>0$ such that
\begin{align}\label{eqn:Seps}
S(\eps,\eps)& \le C_1\eps^{-\gamma},\\
\label{eqn:hSdelta}
S (e\delta,\eps)& \le \lambda S(\delta,\eps)+ C_1 \lambda \delta^{-\gamma},
\end{align}
for any $0<\delta\le \delta_0/e$ and $0<\eps\le \min(\delta,\eps_0)$, where $$S(\delta,\eps)=\max_{c\in\Crit} S_p^\theta(\delta,\eps;c,\delta_0).$$

Let us prove that $S(\delta_0,\eps)$ is bounded from above by a constant. Let $N$ be the maximal integer such that $e^N\eps\le \delta_0$.
Let $S_k=(e^{-k}\delta_0)^ \gamma S(e^{-k}\delta_0,\eps)$.
Then by (\ref{eqn:hSdelta}), for each $0\le k<N$, $S_k\le \lambda_0 (S_{k+1} +C_1)$.
It follows that
$$\delta_0^{-\gamma} S(\delta_0,\eps)=S_0\le \lambda_0^N S_N + C_2,$$ where $C_2>0$ is a constant.
Clearly $S(e^{-N}\delta_0,\eps)/S(e\eps,\eps)$ is bounded from above, so by (\ref{eqn:Seps}) and (\ref{eqn:hSdelta}), $S_N$ is bounded from above by a constant. Thus $S(\delta_0,\eps)$ is bounded from above by a constant.
\end{proof}

\section{Slow recurrence of $\eps$-random-orbits into $\tB(\eps)$}\label{sec:slowrec}

The goal of this section is to prove Propositions~\ref{prop:indsmall} and~\ref{prop:indsmall(1)}.
To this end, we shall first study the recurrence to $\tB(\eps)$ of $\eps$-random orbits, in \S~\ref{subsec:slowrec}. Proposition~\ref{prop:rettbeps} there means that $\eps$-random orbits entering $\tB(\eps)$ too deep and too often are rare.
In \S~\ref{subsec:goodreturnt} we study the expanding property of $\eps$-random orbits with slow recurrence to $\tB(\eps)$ and show that they allow a certain large scale time, see Proposition~\ref{prop:indsmall0}. The proof of Propositions~\ref{prop:indsmall} and~\ref{prop:indsmall(1)} will be completed  in \S~\ref{subsec:indsmall}.

We shall continue to use the notations $h_\delta^\theta, T_\tau, \hh_{\delta,\tau}^\theta, l_\delta$ introduced in \S~\ref{sec:thm2}.


\subsection{Most random orbits satisfy a slow recurrence condition}\label{subsec:slowrec}
For $x\in [0,1]$ and $\bfg=(g_0,g_1,\ldots)\in \Omega_\eps^\N$,
define
\begin{equation}\label{eqn:defq}
q_\eps(x,\bfg)=q_\eps(x,g_0)=\inf\{q\in \N: |Dg_0(x)|\dist(x,\Crit(g_0))\ge e^{-q}\eps\}.
\end{equation}
Moreover, for non-negative integers $0\le n_1\le n_2$, let
\begin{equation}\label{eqn:defQ}
\Q_{n_1}^{n_2} (x,\bfg;\eps)=\sum_{j=n_1}^{n_2} q_\eps(F^j(x,\bfg))
\end{equation}
and
\begin{equation}\label{eqn:defGAMMA}
\Gamma_{n_1}^{n_2}(x,\bfg;\eps)=\#\{n_1\le j\le n_2: \bfg^j(x)\in \tB(\eps)\}.
\end{equation}
Let $\Bad_m(\kappa,\eps)$ be the collection of $(x,\bfg)\in [0,1]\times \Omega_\eps^\N$ such that for each integer $s\ge 0$ we have
$$\Q_0^s(x,\bfg; \eps)> \min \left(m, \kappa\Gamma_0^s (x,\bfg; \eps)\right), $$
and such that $$\lim_{s\to\infty} \Q_0^s (x,\bfg;\eps)\ge m.$$
Moreover, for $c\in\Crit$, let $\Bad_m^c(\kappa,\eps)=\{(x,\bfg)\in \Bad_m(\kappa,\eps)| x\in\tB(c;\eps)\}$.
The main result of this section is the following:

\begin{prop}\label{prop:rettbeps}
There exist $\kappa>0$, $K>0$ and $\rho>0$ such that if $\eps>0$ is small enough, then for each $c\in\Crit$ and each integer $m\ge 0$, we have
\begin{equation}\label{eqn:badsize}
P_\eps (\Bad_m^c(\kappa,\eps))\le Ke^{ -\rho m} |\tB(c;\eps)|.
\end{equation}
\end{prop}

To prove this proposition, we need a few lemmas.
\begin{lemm}\label{lem:psorbland}
For each $\eps>0$ small the following holds: For any $\bfg\in \Omega_\eps^\N$ and $x\in [0,1]$ with $\dist(x, \CV)\le 4 \eps$, if $n:=l_\eps(x,\bfg) <\infty$ and $J$ is the component of $\bfg^{-n}(\tB(\eps))$ which contains $x$, then $\bfg^n$ maps $J$ diffeomorphically onto its image, $\cN(\bfg^n|J)\le 1$ and $|J|<\eps$.
\end{lemm}
\begin{proof}
Let $\theta=\theta_0/e$. By Lemma~\ref{lem:thetadelta}, provided that $\eps$ is small enough, we have
$h_{\eps}^\theta(x,\bfg)=l_\eps(x,\bfg)= n$. Let $I=J_{x,n}^\bfg$ be defined in (\ref{eqn:defnJ}). Then
$\bfg^{n}|I$ is a diffeomorphism onto its image and
$\cN(\bfg^{n}|I)\le 1$.
Let $c\in\Crit$ be such that $\bfg^n(J)\subset \tB(c;\eps)$.
It follows that for each $y\in\partial I\setminus \{0,1\}$, $|\bfg^n(x)-\bfg^n(y)|\ge |D\bfg^n(x)|\theta_0/(eA(x,\bfg, n))\ge  |\tB(c;\eps)|$. Thus $J\subset I$, so $\cN(\bfg^n|J)\le 1$. Since $\dist(x, \CV)\le 4\eps$, by part (i) of Theorem~\ref{theo:dergwrothrandom},
$|D\bfg^n(x)|> e/D_c(\eps)$, provided that $\eps$ is small enough. Thus $|J|\le e|\tB(c;\eps)|/|D\bfg^n(x)|< \eps$.
\end{proof}

Let $\bfF_\eps$ denote the first entry map into the region $\tB(\eps)\times\Omega_\eps^\N$
under $F$, i.e., $$\bfF_\eps (x,\bfg)= F^{R_\eps(x,\bfg)}(x,\bfg),$$
where $R_\eps(x,\bfg)=l_\eps(x,\bfg)$ if $x\not\in\tB(\eps)$ and $R_\eps(x,\bfg)=l_\eps(F(x,\bfg))$ if $x\in \tB(\eps)$.
Note that $\bfF_\eps$ is defined on a subset of $[0,1]\times \Omega_\eps^\N$.

\begin{lemm} \label{lem:rettbeps0}
There exists $K_1>0$ such that for $\eps>0$ small enough the following holds.
For each $\bfg\in\Omega_\eps^\N$ and $v\in\CV$, putting
$$Y_\eps(\bfg, v; q)=\{y\in B(v, 2\eps): q_\eps(\bfF_\eps(y, \bfg))\ge q\},$$ we have
\begin{equation}\label{eqn:setX}
\left|Y_\eps(\bfg,v; q)\right|\le K_1 \eps e^{-\ell_{\max}^{-1}q},
\end{equation}
\end{lemm}

\begin{proof}
Assume $\eps>0$ small.
We first observe that there exists a constant $K_2>0$ such that for each $\bfh\in \Omega_\eps^\N$, $c\in\Crit$, the set
$Z_c^\bfh(q)=\{z\in \tB(c;\eps): q_\eps(z, \bfh)\ge q\}$ satisfies
\begin{equation}\label{eqn:Zcq}
|Z_c^\bfh(q)|\le K_2 e^{-\ell_{\max}^{-1} q}|\tB(c;\eps)|.
\end{equation}
 Moreover, there exists a constant $Q$ such that $Z_c^\bfh(Q)\subset \tB(c;\eps/e)$.

For $q< Q$, the inequality (\ref{eqn:setX}) clearly holds with a suitable choice of $K_1$. So let us assume $q\ge Q$.
For any $y\in Y_\eps(\bfg,v;q)$, let $n(y)=R_\eps(y,\bfg)$,  let $c(y)\in\Crit$ be such that $\bfg^n(y)\in\tB(c(y);\eps)$, and let $J=J(y)$ be the component of $(\bfg^{n(y)})^{-1}(\tB(c(y);\eps))$ containing $y$. Let us prove that there exists a constant $K_3>0$ such that
\begin{equation}\label{eqn:YespinJy}
|J(y)\cap Y_\eps(\bfg,v;q)|\le K_3 e^{-\ell_{\max}^{-1} q}|J(y)|.
\end{equation}
Indeed, since $\bfg^{n(y)}(y)\in\tB(c;\eps/e)$,
$\bfg^{n}(J)$ contains at least one component  of $\tB(c;\eps)\setminus \tB(c;\eps/e)$, so $|\bfg^{n(y)}(J(y))|/|\tB(c;\eps)|$ is bounded away from zero. By Lemma~\ref{lem:psorbland}, $\cN(\bfg^{n(y)}|J(y))\le 1$. In view of (\ref{eqn:Zcq}),
it suffices to prove that
\begin{equation}\label{eqn:0YepsinJy}
\bfg^{n(y)} (J(y)\cap Y_\eps(\bfg,v;q))\subset Z_c^\bfh(q), \mbox{ where }\bfh= \sigma^{n(y)}\bfg.
\end{equation}
To this end, take $y'\in J(y)\cap Y_\eps(\bfg, v;q)$. We need to prove that $n':=R_\eps(y',\bfg)=n(y)$. Otherwise, we would have
$1\le n'< n(y)$. Since $\bfg^{n'}(y')\in \tB(\eps/e)$, by Proposition~\ref{prop:psorbbc}, we would have $\bfg^{n'}(J(y))\subset \tB(\eps)$,  hence $\bfg^{n'}(y)\in \tB(\eps)$, contradicting the minimality of $n(y)$. This proves (\ref{eqn:0YepsinJy}) and hence (\ref{eqn:YespinJy}).

Since the intervals $J(y)$, $y\in Y_\eps(\bfg, v ;q)$ form a covering of $Y_\eps(\bfg,v;q)$, by Besicovic's covering lemma, we can find a sub-covering with bounded intersection multiplicity.
By Lemma~\ref{lem:psorbland},  $|J(y)|\le \eps$, so $J(y)\subset B(v, 3\eps)$. The inequality (\ref{eqn:setX}) follows.
\end{proof}

\begin{lemm} \label{lem:rettbeps}
There exist $K_0>0$ and $\rho_0>0$ such that for each $\eps>0$ small enough the following holds.
For each $x\in \tB(\eps)$ and $q\ge 0$, putting $$U_\eps(x, q)=\{\bfg\in \Omega_\eps^\N:\, q_\eps(\bfF_\eps (x,\bfg))\ge q\},$$
we have
$$\nu_\eps^\N (U_\eps(x, q))\le K_0 e^{-\rho_0 q}.$$
\end{lemm}

\begin{proof}
For $x\in\tB(c;\eps)$, $c\in\Crit$ and $\bfg\in \Omega_\eps^\N$, putting
$$X_\eps(x,\bfg; q)=\left\{g\in \Omega_\eps: q_\eps(\bfF_\eps (g(x),\bfg))\ge q\right\},$$
then for each $g\in X_\eps(x,\bfg;q)$, we have $g(x)\in Y_\eps(\bfg, f(c); q)$, so by (\ref{eqn:regulartranpro}) and (\ref{eqn:setX}),
$$\nu_\eps(X_\eps(x,\bfg;q))\le p_\eps(Y_{\eps}(\bfg, f(c);q)|x)\le L\left(\frac{|Y_\eps(\bfg, f(c);q|)}{2\eps}\right)^{1/L}\le K_0 e^{-\rho_0 q},$$
where $K_0=L(K_1/2)^{1/L}$ and $\rho_0=(\ell_{\max} L)^{-1}$. It follows that
$$\nu_\eps^\N(U_\eps(x,q))=\int_{\Omega_\eps} \nu_\eps (X_\eps(x,\bfg; q)) d\nu_\eps^\N(\bfg)\le K_0e^{-\rho_0 q}.$$
\end{proof}

For $\textbf{q}=(q_1, q_2, \cdots, q_n)\in \N^n$, let $|\textbf{q}|=\sum_{i=1}^n q_i$. Moreover, for $x\in [0,1]$,
let $U_\eps^n(x; \textbf{q})$ be the set of $\bfg\in \Omega_\eps^\N$ with $q(\bfF_{\eps}^i(x,\bfg))\ge  q_i$ for $i=1,2,\ldots, n$.
\begin{lemm}\label{lem:rettbeps'}
For any $\textbf{q}\in\N^n$, $n\ge 1$, the following holds. For any $x\in \tB(\eps)$, we have
\begin{equation}\label{eqn:unq}
\nu_{\eps}^\N \left(U_\eps^n(x;\textbf{q})\right)\le K_0^n e^{-\rho_0 |\textbf{q}|},
\end{equation}
where $K_0$ and $\rho_0$ are constants as in Lemma~\ref{lem:rettbeps}.
\end{lemm}

\begin{proof} The case $n=1$ is given by Lemma~\ref{lem:rettbeps}. For $n\ge 2$, it suffices to show that for any
$\textbf{q}=(q_1,q_2,\ldots, q_n)\in\N^n$ and $x\in \tB(\eps)$, we have
\begin{equation}\label{eqn:Uepsrec}
\nu_\eps^\N \left(U_\eps^{n}(x;\textbf{q})\right)\le K_0e^{-\rho_0 q_n} \nu_\eps^\N \left(U_\eps^{n-1}(x;\textbf{p})\right),
\end{equation}
where $\textbf{p}=(q_1,q_2,\ldots, q_{n-1})$.

To this end, define
$$W_s=\left\{\textbf{g}\in U_\eps^{n-1}(x;\textbf{p}): \sum_{i=0}^{n-2}R_\eps(\textbf{F}_\eps^{i}(x, \textbf{g}))=s\right\}, s=1,2,\ldots.$$
Then $U_\eps^{n-1}(x;\textbf{p})=\bigcup_{s=1}^\infty W_s$.
Note that for $\textbf{g},\widehat{\textbf{g}}\in \Omega_\eps^\N$ with $g_i=\widehat{g}_i$ for $0\le i<s$, then $\textbf{g}\in W_s$ if and only if $\widehat{\textbf{g}}\in W_s$. In other words, $W_s$ can be written in the form $A_s\times \Omega_\eps^\N$, where $A_s$ is a measurable subset of $\Omega_\eps^s$. So $\nu_\eps^\N(W_s)= \nu_\eps^s(A_s)$.

By Lemma~\ref{lem:rettbeps}, for each $(g_0,g_1,\ldots, g_{s-1})\in A_s$, we have
\begin{equation}
\nu_\eps^\N \left(U_\eps(g_{s-1}\circ \cdots\circ g_0(x);q_n)\right)\le K_0 e^{-\rho_0 q_n}.
\end{equation}
Since
$$W_s\cap U_\eps^n(x;\textbf{q})=\left\{(g_0,g_1,\ldots, g_{s-1}, \textbf{h}): (g_0,g_1,\ldots, g_{s-1})\in A_s, \textbf{h}\in U_\eps (g_{s-1}\circ \cdots\circ g_0(x); q_n)\right\}£¬$$
 by Fubini's theorem, we have
\begin{multline*}
\nu_\eps^\N(W_s\cap U_\eps^n(x;\textbf{q}))=\int_{A_s} \nu_\eps^\N \left(U_\eps(g_{s-1}\circ\cdots \circ g_0(x);q_n)\right) d\nu_\eps^s
\le  K_0 e^{-\rho_0 q_n} \nu_\eps^\N(W_s).
\end{multline*}
It follows that
$$\nu_\eps^n (U_\eps^n (x;\textbf{q}))=\sum_{s=1}^\infty \nu_\eps^\N(W_s\cap U_\eps^n(x;\textbf{q}))\le  K_0 e^{-\rho_0 q_n} \sum_{s=1}^\infty \nu_\eps^\N(W_s)= K_0 e^{-\rho_0 q_n} \nu_\eps^\N (U_\eps^{n-1}(x;\textbf{p})).$$
This proves (\ref{eqn:Uepsrec}), and completes the proof of this lemma.
\end{proof}

\begin{proof}[Proof of Proposition~\ref{prop:rettbeps}]
Let $K_0>0$ and $\rho_0>0$ be as in Lemma~\ref{lem:rettbeps} and let
$\rho=\min(\rho_0/5, (2\ell_{\max})^{-1})$.
By the Stirling's formula,  there exists $\kappa>1$ such that if $m, n$ are positive integers with $m>\kappa n/2$ then the binomial coefficient $\binom{m+n-1}{n-1} \le e^{\rho m}.$ Replacing $\kappa$ by a larger constant if necessary, we may assume $K_0 \le e^{\kappa\rho}$.

Fix $c\in\Crit$ and $m\ge 0$. To estimate the size of $\Bad_m^c(\kappa,\eps)$, let
$$\Delta_0=\{(x,\bfg): x\in \tB(c;\eps), q_\eps(x,\bfg)\ge m/2-\kappa\},$$
and for non-negative integers $m', n$, let
$$\Delta_n^{m'}=\{(x,\bfg)\in \tB(c;\eps)\times \Omega_\eps^\N: \exists s\ge 1 \mbox{ such that }
\Gamma_1^s (x,\bfg;\eps)=n,\ \Q_1^s(x,\bfg;\eps)= m'\}.$$
Put $\mathcal{I}=\{(m',n)\in\mathbb{N}^2: 2m'\ge \max(m,\kappa n)>0\}$. Let us prove that
\begin{equation}\label{eqn:decomposebad}
\Bad_m^c(\kappa, \eps)\subset \Delta_0\cup\left(\bigcup_{(m',n)\in\mathcal{I}} \Delta_{m'}^n\right).
\end{equation}

To this end, we first show that for each $(x,\bfg)\in \Bad_m^c(\kappa,\eps)$, there exists $s\ge 0$ such that
\begin{equation}\label{eqn:bad2s}
\Q_0^{s}(x,\bfg;\eps)>\max ( m-\kappa, \kappa \Gamma_0^s(x,\bfg;\eps)).
\end{equation}
Indeed, let $s_0$ be minimal such that $\Q_0^{s_0}(x,\bfg;\eps)\ge m$. If $\Q_0^{s_0}(x,\bfg;\eps) > \kappa \Gamma_0^{s_0}(x,\bfg;\eps)$ then we take $s=s_0$.
Otherwise, take $s<s_0$ to be the maximal integer such that $\bfg^s(x)\in \tB(\eps)$.  Since $\Q_0^s(x,\bfg;\eps)<m$, we have
$Q_0^s(x, \bfg;\eps)> \kappa \Gamma_0^s(x, \bfg;\eps)$. Moreover,
since $\Q_0^{s_0}(x,\bfg;\eps) \le \kappa \Gamma_0^{s_0}(x,\bfg;\eps)$ and $\Gamma_0^{s_0}(x,\bfg;\eps)=\Gamma_0^s(x, \bfg; \eps)+1$,
we have
$$\Q_0^s(x,\bfg;\eps)> \kappa \Gamma_0^{s_0}(x,\bfg; \eps)-\kappa\ge \Q_0^{s_0}(x,\bfg;\eps)-\kappa\ge m-\kappa. $$
Thus the inequality (\ref{eqn:bad2s}) holds.
For each $(x,\bfg)\in \Bad_m^c(\kappa,\eps)\setminus \Delta_0$, putting $m'=\Q_1^s(x,\bfg;\eps)$ and $n=\Gamma_1^s(x,\bfg;\eps)$.
Then $(x,\bfg)\in\Delta_{m'}^n$. So it remains to show that $(m',n)\in\mathcal{I}$. Note that
$$2m'= 2(\Q_0^s(x, \bfg;\eps)-q_\eps(x,\bfg))> 2((m-\kappa) -(m/2-\kappa) )=m.$$
In particular, this implies $n>0$. Moreover, since $2q_\eps(x,\bfg)\le m-\kappa<Q_0^s(x,\bfg)$ we have
$2m'\ge Q_0^s(x,\bfg)> \kappa \Gamma_0^s(x,\bfg;\eps))>\kappa n.$ This proves $(m',n)\in\mathcal{I}$ and completes the proof of (\ref{eqn:decomposebad}).

Now let us estimate $P_\eps(\Delta_0)$. By definition of $q_\eps$, there exists a constant $C=C(\kappa)$ such that for each $\bfg\in\Omega_\eps^n$,
$$|\{x\in\tB(c;\eps): q_\eps (x,\bfg)\ge m/2-\kappa\}|\le C \eps e^{-m/(2\ell_c)} \le C \eps e^{-\rho m}.$$
Thus by Fubini's theorem,
\begin{equation}\label{eqn:Pdelta0}
P_\eps(\Delta_0)\le C\eps e^{-\rho m}.
\end{equation}

Finally, let us estimate $P_\eps(\Delta^{m'}_n)$ for $(m',n)\in\mathcal{I}$. For each $x\in \tB(\eps)$,
let
$$E_n^{m'}(x,\eps)=\left\{\bfg\in \Omega_\eps^\N: (x,\bfg)\in\Delta_n^{m'}\right\}.$$
Then $$E_n^{m'}(x;\eps)\subset \bigcup_{\substack{\textbf{q}\in\N^n\\ |\textbf{q}|=m'}} U_\eps^n(x; \textbf{q}).$$
Since the number of $\textbf{q}\in\N^n$ with $|\textbf{q}|=m'$ is $\binom{m'+n-1}{n-1}$, by Lemma~\ref{lem:rettbeps'}, we obtain
\begin{align*}
\nu_\eps^\N (E_n^{m'}(x,\eps)) \le \binom{m'+n-1}{n-1} e^{-\rho_0 m'}K_0^n.
\end{align*}
By our choice of constants, this gives us
$$\nu_\eps^\N\left(E_n^{m'}(x,\eps)\right)\le e^{\rho m'} e^{-\rho_0 m'} e^{\kappa \rho n}\le e^{-\rho (4m'-\kappa n)}.$$
For each $(m',n)\in\mathcal{I}$, we have $m'>\kappa n/2$, and hence
\begin{equation*}\label{eqn:enm'}
\nu_\eps^\N\left(E_n^{m'}(x,\eps)\right)\le e^{-2\rho m'},
\end{equation*}
which implies by Fubini's theorem that
$$P_\eps(\Delta_n^{m'})=\int_{\tB(c;\eps)} \nu_\eps(E_n^{m'}(x,\eps)) dx\le e^{-2\rho m'}|\tB(c;\eps)|.$$
Thus
\begin{multline*}
\sum_{(m',n)\in\mathcal{I}} P_\eps (\Delta_n^{m'})\le \sum_{m'= [m/2]}^\infty \sum_{n: (m', n)\in\mathcal{I} } P_\eps(\Delta_n^{m'})
\\ \le \sum_{m'=[m/2]}^\infty e^{-2\rho m'} \frac{2m'}{\kappa}|\tB(c;\eps)|=O(e^{-\rho m}|\tB(c;\eps)|).
\end{multline*}
Combining the last inequality with (\ref{eqn:decomposebad}) and (\ref{eqn:Pdelta0}), we obtain the desired estimate (\ref{eqn:badsize}).
\end{proof}

\subsection{Good return time and large scale time}\label{subsec:goodreturnt}
The main result of this section is the following proposition.
\begin{prop}\label{prop:indsmall0}
Given $\theta>0$, $\kappa>0$ and $\alpha>0$, there exists  $\tau>0$ such that the following holds provided that $\eps>0$ is small enough. For $(x,\bfg)\in\tB(\eps)\times\Omega_\eps^\N$ and $m\ge 1$, if $(x,\bfg)\not\in\Bad_m(\kappa,\eps)$,
then $\hh_{\eps,\tau}^\theta(x,\bfg)\le m\eps^{-\alpha}$.
\end{prop}
To prove this proposition, we shall need a few lemmas.
\begin{lemm}\label{lem:shallowuman}
Consider $f\in \SC_1$ and $\Omega=\sF_1$. Given $K>0$ and $\beta>0$, the following hold
for each $(y,\bfh)\in \tB(\delta)\times \Omega_\delta^\N$, provided that $\delta>0$ is small enough.
\begin{enumerate}
\item [(i)] If $t\ge 1$ is an integer such that $\bfh^t(y)\in\tB(c;\delta)$ for some $c\in\Crit$, then
$$\log\left(\frac{|D\bfh^t(y)|\dist(y,\Crit(h_0))}{|\tB(c;\delta)|}\right)\ge K\Gamma_0^{t-1}(y,\bfh;\delta)-\Q_0^{t-1} (y,\bfh;\delta)+\delta^\beta t.$$
\item [(ii)] If $t\ge 1$ is an integer such that $\bfh^t(y)\not\in\tB(\delta)$, then
$$\log \left(\frac{|D\bfh^t(y)|\dist(y,\Crit(h_0))}{d(\bfh^t(y),\Crit(h_t))}\right)\ge \delta^\beta t+2\log \delta +K\Gamma_0^{t-1}(y,\bfh;\delta)-\Q_0^{t-1} (y,\bfh;\delta).$$
\end{enumerate}
\end{lemm}
\begin{proof}
(i) It suffices to consider the case that $\bfh^j(y)\not\in\tB(\delta)$ for each $1\le j<t$, since the general case follows by induction on $\Gamma_0^{t-1}(y,\bfh;\delta)$. Let $\bfg=\sigma \bfh$ and $x=h_0(y)$.
Then by part (i) of Theorem~\ref{theo:dergwrothrandom}, we have
$$|D\bfg^{t-1}(x)|\ge \frac{e^K}{ D_c(\delta)} \exp\left(\delta^\beta t\right),$$
provided that $\delta>0$ is small enough.
Since
\begin{multline*}
\frac{|D\bfh^t(y)|\dist(y,\Crit(h_0))}{|\tB(c;\delta)|}=|D\bfg^{t-1}(x)|\frac{|D h_0(y)|\dist(y,\Crit(h_0))}{|\tB(c;\delta)|}
\\ \ge  |D\bfg^{t-1}(x)| \frac{e^{-q_\delta(y,\bfh)}\delta}{|\tB(c;\delta)|}=|D\bfg^{t-1}(x)| e^{-q_\delta(y,\bfh)} D_c(\delta),
\end{multline*}
it follows that
$$\log \left(\frac{|D\bfh^t(y)|\dist(y,\Crit(h_0))}{|\tB(c;\delta)|} \right)\ge K-q_\delta(y,\bfh)+\delta^\beta t.$$

(ii) Put $\rho_j=\dist_*(\bfh^j(y),\Crit)$ for $0\le j\le t$.
By part (i) of this lemma, it suffices to consider the case that $\rho_j\ge \delta$ for all $1\le j\le t$.
By part (ii) of Theorem~\ref{theo:dergwrothrandom},
$$|D\bfh^t(y)|\dist(y,\Crit(h_0))\ge e^{-q_\delta(y,\bfh)}\delta A \delta^{1-\ell_{\max}^{-1}} \exp (\delta^\beta t)> \delta^2 e^{K-q_\delta(y,\bfh)} \exp (\delta^\beta t),$$
provided that $\delta>0$ is small enough. Since $\dist(\bfh^t(y), \Crit(h_t))\le 1$, it follows that
$$\log \left(\frac{|D\bfh^t(y)|\dist(y,\Crit(h_0))}{d(\bfh^t(y),\Crit(h_t))}\right)\ge \delta^\beta t+2\log \delta + K -q_\delta(y,\bfh),$$
as desired.
\end{proof}

\begin{lemm}\label{lem:shallowret}
Given $\kappa>0$ and $\theta>0$, the following holds provided that $\eps>0$ is small enough. Let $(x,\bfg)\in \tB(\eps)\times\Omega_\eps^\N$, and let $s\ge 1$ be an integer such that $\bfg^s(x)\in \tB(\eps)$ and such that for each $0\le j<s$,
$$\kappa \Gamma_j^{s-1} (x,\bfg; \eps)\ge \Q_j^{s-1} (x,\bfg; \eps).$$
Then $s$ is a $\theta$-good return time of $(x,\bfg)$ into $\tB(\eps)\times\Omega^\N$.
\end{lemm}
\begin{proof}  Let $0=s_0<s_1<s_2<\cdots<s_n=s$ be all the integers such that $\bfg^{s_i}(x)\in\tB(\eps)$ and let $c_i\in\Crit$ be such that $\bfg^{s_i}(x)\in \tB(c_i,\eps)$. For each $0\le i\le n$,
let $A_i=|D\bfg^{s_i}(x)|/d(\bfg^{s_i}(x), \Crit(g_{s_i}))$ and $\tA_i=|D\bfg^{s_i}(x)|/|\tB(c_i;\eps)|$. We need to prove that
$\theta\tA_n\ge  \ A(x,\bfg, s)$.
By Proposition~\ref{prop:psorblocdis}, we have
$$A(x,\bfg,s)\le \sum_{i=0}^{n-1} A_i +\theta(\eps) \sum_{i=1}^n \tA_i\le A_0+ (1+\theta(\eps) )\sum_{i=1}^{n-1} A_i +\theta(\eps) \tA_n,$$
where $\theta(\eps)\to 0$ as $\eps\to 0$.

Let $K_0$ be a large constant and assume $\eps>0$ small. By Lemma~\ref{lem:shallowuman} (i),  for each $i=0,1,\ldots, n-1$,
$$\log \frac{A_i}{\tA_n}\ge (K_0+\kappa) \Gamma_{s_i}^{s-1}(x,\bfg;\eps)-\Q_{s_i}^{s-1}(x,\bfg;\eps, Q)\ge (n-i)K_0,$$
hence $A_i\le e^{-(n-i)K_0} \tA_n$.
Thus $$A(x,\bfg,s) \le \left(e^{-K_0n}+(1+\theta(\eps) ) \sum_{i=1}^{n-1} e^{-K_0(n-i)} +\theta(\eps)\right)\tA_n\le \tA_n /\theta,$$
provided that $\eps>0$ is small enough.
\end{proof}

\begin{lemm}\label{lem:shallowupb}
Given $\theta>0$ and $\gamma>0$, there exists $\tau>0$ such that the following holds provided that $\eps>0$ is small enough. Let $(x,\bfg)\in \tB(\eps)\times\Omega_\eps^\N$ and let $s\ge 1$ be an integer such that for each $0\le j<s$,
\begin{equation}\label{eqn:s0backdpth}
s-j>  \eps^{-\gamma}\Q_j^{s-1} (x,\bfg;\eps).
\end{equation}
Then $\hh_{\eps,\tau}^\theta (x,\bfg) \le s$.
\end{lemm}
\begin{proof}
Let $\beta=\gamma/4$ and $K=1$. Let $\eps_0>0$ be a small constant such that the conclusion of Lemma~\ref{lem:shallowuman} holds for all $\delta\in (0,\eps_0]$.
In the following, we assume that $\eps\in (0, \eps_0/e]$ is small.
Let $N$ be the maximal integer such that $e^{N-1}\eps\le \eps_0$. Then $2\le N\le \log (\eps_0/\eps)+1<\eps^{-\beta}$.

Let $s_0=s$ and for $i=1,2,\ldots, N$, let
$$s_i=\max\{0\le j\le s: \bfg^j(x)\in \tB(e^{N-i}\eps)\}.$$
Define an integer $n\in \{0,1,\ldots, N\}$ as follows: if $s_0-s_{N}<\eps^{-3\beta}$, then $n=N$; otherwise, let $n$ be the minimal integer in $\{0,1,\ldots, N-1\}$ such that $s_n-s_{n+1}\ge \eps^{-2\beta}$.
Note that the minimality of $n$ implies $s_0-s_n< N\eps^{-2\beta}\le \eps^{-3\beta}$.
By (\ref{eqn:s0backdpth}), it follows that for $0\le j<s_n$, we have
\begin{equation}\label{eqn:skret}
s_n-j \ge \eps^{-3\beta} \Q_j^{s_n-1} (x,\bfg; \eps).
\end{equation}
Note that this inequality is clear if the right hand side is zero.

If $n=N$, then by Lemma~\ref{lem:shallowret},
$s_N$ is a $\theta$-good return time of $(x,\bfg)$ into $\tB(\eps)\times\Omega^\N$, so we are done in this case. Assume from now on $n<N$, so that $\bfg^{s_n}(x)\not\in \tB(\eps)$.
For each $0\le j\le s_n$, let $$A_j:= |D\bfg^j(x)|/\dist(\bfg^j(x), \Crit).$$
We need to estimate $A_{s_n}/A(x,\bfg, s_n)$ from below. Let $s_{N+1}>s_{N+2}>\cdots>s_{N+N_0}=0$ be all the integers such that $\bfg^{s_{N+i}}(x)\in\tB(\eps)$. Then, by Lemma~\ref{lem:shallowuman} (ii),
for all $N+N_0\ge k>n$, we have
\begin{equation*}
\log \frac{A_{s_n}}{A_{s_k}}\ge \eps^{\beta} (s_n-s_k) +2\log \eps -\Q_{s_k}^{s_n-1}(x,\bfg;\eps).
\end{equation*}
(To apply the lemma, we take $(y,\bfh)=F^{s_k}(x,\bfg)$ and $\delta=\eps$ in the case $k\ge N$ and $\delta=e^{N-k}\eps$ otherwise.)
Since $s_n-s_k\ge s_n-s_{n+1}\ge \eps^{-2\beta}$, by (\ref{eqn:skret}), this implies that
\begin{equation}\label{eqn:indasnask}
\log \frac{A_{s_n}}{A_{s_k}}\ge \frac{\eps^{\beta}}{2} (s_n-s_k) \ge \frac{\eps^{-\beta}}{2}.
\end{equation}
Thus
$$\sum_{i=n+1}^{N+N_0} A_{s_i}\ll A_{s_n}.$$
By Proposition~\ref{prop:psorblocdis}, this implies,
\begin{equation}\label{eqn:Asn+1}
A(x,\bfg, s_{n+1}+1)\le 2\sum_{i=n+1}^{N+N_0} A_{s_i} \ll A_{s_n}.
\end{equation}

Let us now distinguish two cases to complete the proof.

{\em Case 1.} $n>0$. Then $\bfg^{s_n}(x)\in \tB(\eps_0)$. By Proposition~\ref{prop:psorblocdis} again,
$$\sum_{j=s_{n+1}+1}^{s_n-1} \frac{|D\bfg^j(x)|}{\dist (\bfg^j(x), \Crit)}\ll \frac{|D\bfg^{s_n}(x)|}{\dist(\bfg^{s_n}(x), \Crit)}.$$
Together with (\ref{eqn:Asn+1}), this gives $A(x,\bfg, s_n) \ll A_{s_n}$, so
$s_n$ is a $\theta$-good return time of $(x,\bfg)$ into $\tB(\eps)$.

{\em Case 2.} $n=0$. Then $\bfg^j(x)\not\in \tB(\eps_0/e)$ for $s_1<j<s_0$. By part (ii) of Theorem~\ref{theo:dergwrothrandom} (or Lemma~\ref{lem:pseudoorbitland}), $|D\bfg^{s_n}(x)|/|D\bfg^j(x)|$ is exponentially big in $s_n-j$, hence
$$\sum_{j=s_{n+1}+1}^{s_n-1}\frac{|D\bfg^j(x)|}{\dist (\bfg^j(x), \Crit)} \prec \frac{|D\bfg^{s_n}(x)|}{\dist(\bfg^{s_n}(x), \Crit)}.$$
Together with (\ref{eqn:Asn+1}), this gives $A(x,\bfg, s_n)\asymp A_{s_n}$, which implies that $s_n$ is a $\tau$-scale expansion time of $(x,\bfg)$ for some constant $\tau>0$.
\end{proof}

\begin{proof}[Proof of Proposition~\ref{prop:indsmall0}]
Fix $\beta\in (0,\alpha/4)$.
Let $(x,\bfg)\in \tB(\eps)\times\Omega_\eps^\N$ with $\eps>0$  small.
We first prove that there exists $\tau>0$ such that
\begin{equation}\label{eqn:t1}
\hh_{\eps,\tau}^\theta(x,\bfg)\le T_1:=\inf \{s\ge 1: s> \eps^{-4\beta}\Q_0^{s-1}(x,\bfg;\eps)\}.
\end{equation}
Indeed, if $T_1<\infty$, then
the minimality of $T_1$ implies that for each $0\le j<T_1$,
$$T_1-j \ge \eps^{-4\beta} \Q_j^{T_1-1}(x,\bfg;\eps).$$
By Lemma~\ref{lem:shallowupb}, there exists $\tau>0$ such that the inequality (\ref{eqn:t1}) holds.

Assume now that $(x,\bfg)\not\in \Bad_m(\kappa,\eps)$.
If $T_1\le m\eps^{-\alpha}$ then $\hh_{\eps,\tau}^\theta(x,\bfg) \le T_1\le m\eps^{-\alpha}$, and the proof is completed.
So assume $T_1> m\eps^{-\alpha}$.
Then for $s_0=[m\eps^{-\alpha}]$ we have
$s_0\le \eps^{-4\beta} \Q_0^{s_0-1}(x,\bfg,\eps)$,
which implies that $\Q_0^{s_0-1}(x,\bfg;\eps)> m$. Since $(x,\bfg)\not\in\Bad_m(\kappa,\eps)$,
there exists a minimal non-negative integer $s_1$
such that
$$\Q_0^{s_1}(x,\bfg;\eps)\le \min (m, \kappa \Gamma_0^{s_1}(x,\bfg;\eps)).$$
Since $\Q_0^{s_1}(x,\bfg;\eps)\le m <\Q_0^{s_0}(x,\bfg;\eps)$ we have $s_1<s_0$.
Moreover, the minimality of $s_1$ implies that $\bfg^{s_1}(x)\in \tB(\eps)$
and that for each
$0\le j<s_1$,
$$\Q_j^{s_1}(x,\bfg;\eps)\le \kappa \Gamma_j^{s_1} (x,\bfg;\eps).$$
If $s_1\ge 1$, then by Lemma~\ref{lem:shallowret}, $h_{\eps}^\theta(x,\bfg)\le s_1<s_0\le m\eps^{-\alpha},$
and the proof is completed. Suppose $s_1=0$ and let
$$s_2=\inf\{s> s_1: \bfg^{s}(x)\in \tB(\eps)\}.$$
Again by Lemma~\ref{lem:shallowret}, $h_\eps^\theta(x,\bfg)\le s_2$, so by (\ref{eqn:t1}),
$\hh_{\eps,\tau}^\theta(x,\bfg)\le \min(s_2, T_1).$
By the minimality of $T_1$ once again,  for each $0<s< \min (s_2, T_1)$, we have
$$s\le \eps^{-4\beta} \Q_0^{s-1}(x, \bfg)=\eps^{-4\beta} q_\eps(x,\bfg)\le \eps^{-4\beta} \kappa.$$
Therefore
$$\hh_{\eps, \tau}^\theta(x,\bfg)\le \kappa \eps^{-4\beta} +1<m\eps^{-\alpha}.$$
The proof is completed.
\end{proof}

\subsection{Proof of Propositions~\ref{prop:indsmall} and~\ref{prop:indsmall(1)}}\label{subsec:indsmall}

\begin{proof} [Proof of Proposition~\ref{prop:indsmall}]
Take $\alpha\in (0,\gamma/p)$, and let $\tau>0$ be given by Proposition~\ref{prop:indsmall0}.
Then for $(x,\bfg)\in \tB(c;\eps)\times\Omega_\eps^\N$ with $\eps>0$ small enough,
$$\hh_{\eps,\tau}^\theta (x,\bfg)> m\eps^{-\alpha}\Rightarrow (x,\bfg)\in \Bad_m^c(\kappa,\eps).$$
Thus
\begin{align*}
\int\int_{\tB(c;\eps)\times \Omega_\eps^\N} \left(\hh_{\eps,\tau}^\theta(x,\bfg)\right)^p dP_\eps & \le
\sum_{m=1}^\infty m^{p}\eps^{-\alpha p} P_\eps (\Bad_m^c(\kappa,\eps)).
\end{align*}
By Proposition~\ref{prop:rettbeps}, (\ref{eqn:speps}) follows.
\end{proof}
\begin{proof} [Proof of Proposition~\ref{prop:indsmall(1)}]
Clearly, there exists $\kappa=\kappa(b)$ and $m=m(b)\ge 1$ such that
for $(x,\bfg)\in \tB(\eps)\times \Omega_\eps^\N$ with $x\not\in\tB(b\eps)$,  we have $(x,\bfg)\not\in \Bad_m(\kappa,\eps)$, provided that $\eps>0$ is small.
So the desired estimate follows from Proposition~\ref{prop:indsmall0}.
\end{proof}

\section{Inducing to a large scale}\label{sec:inducing}
In this section, we shall prove Propositions~\ref{prop:indsma} and~\ref{prop:indind}, hence complete the proof of Theorem~\ref{theo:reduced}.
Let $f, \Omega, \nu_\eps, P_\eps$ be as introduced at the beginning of \S~\ref{sec:thm2}.

\subsection{Preparatory lemmas}\label{subsec:prep}
We say that a positive integer $s$ is a {\em $\theta$-close return time} of $(x,\bfg)\in [0,1]\times\Omega^\N$ if $\theta |D\bfg^s(x)|\ge A(x,\bfg, s) \dist (\bfg^s(x), \Crit (g_s))$.
So for $0<\eps\le \delta$ small enough, if $(x,\bfg)\in [0,1]\times \Omega_\eps^\N$ and $s$ is a $\theta$-good return time of $(x,\bfg)$ into $\tB(\delta)\times \Omega^\N$, then $s$ is a $\theta$-close return, since $g_s$ has a critical point in each component of $\tB(\delta)$.  If $s$ is a $\tau$-scale expansion time of $(x,\bfg)\in [0,1]\times \Omega^\N$, then it is a $\theta_0/\tau$-close return time.
\begin{lemm} \label{lem:join1}
Consider $(x,\bfg)\in [0,1]\times \Omega^\N$.
\begin{enumerate}
\item[(i)] Let $0=T_0<T_1<\cdots< T_{n}$ be integers such that for each $0\le i< n$, $T_{i+1}-T_i$ is a $1/2$-close return of $F^{T_i}(x,\bfg)$. Then $T_{n}$ is a $1$-close return time of $(x,\bfg)$.
\item[(ii)] If $t$ is $\theta_1$-close return of $(x,\bfg)$ and
$s$ is a $\theta_2$-good return time of $(y,\bfh):=F^{t}(x,\bfg)$ into $\tB(\delta)\times \Omega^\N$ for some $\delta>0$, then
$t+s$ is a $(1+\theta_1)\theta_2$-good return time of $(x,\bfg)$ into $\tB(\delta)\times \Omega^\N$.
\end{enumerate}
\end{lemm}
\begin{proof}
(i) For $0\le i< n$, let $A_i= A(F^{T_i}(x, \bfg), T_{i+1}-T_i)$ and $\tA_i=|D\bfg^{T_i}(x)|A_i$.
By assumption, for each $i$, we have
$$\frac{1}{2}\frac{|D\bfg^{T_{i+1}}(x)|}{|D\bfg^{T_i}(x)|}\ge A_i \dist (\bfg^{T_{i+1}}(x), \Crit(g_{T_{i+1}})).$$
For $i<n-1$, since $\dist(\bfg^{T_{i+1}}(x),\Crit(g_{T_{i+1}})) A_{i+1}\ge 1$, this implies that $\tA_{i+1}\ge 2\tA_i.$
Thus $$A(x,\bfg, T_{n} )=\sum_{i=0}^{n-1} \tA_i \le 2\tA_{n-1},$$
which implies the statement.

(ii) Since $A(y,\bfh, s)\ge 1/ \dist(\bfg^t(x),\Crit(g_t)) $, we have
$$\theta_1 |D\bfg^t(x)|\ge A(x,\bfg,t)/A(y,\bfh,s).$$
Thus
\begin{multline*}
A(x,\bfg, t+s)=A(x,\bfg,t)+|D\bfg^t(x)|A(y,\bfh,s) \le (1+\theta_1) |D\bfg^t(x)| A(y,\bfh,s)\\
=(1+\theta_1) |D\bfg^{t+s}(x)|\frac{A(y,\bfh,s)}{|D\bfh^s(y)|}\le (1+\theta_1)\theta_2 |D\bfg^{t+s}(x)| |\tB(c;\delta)|,
\end{multline*}
where $c\in\Crit$ is such that $\bfg^{t+s}(x)\in \tB(c;\delta)$. The statement follows.
\end{proof}

We say that a Borel measurable map $\sG: \sE\to [0,1]\times \Omega^\N$ defined on a Borel subset $\sE$ of $[0,1]\times \Omega^\N$ is {\em induced by $F$} if there exists a Borel measurable function $T:\sE\to \Z^+$ such that $\sG(x,\bfg)=F^{T(x,\bfg)}(x,\bfg)$ for each $(x,\bfg)\in\sE$. We say that $\sG$ is {\em future-free} if the following holds: for $(x,\bfg)\in \sE$  and $\bfh\in \Omega^\N$ with $g_i=h_i$ for each $0\le i<T(x,\bfg)$, we have $(x,\bfh)\in \sE$ and $T(x,\bfh)=T(x,\bfg)$.

Given Borel probability measure $\nu$ on $\Omega$, let
\begin{equation}\label{eqn:Jacob}
\cL_{\sG}^\nu(y)= \int_{\Omega^\N} \cL^{\bfg}_{\sG} (y) d\nu^\N(\bfg)
\end{equation}
for each $y\in [0,1]$, where
\begin{equation}\label{eqn:Jacobbfg}
\cL_\sG^{\bfg}(y)=\sum_{\substack{x\in \sE^\bfg\\
\bfg^{T(x,\bfg)}(x)=y}}\frac{1}{|D\bfg^{T(x,\bfg)}(x)|}.
\end{equation}

The following is a simple consequence of the Fubini's Theorem.

\begin{lemm}\label{lem:pushforward}
Let $\sG:\sE\to [0,1]\times \Omega^\N$ be a future-free, Borel measurable induced map with an inducing time function $T$ and let $\phi: \sE\to [0,\infty)$ be a Borel measurable function. Then for any Borel probability measure $\nu$ on $\Omega$, we have
$$\int_{\sE} \phi(\sG(x,\bfg)) dx d\nu^\N(\bfg) = \int_{\Omega^\N}\int_0^1 \cL_{\sG}^\nu(y) \phi(y,\bfh) dy d\nu^\N(\bfh).$$
\end{lemm}
\begin{proof}
Let $X_T=\{(x,\bfg)\in \sE: T(x,\bfg)=T\}$ and let $\cL^\bfg_T(y)=\sum_{x\in X_T^\bfg\cap \bfg^{-T}(y)}|D\bfg^T(x)|^{-1}$.
Then
\begin{multline*}
\int_{X_T} \phi(\sG(x,\bfg)) dx d\nu^\N  =\int_{\Omega^\N} \int_{X_T^\bfg} \phi(\sG(x,\bfg)) dx d\nu^\N
\\ =\int_{\Omega^\N} \int_0^1 \cL^\bfg_T(y) \phi(y, \sigma^T\bfg) dy d\nu^\N
= \int_0^1 \int_{\Omega^\N}\cL^\bfg_T(y) \phi(y, \sigma^T\bfg)  d\nu^\N dy.
\end{multline*}
Since $\sG$ is future free, $\cL^\bfg_T(y)$ depends only on the first $T$ coordinates of $\bfg$. Thus
\begin{multline*}
\int_{\Omega^\N}\cL^\bfg_T(y)\phi(y, \sigma^T\bfg)  d\nu^\N=\int_{\Omega^\N} \cL^\bfg_T(y)d\nu^\N \int_{\Omega^\N} \phi(y, \sigma^T\bfg)d\nu^\N
=\int_{\Omega^\N} \cL^\bfg_T(y)d\nu^\N \int_{\Omega^\N} \phi(y, \bfg)d\nu^\N,
\end{multline*}
and so
$$\int_{X_T}\phi(\sG(x,\bfg)) dx d\nu^\N  =\int_0^1 \left(\int_{\Omega^\N} \cL^\bfg_T(y)d\nu^\N \int_{\Omega^\N} \phi(y, \bfg)d\nu^\N  \right)dy.$$
Since $\sE=\bigcup_{T=1}^\infty X_T$ and $\cL_\sG^\bfg(y)=\sum_{T=1}^\infty \cL^\bfg_T(y)$, the lemma follows.
\end{proof}

\subsection{Proof of Proposition~\ref{prop:indsma}}\label{subsec:indsma}
\begin{lemm}\label{lem:delta0}
Given $\tau>0$ and $\theta>0$ the following holds provided that $0<\eps\le \delta\le \delta_0$ are small enough.
For any $x\in\tB(\delta)$ and $\bfg\in\Omega_\eps^\N$, if $h=\hh_{\delta, \tau}^\theta(x,\bfg)<\infty$ and $l=l_{\delta_0}(F^h(x,\bfg))<\infty$, then
$l+h$ is a $\theta$-good return time of $(x,\bfg)$ into $\tB(\delta')\times \Omega^\N$ for some $\delta'\in [\delta,\delta_0]$.
\end{lemm}
\begin{proof} Let $\theta_1=\max (\theta_0/\tau, 1)$ and let $\theta_2=\theta/(1+\theta_1)$.
Provided that $\delta_0>0$ is small enough, we have
\begin{equation}\label{eqn:taudel0}
|\tB(c;\delta_0)|<\tau\theta/\theta_0\mbox{ for each }c\in\Crit,
\end{equation}
and moreover, by Lemma~\ref{lem:thetadelta},
\begin{equation}\label{eqn:thetadel0}
\mbox{ either } l=0, \mbox{ or }l=h_{\delta_0}^{\theta_2}(F^{h}(x,\bfg)).
\end{equation}

Assume first that $\bfg^h(x)\in\tB(\delta_0)$. Then $l=0$. If $h$ is a $\tau$-scale expansion time, then by (\ref{eqn:taudel0}), it is a $\theta$-good return time of $(x,\bfg)$ into $\tB(\delta_0)\times \Omega^\N$, so we are done. Otherwise, it is a $\theta$-good return time of $(x,\bfg)$ into $\tB(\delta'')\times \Omega^\N$ for some $\delta''\ge \delta$. Let $\delta'=\min (\delta'', \delta_0)$. By definition, it follows that $h$ is a $\theta$-good return time of $(x,\bfg)$ into $\tB(\delta')\times \Omega^\N$.

Assume now that $\bfg^h(x)\not\in \tB(\delta_0)$. Then $l\ge 1$ and so the latter part of (\ref{eqn:thetadel0}) holds. Since $h$ is a $\theta_1$-close return of $(x,\bfg)$, by Lemma~\ref{lem:join1}, we conclude that $l+h$ is a $\theta$-good return time of $(x,\bfg)$ into $\tB(\delta_0)\times \Omega^\N$.
\end{proof}
\begin{proof}[Proof of Proposition~\ref{prop:indsma}]
(i) Fix $\theta>0$, $p\ge 1$ and $\gamma>0$ and let $\tau>0$ be given by Proposition~\ref{prop:indsmall}.
Assume $\delta_0>0$ is small.
Then for each $\eps\in (0, \delta_0]$, we have
\begin{equation}\label{eqn:hhp0}
\frac{1}{|\tB(c;\eps)|}\int_{\tB(c;\eps)\times \Omega_\eps^\N} (\hh_{\eps,\tau}^\theta(x,\bfg))^{p} dP\eps\le \eps^{-\gamma}.
\end{equation}
By Proposition~\ref{prop:mane1}, there exist constants $\eps_0\in (0,\delta_0)$, $C_1>0$ and $\rho_0>0$ such that for each $\bfg\in\Omega_\eps^\N$ with $\eps\in (0,\eps_0]$, we have $|\{y: l_{\delta_0}(y, \bfg)\ge l\}|\le C_1 e^{-\rho_0 n}.$

Now fix $c\in\Crit$ and $\eps\in (0,\eps_0]$.   Write $h(x,\bfg)= \hh_{\eps,\tau}^\theta (x,\bfg)$, $l(x,\bfg)=l_{\delta_0}(x,\bfg)$ and
$H(x,\bfg)=\inf_{\delta'\in [\eps,\delta_0]} h_{\delta'}^\theta(x,\bfg).$    Then $h(x,\bfg)$ and $l(x,\bfg)$ are finite $P_\eps$-almost everywhere. By Lemma~\ref{lem:delta0},
for each $(x,\bfg)\in \tB(\eps)\times \Omega_\eps^\N$, we have
\begin{equation}\label{eqn:joinhtau}
h(x,\bfg) + l(F^{h(x,\bfg)}(x,\bfg))\ge H(x,\bfg),
\end{equation}
provided that $\delta_0$ is small enough.

For each $k=1,2,\ldots$, let
$X_k=\{(x,\bfg)\in \tB(c;\eps)\times \Omega_\eps^\N: h(x,\bfg)=k\}$, let $\sG_k: X_k\to [0,1]\times \Omega^\N$ be the measurable induced  map defined by
$(x,\bfg)\mapsto F^k(x,\bfg)$, and let $\phi_k: [0,1]\times \Omega^\N\to [0,\infty)$ be defined as
$$\phi_k (y,\bfh)=\left\{
\begin{array} {ll}
l(y,\bfh)& \mbox{ if }\infty>l(y,\bfh)>k \mbox{ and }\bfh\in\Omega_\eps^\N;\\
0 & \mbox{ otherwise.}
\end{array}
\right.
$$
Let $L_k:=  \int\int_{[0,1]\times \Omega_\eps^\N} (\phi_k(y,\bfh))^p dP_\eps$. By the choice of $\eps_0$, there exists a constant $C_2>0$ such that
\begin{equation}\label{eqn:sumLk}
\sum_{k=1}^\infty L_k\le C_2.
\end{equation}

\noindent
{\bf Claim. } There exists a constant $C_3>0$ such that for each $y\in [0,1]\setminus \tB(\delta_0)$, each $\bfg\in\Omega_\eps^\N$, and each $k=1,2,\ldots$, we have $\sL_{\sG_k}^{\nu_\eps} (y) \le C_3 |\tB(c;\eps)|.$

Indeed, for each $(x,\bfg)\in X_k$, letting $J_{x,k}^\bfg$ be defined as in (\ref{eqn:defnJ}), we have that $\cN(\bfg^k|J_{x,k}^{\bfg})\le 1$, $|J_{x,k}^{\bfg}|<\theta_0 \dist(x,\Crit(g_0))\le \theta_0 |\tB(c;\eps)|$.
 Given $\bfg\in\Omega_\eps^\N$ and $k=1,2,\ldots$, these intervals $J_{x,k}^{\bfg}$, with $(x,\bfg)\in X_k$ and $\bfg^k(x)=y$, are pairwise disjoint. Moreover, for these $(x,\bfg)$, $|\bfg^k(J_{x,k}^\bfg)|$ is bounded from below by a constant $\tau_1=\tau_1(\tau,\delta_0)>0$.  Therefore
$$\sL_{\sG_k}^\bfg (y)= \sum_{\substack{x\in X_k^\bfg\\ \bfg^k(x)=y}}\frac{1}{|D\bfg^k(x)|}
\le \frac{e}{\tau_1} \sum_{\substack{x\in X_k^\bfg\\ \bfg^k(x)=y}} |J_{x,k}^\bfg|\le e (1+\theta_0)\tau_1^{-1} |\tB(c;\eps)|.$$
The claim is proved.

Since $\sG_k$ is future-free, by Lemma~\ref{lem:pushforward}, we obtain
$$
M_k: =\int_{X_k} (\phi_k(\sG_k(x,\bfg)))^p dP_\eps=\int_{\Omega_\eps^\N}\int_0^1 \sL_{\sG_k}^{\nu_\eps}(y) (\phi_k(y,\bfh))^p dP_\eps
$$
Since $\phi_k(y,\bfh)=0$ for each $y\in \tB(\delta_0)$, and by the claim above, we obtain
$$M_k=\int_{\Omega_\eps^\N}\int_{[0,1]\setminus \tB(\delta_0)} \sL_{\sG_k}^{\nu_\eps}(y) (\phi_k(y,\bfh))^p dP_\eps\le C_3 |\tB(c;\eps)|L_k.$$
Combining with (\ref{eqn:sumLk}), we obtain
\begin{equation}\label{eqn:sumMk}
\sum_{k=1}^\infty M_k\le C_2C_3 |\tB(c;\eps)|.
\end{equation}
On each $X_k$ we have
$H(x,\bfg) \le 2 h(x,\bfg)+\phi_k(x,\bfg)$, so
$$\int_{X_k} (H(x,\bfg))^p dP_\eps\le  C_4  \int_{X_k} (h(x,\bfg))^p dP_\eps +C_4 M_k, $$
where $C_4>0$ is a constant. Thus
\begin{multline*}
|\tB(c;\eps)|S_p^\theta(\eps,\eps;c,\delta_0)=  \sum_{k=1}^\infty \int_{X_k} (H(x,\bfg))^p dP_\eps
\le C_4 \sum_{k=1}^\infty \int_{X_k} (h(x,\bfg))^p dP_\eps+C_4\sum_{k=1}^\infty M_k
\\
= C_4 \int_{\tB(c;\eps)\times \Omega_\eps^\N} h(x,\bfg)^p dP_\eps +C_4 \sum_{k=1}^\infty M_k.
\end{multline*}
By (\ref{eqn:hhp0}) and (\ref{eqn:sumMk}), the desired estimate follows.

(ii) Similarly as in (i), using Proposition~\ref{prop:indsmall(1)} instead of Proposition~\ref{prop:indsmall}, we prove that for $\delta_0$ small enough, there exists $\eps_0>0$ such that if $0< \delta\le\delta_0/e$ and $0<\eps\le \min (\delta, \eps_0)$, then we have
$$\frac{1}{|\tB(c;\delta)|}\int\int_{(\tB(c;e\delta)\setminus \tB(\delta))\times \Omega_\eps^\N} \left(\inf_{\delta'\in [e\delta, \delta_0]} h_{\delta'}^\theta\right)^p dP_\eps\le C'\delta^{-\gamma},$$
where $C'>0$ is a constant, which implies
$$\int\int_{(\tB(c;e\delta)\setminus \tB(\delta))\times \Omega_\eps^\N}\frac{1}{\dist (x,\Crit)} \left(\inf_{\delta'\in [e\delta, \delta_0]} h_{\delta'}^\theta\right)^p dP_\eps \le C''\delta^{-\gamma}.$$
 The desired estimate follows.
\end{proof}
\newcommand{\hs}{\widehat{s}}

\subsection{Proof of Proposition~\ref{prop:indind}}\label{subsec:indind}
Fix $p\ge 1$, $\gamma>0$ and $\lambda\in (e^{-\ell_{\max}^{-1}}, 1)$. Let $\theta>0$ be small such that
$$\left(1-(36\theta \#\Crit)^{1/p}\right)^p \lambda e^{\ell_{\max}^{-1}}>1.$$
Let $\delta_0>0$ be a small constant and consider $0<\eps\le\delta\le \delta_0/e$.
Let
$$s(x,\bfg)=\inf_{\delta'\in [\delta, \delta_0]} h_{\delta'}^\theta (x,\bfg),\mbox{ and } \hs(x,\bfg)=\inf_{\delta'\in [e\delta,\delta_0]} h_{\delta'}^{\theta/e}(x,\bfg),$$
and let
\begin{equation}\label{eqn:sxg}
\varphi(x,\bfg)=\left\{
\begin{array}{ll}
s(x,\bfg) & \mbox{ if } x\in \tB(\delta);\\
&\\
\hs(x,\bfg) &\mbox{ otherwise.}
\end{array}
\right.
\end{equation}
Let $\hE_0=\tB(\delta_0)\times \Omega_\eps^\N\supset E_0=\tB(e\delta)\times \Omega_\eps^\N$, let
$$E_1=\{(x,\bfg)\in \tB(\delta)\times \Omega_\eps^\N: s(x,\bfg)< \hs(x,\bfg)\}$$ and let $\sG: E_1\to \hE_0$ denote the map
$(x,\bfg)\mapsto F^{s(x,\bfg)}(x,\bfg)$.
For each $n=1,2,\ldots$, let $E_n=\dom (\sG^n)$ and
$\varphi_n=1_{E_n}\varphi\circ \sG^n$. Furthermore, for each $c\in\Crit$, and $n=0,1,\ldots$, let
$E_n(c)=E_n\cap (\tB(c;\delta_0)\times \Omega_\eps^\N)$, $\sG_c=\sG|E_1(c)$, and
$$K_n(c) =\left(\int\int_{E_n(c)} \varphi(\sG^n(x,\bfg))^p dP_\eps\right)^{1/p}.$$

\begin{lemm}\label{lem:join2}
Assume $\delta_0$ small. Then
\begin{equation}\label{eqn:esticp2del}
\left(|\tB(c;e\delta)|S_p^\theta(e\delta,\eps; c,\delta_0)\right)^{1/p}\le \sum_{n=0}^\infty K_n(c).
\end{equation}
\end{lemm}

\begin{proof} It suffices to prove that for each $(x,\bfg)\in E_0$, we have
\begin{equation}\label{eqn:joinest}
\inf_{\delta'\in [e\delta, \delta_0]} h_{\delta'}^\theta(x,\bfg)\le \sum_{n=0}^\infty \varphi_n(x,\bfg),
\end{equation}
provided that $\delta_0$ is small enough.

If $(x,\bfg)\in \bigcap_{n=0}^\infty E_n$, then the right hand side is infinity, so the inequality holds.  If $(x,\bfg)\in E_0\setminus E_1$, then the inequality holds by definition.
So assume that there exists an integer $n\ge 1$ such that $(x,\bfg)\in E_n\setminus E_{n+1}$.
In this case, the inequality follows from Lemma~\ref{lem:join1}. Indeed, letting $T_0=0$ and $T_i= \sum_{j=0}^{i-1} \varphi_j(x,\bfg)$ for each $1\le i\le n+1$, then for each $0\le i< n$, $T_{i+1}-T_i$ is a $1/2$-close return of $F^{T_i}(x,\bfg)$, so by part (i) of that lemma, $T_n$ is a $1$-close return of $(x,\bfg)$. Since $T_{n+1}$ is a $\theta/2$-good return time of $F^{T_n}(x,\bfg)$ into $\tB(\delta')\times \Omega_\eps^\N$ for some $\delta'\in [e\delta, \delta_0]$, it follows by part (ii) of that lemma that $T_{n+1}$ is a $\theta$-good return time of $(x,\bfg)$ into $\tB(\delta')\times \Omega_\eps^\N$.
\end{proof}

To complete the proof of Proposition~\ref{prop:indind}, we shall estimate $K_n(c)$.

\begin{lemm}\label{lem:cG}
Assume $\delta_0$ small. For any $c,c'\in\Crit$ and  $y\in\tB(c';\delta_0)$, we have
$$\cL^{\nu_\eps}_{\sG_c}(y)\le 36\theta
\frac{|\tB(c;\delta)|}{|\tB(c';\delta')|},
$$
where $\delta'=\max(\delta, \dist_*(y,\Crit))$.
\end{lemm}

\begin{proof}
Fix $c, c'\in\Crit$, $y\in\tB(c';\delta_0)$, $\bfg\in\Omega_\eps^\N$ and let $\delta'=\max(\delta, \dist_*(y,\Crit))$. We shall prove that
$\cL_{\sG_c}^\bfg (y)\le 36\theta |\tB(c;\delta)|/|\tB(c';\delta')|$ which clearly implies the lemma.
Let $\cX$ denote the set of all $x\in \tB(c;\delta)$ for which $(x,\bfg)\in E_1(c)$ and $\bfg^{s(x,\bfg)}(x)=y$. For each $x\in \cX$,
let $\hJ_{x}:=\hJ_{x,s(x,\bfg)}^\bfg$ be as defined in (\ref{eqn:defnJ}). Then $\hJ_x\subset (0,1)$, $\bfg^{s(x,\bfg)}$ maps $\hJ_{x}$ diffeomorphically onto its image with $\cN(\bfg^{s(x,\bfg)}|\hJ_{x})\le 1$.
Let $I\supset I_0$ be the closed intervals centered at $c'$ and such that
$|I|=|\tB(c';\delta')|/(e\theta)$,   $|I_0|=4|\tB(c';\delta')|$. Since $\bfg^{s(x,\bfg)}(x)=y$, $s(x,\bfg)$ is  a $\theta$-good of $(x,\bfg)$ into $\tB(\delta_x)\times \Omega_\eps^\N$ for some $\delta_x\in [\delta',\delta_0]$. Thus $\bfg^{s(x,\bfg)}(\hJ_{x})\supset I$. Let $J_{x}\supset \tJ_{x}$ be subintervals of $\hJ_{x}$  such that
$\bfg^{s(x,\bfg)}(J_{x})=I$ and $\bfg^{s(x,\bfg)}(\tJ_{x})=I_0$.
Then
\begin{equation}\label{eqn:hJvsJ}
|\tJ_{x}|\le e \frac{|I_0|}{|I|}|J_{x}|\le 4e^2\theta |J_{x}|,
\end{equation}
and both components of $J_x\setminus \tJ_x$ have length bigger than $|\tJ_x|$.
Since
\begin{equation}\label{eqn:LGc0}
\cL_{\sG_c}^\bfg(y)\le e \sum_{x\in \cX} |J_{x}|/|I|\le e^2\theta \sum_{x\in \cX} |J_{x}|/|\tB(c';\delta')|,
\end{equation}
it suffices to prove
\begin{equation}\label{eqn:sumlengthJ}
\sum_{x\in\cX} |J_x|\le 4|\tB(c; \delta)|.
\end{equation}

To this end, we shall first prove the following
claim: {\em  For each $x'\in J_{x}\cap \cX$ with $s(x',\bfg)> s(x,\bfg)$, we have
$J_{x}\supset \tJ_{x}\supset J_{x'}$.}

To prove this claim, let $s=s(x,\bfg)$, $s'=s(x',\bfg)$ and let $(z,\bfh)= F^s(x',\bfg)$.
We first prove that $\dist_*(z,\Crit) \le \delta'$. Arguing by contradiction, assume $\dist_*(z,\Crit)>\delta'$. Since $\bfh^{s'-s}(x')=\bfg^{s'}(x')=y$ there exists a minimal positive integer $t\le s'-s$ such that $\dist_*(\bfh^t(z),\Crit)\le \delta'$.  Let  $\delta''\in (\delta',\delta_0]$ be such that $\dist_*( \bfh^j(z),\Crit)\ge \delta''$ for all $0\le j<t$. Then by Lemma~\ref{lem:thetadelta}, 
$t$ is a $\theta/(2e^2)$-good return time of $(z,\bfh)$ into $\tB(\delta'')\times \Omega_\eps^\N$,
provided that $\delta_0$ is small enough.
By Lemma~\ref{lem:theta0},
$$\frac{|D\bfg^{s}(x')|}{A(x', \bfg, s)}\ge e^{-1} \frac{|D\bfg^s(x)|}{A(x,\bfg,s)}, $$
which implies that
$|D\bfg^{s}(x')|\ge A(x', \bfg, s)\dist(\bfg^s(x'),\Crit(g_s)),$
i.e.,
$s$ is a $1$-close return of $(x',\bfg)$.
By Lemma~\ref{lem:join1}, it follows that $s+t$ is a $\theta/e^2$-good return time of $(x,\bfg)$ into $\tB(\delta'')\times \Omega_\eps^\N$. Since $\delta''\ge \delta$, this implies that $\hs(x',\bfg) \le s+ t\le s'.$ Since
$\hs(x',\bfg)\ge s(x',\bfg)=s'$, it follows that $\hs(x',\bfg)=s(x',\bfg)$, which contradicts the assumption that $(x',\bfg)\in E_1$.
This proves $\dist_*(z,\Crit) \le \delta'$.
Since $\cN(g^s|J_{x'})\le 1$, we have
$$\frac{|\bfg^s(J_{x'})|}{\dist(\bfg^{s}(x'),\Crit)}\le \frac{e |D\bfg^s(x')||J_{x'}|}{\dist(\bfg^{s}(x'),\Crit)}
\le
\frac{2e\theta_0}{A(x',\bfg,s')} \frac{|D\bfg^s(x')|}{\dist(\bfg^s(x'),\Crit)} \le 2e\theta_0,$$
it follows that $\bfg^s(J_{x'})\subset I_0$. The claim follows.

Let us now complete the proof of (\ref{eqn:sumlengthJ}).
Indeed, we can decompose $\cX$ as a disjoint union of subcollections $\cX(k)$, $k=0,1,\ldots$, as follows: $\cX(0)$ is the subset of $\cX$ consisting of those $x$'s for which $s(x,\bfg)\le s(x',\bfg)$ for each $x'\in J_x\cap \cX$, and for each $k=1,2,\ldots$, $\cX(k)$ is the subset of $\cX\setminus (\bigcup_{i=0}^{k-1}\cX(i))$ consisting of those $x$'s for which $s(x,\bfg)\le s(x',\bfg)$ for each $x'\in J_x\cap (\cX\setminus(\bigcup_{i=0}^{k-1}\cX(i)))$.
Then the claim and (\ref{eqn:hJvsJ}) imply that for each $k=1,2,\ldots$,
$$\sum_{x'\in\cX(k)}|J_{x'}|\le \sum_{x\in\cX(k-1)} |\tJ_x|\le \frac{1}{2} \sum_{x\in\cX(k-1)} |J_x|.$$
Since each $J_x$, $x\in\cX$, has length less than $|\tB(c; \delta)|$ the inequality (\ref{eqn:sumlengthJ}) follows.
\end{proof}

\begin{proof}[Proof of Proposition~\ref{prop:indind}]
Let  $S=\max_{c\in\Crit} S_p^\theta(\delta,\eps; c,\delta_0)$, let
$\hS= \hS_p^{\theta/e}(\delta,\eps; \delta_0)$ and let $\tK_n=\max_{c\in\Crit} (K_n(c))^p/ |\tB(c;\delta)|$, for each $n=0,1,\ldots$ and $c\in\Crit$.
In the following we shall prove that
\begin{equation}\label{eqn:tMn}
\tK_n \le (36\theta \#\Crit)^n (S +2\hS).
\end{equation}
First of all, by definition, for each $c\in\Crit$,
\begin{equation}\label{eqn:M0}
(K_0(c))^p\le |\tB(c;\delta)|S_p^\theta(\delta,\eps; c,\delta_0)+ |\tB(c;2\delta)|\hS.
\end{equation}
Thus (\ref{eqn:tMn}) holds for $n=0$, provided that $\delta_0$ is small enough.

Note that $\sG_c$ is future-free. Applying Lemmas~\ref{lem:pushforward}  and~\ref{lem:cG} to $\sG_c$ and $\phi=\varphi^p$, we obtain
\begin{align*}
(K_1(c))^p & = \int_{\tB(\delta_0)\times \Omega_\eps^\N} \cL_{\sG_c}^{\nu_\eps}(y) (\varphi(y,\bfg))^p  dy d\nu_\eps^\N\\
& =\sum_{c'\in\Crit} \int_{\tB(c';\delta)\times \Omega_\eps^\N} \cL_{\sG_c}^{\nu_\eps}(y) (\varphi(y,\bfg))^p  dP_\eps
+ \int_{(\tB(\delta_0)\setminus \tB(\delta))\times \Omega_\eps^\N} \cL_{\sG_c}^{\nu_\eps}(y)(\varphi(y,\bfg))^p  dP_\eps\\
& \le
36 \theta|\tB(c;\delta)|\left(\sum_{c'\in\Crit} S_p^\theta(\delta,\eps; c',\delta_0)+
\hS_p^{\theta/e} (\delta,\eps;\delta_0)\right)\le 36\theta \# \Crit (S+\hS).
\end{align*}
So (\ref{eqn:tMn}) holds for $n=1$.
Similarly, for each $n\ge 1$, applying Lemmas~\ref{lem:pushforward}  and~\ref{lem:cG} to $\sG_c$ and $\phi=\varphi_n^p$, we obtain
$$(K_{n+1}(c))^p = \int_{E_n} \cL_{\sG_c}^{\nu_\eps}(y)(\varphi_n (y,\bfg))^p dP_\eps\le 36\theta |\tB(c,\delta)|\sum_{c'\in\Crit} \frac{(K_n(c'))^p}{|\tB(c';\delta)|},$$
which implies that
$\tK_{n+1}\le 36\theta \#\Crit \tK_n.$
By induction, (\ref{eqn:tMn}) holds for all $n$.

By (\ref{eqn:esticp2del}), (\ref{eqn:tMn}) and the choice of $\theta$, the proposition follows.
\end{proof}

\bibliographystyle{alpha}

\end{document}